\documentclass[final]{elsarticle}
\usepackage{amssymb}
\usepackage{algorithm} 
\usepackage{algpseudocode} 
\usepackage{subfigure}
\usepackage{multicol, multirow}
\usepackage{ltablex}
\usepackage{float}
\usepackage{xcolor}

\usepackage{mathtools, amsthm, graphics, array}

\newtheorem{theorem}{Theorem}

\newtheorem{lemma}{Lemma}
\newtheorem{remark}{Remark}

\usepackage{amsmath, array, bm}

\usepackage{tikz}
\usetikzlibrary{shapes.geometric, calc, backgrounds}

\journal{Elsevier}

\usepackage[colorlinks, linkcolor=blue, anchorcolor = blue, citecolor = blue]{hyperref}

\begin{document}

\begin{frontmatter}




\title{A non-nested unstructured mesh perspective on highly parallel multilevel smoothed Schwarz preconditioner for linear parametric PDEs}

\author[inst1]{Chengdi Ma}
\ead{mcd2020@stu.pku.edu.cn}

\affiliation[inst1]{organization={School of Mathematical Sciences},
            addressline={Peking University}, 
            city={Beijing},
            postcode={100871}, 
            country={P.R. China}}





\begin{abstract}
The multilevel Schwarz preconditioner is one of the most popular parallel preconditioners for enhancing convergence and improving parallel efficiency. However, its parallel implementation on arbitrary unstructured triangular/tetrahedral meshes remains challenging. The challenges mainly arise from the inability to ensure that mesh hierarchies are nested, which complicates parallelization efforts. This paper systematically investigates the non-nested unstructured case of parallel multilevel algorithms and develops a highly parallel non-nested multilevel smoothed Schwarz preconditioner. The proposed multilevel preconditioner incorporates two key techniques. The first is a new parallel coarsening algorithm that preserves the geometric features of the computational domain. The second is a corresponding parallel non-nested interpolation method designed for non-nested mesh hierarchies. This new preconditioner is applied to a broad range of linear parametric problems, benefiting from the reusability of the same coarse mesh hierarchy for problems with different parameters. Several numerical experiments validate the outstanding convergence and parallel efficiency of the proposed preconditioner, demonstrating effective scalability up to 1,000 processors. 
\end{abstract}


\end{frontmatter}

\section{Introduction}
The multilevel Schwarz methods are widely applied in large-scale scientific and engineering computing \citep{antonietti2014domain, brix2008multilevel, antonietti2007schwarz, liu2024one}, as they are  \textit{provably} scalable in the sense that the number of iterations does not grow much when the number of processor cores is increased. Notably, the condition number of the multilevel Schwarz preconditioned linear system is \textit{theoretically} bounded by $O(H/\delta)$ and $O(H/h)$ \citep{li2024preconditioned, smith1997domain, toselli2004domain}, where $H$, $h$ and $\delta$ represent the granularity of the coarse and fine meshes and the size of the overlap, respectively. To handle complex geometries, multilevel Schwarz methods are often used with unstructured meshes \citep{chan1999boundary, kong2020highly}. However, designing a multilevel Schwarz preconditioner on arbitrary unstructured triangular/tetrahedral meshes remains challenging. The challenges mainly arise from the inability to ensure that mesh hierarchies are nested, which complicates parallelization efforts \citep{antonietti2019v, kong2017scalable, kong2016highly}. 

To systematically investigate the parallelization of non-nested multilevel algorithms, a non-nested unstructured mesh perspective is introduced \citep{lange2016efficient, ibanez2016pumi}. In our discussion, the mesh and degrees of freedom (DoFs) are abstracted from the physical problem and encapsulated as reusable objects \citep{batty2010tetrahedral, alauzet20073d}. The non-nested mesh hierarchies need to be constructed based on the finest  unstructured mesh object. Then, non-nested mesh hierarchies and their associated DoFs are partitioned into subdomains and distributed across multiple processors for parallelization. The non-nested hierarchies typically lead to mismatched subdomain boundaries, which present major challenges in the design of parallel algorithms. Then, the abstraction of meshes and DoFs categorizes these challenges into two aspects. On the one hand, there is a need for robust and fast parallel coarsening algorithms to construct distributed coarse mesh hierarchies \citep{ollivier2003coarsening, de1999parallel, liu2017new}. On the other hand, mismatched subdomain boundaries necessitate the redesign of mesh-to-mesh interpolation for DoFs \citep{farrell2011conservative, farrell2009conservative, sampath2010parallel}.

There are several algorithms for constructing coarse mesh hierarchies from an existing fine unstructured mesh. Due to the significant computational cost, repeatedly invoking a mesh generator to create a multigrid hierarchy is generally impractical \citep{liu2017new}. The most popular approach is the topologically based coarsening paradigm, which includes Delaunay coarsening method \citep{guillard1993node}, edge contraction method \citep{ollivier2003coarsening, de1999parallel}, and small polyhedron reconnection (SPR) method \citep{liu2017new}. The Delaunay coarsening method ensures the quality of the coarse mesh but fails to preserve geometric features. In contrast, the edge contraction method maintains geometric features but struggles with narrow regions in a tetrahedral mesh. 
Compared to the edge contraction method, the SPR method enhances robustness in narrow regions but incurs significant computational overhead, thereby necessitating further exploration of parallel acceleration.
In practice, we seek a robust and fast parallel coarsening algorithm that preserves boundary geometry to reduce the overhead associated with special treatments on mismatched boundaries between fine and coarse meshes \citep{chan1999boundary}.
Then, we need to focus on interpolating both scalar and vector fields between different meshes within the computational domain in a fast and accurate way. 
Several interpolation methods for unstructured meshes are proposed and widely applied. 
One method involves interpolation via supermesh construction and projection \citep{farrell2011conservative, farrell2009conservative, di2021towards}, which effectively preserves conservation properties. Another method uses mesh-free approaches \citep{salvador2020intergrid, forti2015semi}, offering an operator with excellent locality. However, the current interpolation methods are not sufficiently explored in large-scale parallel computing scenarios.



In this paper, we propose a novel multilevel smoothed Schwarz preconditioner from a non-nested unstructured mesh perspective. The proposed preconditioner primarily incorporates two new techniques, each aimed at addressing challenges in parallel coarsening and parallel interpolation within non-nested scenarios. Leveraging the high parallelism and boundary-preserving characteristics of the consistent parallel advancing front technique (CPAFT) \citep{ma2024cpaft}, we design a new parallel geometry-preserving coarsening algorithm. The algorithm integrates the geometry-preserving characteristics of the edge contraction method with the robustness of the SPR method. With the proposed parallel coarsening algorithm, the non-nested coarse mesh hierarchies are constructed in parallel and stored on each processor. To address the mismatched subdomain boundaries, we introduce overlap into the interpolation algorithm and design a consistent and flexible interpolation operator based on the moving least squares (MLS) method \citep{hegen1996element, belytschko1994element, levin1998approximation, nogueira2010new}. This interpolation method mathematically provides a parallel interpolation framework for multiphysics problems on non-nested meshes, featuring excellent local dependency properties. Compared to the existing methods, the newly proposed interpolation method provides a flexible paradigm that circumvents the heavy overhead associated with intersection computations. By integrating the new parallel coarsening algorithm and the new parallel interpolation method into the V-cycle framework \citep{antonietti2019v, kong2017scalable}, we propose the novel non-nested multilevel preconditioner. The non-nested multilevel preconditioner is applied to a broad range of linear parametric partial differential equations (PDEs), which benefits from reusing the same coarse mesh hierarchies for problems with different parameters. Numerical experiments validate the outstanding convergence and parallel efficiency of the non-nested multilevel algorithm, demonstrating effective scalability up to 1,000 processors.

The rest of this paper is organized as follows. In Section~\ref{sec: Preliminaries}, the unstructured mesh perspective and the multilevel smoothed Schwarz framework are introduced. Section~\ref{sec: algorithm} presents the parallel geometry-preserving mesh coarsening algorithm, the multiphysics-oriented parallel non-nested interpolation methods, and the workflow of the proposed parallel non-nested multilevel smoothed Schwarz preconditioner. Results of a series of numerical experiments are presented in Section~\ref{sec: experiments}. Finally, conclusions are drawn in Section~\ref{sec: conclusions}. 

\section{Preliminaries}
\label{sec: Preliminaries}
\subsection{Linear parametric PDEs and discretization from an unstructured mesh perspective}
\label{subsec: discretization}
In this paper, we consider the second-order linear parametric PDEs in a complex computational domain $\Omega \subset \mathbb{R}^{d}\ (d=2,3)$ with boundary $\partial \Omega \subset \mathbb{R}^{d-1}$. The governing equations can be represented using linear operators $\bm \Phi$ and ${\bm \Phi}_{b}$ as follows, 
\begin{equation}
    \left \{
    \begin{aligned}
    {\bm \Phi}(\bm u, \mathcal{D} \bm u, \mathcal{D}^{2} \bm u, \bm x; \bm \mu) &= \bm 0,\ \bm x\ \texttt{in}\ \Omega, \\ 
    {\bm \Phi}_{b} (\bm u, \mathcal{D} \bm u, \bm x; \bm \mu) &= \bm 0,\  \bm x\ \texttt{on}\ \partial \Omega,
    \label{eq: linear-PDE}
    \end{aligned}
    \right.
\end{equation}
where $\bm u \in \mathbb{R}^{F}$ denotes a vector consisting of $F$ physical variables, and $\bm \mu$ represents the parameters of the problem.

We first assume $\bar{\Omega}_{h} = \bigcup_{k \in \mathcal{K}} \bar{k}$ represents a conforming triangulation/tetrahedralization of the computational domain $\Omega$, and $\mathcal{K}$ represents the set of elements. Let $\mathcal{N} = (z_{0},\dots,z_{d})\in \mathbb{Z}_{\geq 0}^{d+1}$ be the distribution of DoFs on each triangular/tetrahedral element, where $z_{0}, \dots, z_{d}$ denote the number of DoFs on geometric nodes associated with vertices, edges, faces, and volumes, respectively, as illustrated in Figure~\ref{multiphysics}. 
The set of elements $\mathcal{K}$ and distribution of DoFs $\mathcal{N}$ determines the discretization of physical problems.
In our subsequent discussion, we focus on the parameter-independent tuple $\mathcal{M} := (\mathcal{K},\ \mathcal{N})$, which serves as an unstructured mesh perspective in this work.
\begin{figure}[H]
    \centering
    \subfigure[$\mathcal{N}=(0,0,1)$]{
    \begin{minipage}[]{0.2\linewidth}
    \centering
    \begin{tikzpicture} [scale=1.0, every node/.style={scale=1}]
        \coordinate (A) at (0,0);
        \coordinate (B) at (2,0);
        \coordinate (C) at (1, 1.7320508075688772);
        \draw (A) -- (B) -- (C) -- cycle;

        \fill [black] (1, 0.5773502691896257) circle (2pt);
    \end{tikzpicture}
    \end{minipage}
    } \subfigure[$\mathcal{N}=(3, 3, 0, 0)$]{
    \begin{minipage}[]{0.2\linewidth}
    \centering
    \begin{tikzpicture} [scale=1.0, every node/.style={scale=1}]
        \coordinate (A) at (0,0);
        \coordinate (B) at (2,0);
        \coordinate (C) at (1, 1.7320508075688772);
        \coordinate (D) at (1, 0.5773502691896257);
        \draw (A) -- (B) -- (C) -- cycle;
        \draw[dashed] (A) -- (D); \draw[dashed] (B) -- (D); \draw[dashed] (C) -- (D);

        \filldraw [fill=gray] (A) circle (2pt); \filldraw [fill=gray] (B) circle (2pt);
        \filldraw [fill=gray] (C) circle (2pt); \filldraw [fill=gray] (D) circle (2pt);
        
        \filldraw [fill=gray] (1, 0) circle (2pt); \filldraw [fill=gray] (0.5, 0.8660254037844386) circle (2pt);
        \filldraw [fill=gray] (1.5, 0.8660254037844386) circle (2pt);
        \filldraw [fill=gray] (0.5, 0.28867513459481287) circle (2pt);
        \filldraw [fill=gray] (1.5, 0.28867513459481287) circle (2pt);
        \filldraw [fill=gray] (1, 1.1547005383792515) circle (2pt);
        
    \end{tikzpicture}
    \end{minipage}
    } \subfigure[$\mathcal{N}=(3,2,0)$]{
    \begin{minipage}[]{0.2\linewidth}
    \centering
    \begin{tikzpicture} [scale=1.0, every node/.style={scale=1}]
        \coordinate (A) at (0,0);
        \coordinate (B) at (2,0);
        \coordinate (C) at (1, 1.7320508075688772);
        \draw (A) -- (B) -- (C) -- cycle;

        \filldraw [fill=gray] (A) circle (2pt); 
        \filldraw [fill=gray] (B) circle (2pt);
        \filldraw [fill=gray] (C) circle (2pt);
        \filldraw [fill=white] (1, 0) circle (2pt);
        \filldraw [fill=white] (0.5, 0.8660254037844386) circle (2pt);
        \filldraw [fill=white] (1.5, 0.8660254037844386) circle (2pt);
    \end{tikzpicture}
    \end{minipage}
    } \subfigure[$\mathcal{N}=(4, 3, 0, 0)$]{
    \begin{minipage}[]{0.2\linewidth}
    \centering
    \begin{tikzpicture} [scale=1.0, every node/.style={scale=1}]
        \coordinate (A) at (0,0);
        \coordinate (B) at (2,0);
        \coordinate (C) at (1, 1.7320508075688772);
        \coordinate (D) at (1, 0.5773502691896257);
        \draw (A) -- (B) -- (C) -- cycle;
        \draw[dashed] (A) -- (D); \draw[dashed] (B) -- (D); \draw[dashed] (C) -- (D);

        \filldraw [fill=blue] (A) circle (2pt); \filldraw [fill=blue] (B) circle (2pt);
        \filldraw [fill=blue] (C) circle (2pt); \filldraw [fill=blue] (D) circle (2pt);

        \filldraw [fill=gray] (1, 0) circle (2pt); 
        \filldraw [fill=gray] (0.5, 0.8660254037844386) circle (2pt);
        \filldraw [fill=gray] (1.5, 0.8660254037844386) circle (2pt);
        \filldraw [fill=gray] (0.5, 0.28867513459481287) circle (2pt);
        \filldraw [fill=gray] (1.5, 0.28867513459481287) circle (2pt);
        \filldraw [fill=gray] (1, 1.1547005383792515) circle (2pt);
    \end{tikzpicture}
    \end{minipage}
    } \subfigure[legend]{
        \begin{minipage}[]{0.1\linewidth}
        \centering
        \begin{tikzpicture} [scale=1.0, every node/.style={scale=1}]
            \filldraw [fill = black] (0, 1.35) circle (2.0pt);
            \filldraw [fill = white] (0, 0.9) circle (2.0pt);
            \filldraw [fill = gray]  (0, 0.45) circle (2.0pt);
            \filldraw [fill = blue]  (0, 0.0) circle (2.0pt);

            \node at (0.15, 1.35) [right] {$1$};
            \node at (0.15, 0.9) [right] {$2$};
            \node at (0.15, 0.45) [right] {$3$};
            \node at (0.15, 0.0) [right] {$4$};
        \end{tikzpicture}
        \end{minipage}
    }
    \caption{4 examples of $\mathcal{N}$: (a) Finite volume method for 2D convection-diffusion equation; (b) Finite element method $[\mathcal{P}_2]^3$ for 3D linear elasticity equations; (c) Mixed finite element method $[\mathcal{P}_2]^2$-$\mathcal{P}_1$ for 2D Stokes flow; (d) Mixed finite element method
    $[\mathcal{P}_2]^3$-$\mathcal{P}_1$ for 3D Stokes flow;}
    \label{multiphysics}
\end{figure}
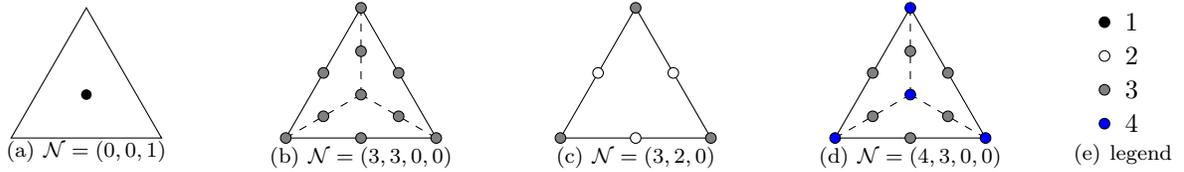

The unstructured mesh perspective $\mathcal{M}$ abstracts the mesh and DoFs from the parametric problem and encapsulates them into a reusable object, facilitating mesh-to-mesh operations. 
For calculation, let us establish an ordering for the geometric nodes in the unstructured mesh $\mathcal{M}$, where the dimension of the $n_{th}$ geometric node is $s(n)$, satisfying $0 \leq s(n) \leq d$. We then consider geometric nodes $X_{\mathcal{M}} := \{\bm x_{n}\in \mathcal{M}\ |\ z_{s(n)} \geq 1\}$ and corresponding unknowns $U_{\mathcal{M}} := \{ u\ \texttt{at}\ (\bm x_{n}, z)\ |\ \bm x_{n} \in X_{\mathcal{M}},\  z = 0, \dots, z_{s(n)}-1\}$. The unknown $u$ at $(\bm x_{n}, z)$ represents the $z_{th}$ DoF located at the geometric node $\bm x_{n}$.
Then, the PDE~\eqref{eq: linear-PDE} will be discretized on the mesh $\mathcal{M}$ using an $\mathcal{N}$-matched numerical scheme into $|U_{\mathcal{M}}|$ equations, described as ${\bm \Psi}_{\mathcal{M}} (\bm u_{\mathcal{M}}) = \bm 0$, in which ${\bm \Psi}_{\mathcal{M}}: \mathbb{R}^{|U_{\mathcal{M}}|} \to \mathbb{R}^{|U_{\mathcal{M}}|}$ and solution $\bm u_{\mathcal{M}} \in \mathbb{R}^{|U_{\mathcal{M}}|}$. Furthermore, the formulation as a system of linear algebraic equations is as follows:
\begin{equation}
    \mathcal{J}_{\mathcal{M}}\ {\bm u}_{\mathcal{M}} = {\bm b}_{\mathcal{M}},
    \label{equ: discretized}
\end{equation}
where right-hand side vector ${\bm b}_{\mathcal{M}}$ and Jacobian matrix $\mathcal{J}_{\mathcal{M}}$ of the discretization system are given by 
\begin{align}
    &\bm b_{\mathcal{M}}(i) = -{\bm \Psi}_{\mathcal{M}}(\bm 0; i),\ \ \ \ \ \ \ \ \ \ \ \ \ \ \ \ \  i = 0, \dots, |U_{\mathcal{M}}|-1, \\
    &\mathcal{J}_{\mathcal{M}}(i, j) = {\bm \Psi}_{\mathcal{M}}(\bm u_{\mathcal{M}}; i)/ \bm u_{\mathcal{M}}(j),\ \ \  j = 0, \dots, |U_{\mathcal{M}}|-1.
\end{align}
Given a preconditioner $B_{\mathcal{M}}^{-1}$, we then consider the right-preconditioned problem arsing from Equation~\eqref{equ: discretized}:
\begin{equation}
    \mathcal{J}_{\mathcal{M}} B_{\mathcal{M}}^{-1} \bm{z}_{\mathcal{M}} = \bm{b}_{\mathcal{M}},\quad \quad B_{\mathcal{M}} {\bm u}_{\mathcal{M}} = \bm{z}_{\mathcal{M}}.
    \label{equ: right-preconditioned}
\end{equation}


\subsection{Additive Schwarz preconditioner}
\label{subsec: RAS}
The set of elements $\mathcal{K}$ can be uniformly partitioned into $np$ subsets as $\mathcal{K} = \bigcup_{p = 1}^{np} \mathcal{K}_{p}$. The overlapping restricted additive Schwarz (RAS) method \citep{cai1999restricted} begins with a partition of the computational domain into $np$ nonoverlapping subdomains $\bar{\Omega}_{\mathcal{M}, p} = \bigcup_{k \in \mathcal{K}_{p}} \bar{k},\ p=1, \dots, np$, and then extend them into overlapping subdomains $\Omega_{\mathcal{M}, p}^{\delta},\ p=1, \dots, np, $ by including $\delta$ layers of elements from the neighboring subdomain. We then employ the restriction approach to extract $\mathcal{J}^{\delta}_{\mathcal{M}, p}$, for $p = 1, \dots, np$ on each subdomain from the global matrix $\mathcal{J}_{\mathcal{M}}$. The one-level RAS preconditioner on unstructured mesh $\mathcal{M}$ is defined as
\begin{align}
        &\mathcal{J}^{\delta}_{\mathcal{M}, p} = {\bm R}^{\delta}_{\mathcal{M}, p} \mathcal{J}_{\mathcal{M}} ({\bm R}^{\delta}_{\mathcal{M}, p})^{\top}, \\
        &B_{\mathcal{M}}^{-1}\  = \sum_{p=1}^{np} ({\bm R}^{0}_{\mathcal{M}, p})^{\top} {(\mathcal{\tilde{J}}^{\delta}_{\mathcal{M}, p})}^{-1} {\bm R}^{\delta}_{\mathcal{M}, p},
\end{align}
where ${(\mathcal{\tilde{J}}^{\delta}_{\mathcal{M}, p})}^{-1}$ is an (approximate) subdomain solver constructed from $\mathcal{J}^{\delta}_{\mathcal{M}, p}$. The operator ${\bm R}_{\mathcal{M}, p}^{\delta}$ denotes a restriction matrix which extracts a local overlapping vector ${\bm u}_{\mathcal{M}, p}^{\delta}$ from the global vector ${\bm u}_{\mathcal{M}}$, that is,
\begin{equation}
    {\bm u}^{\delta}_{\mathcal{M}, p} = {\bm R}^{\delta}_{\mathcal{M}, p} {\bm u}_{\mathcal{M}} = \begin{pmatrix} {\bm I} & {\bm O} \end{pmatrix} \begin{pmatrix} {\bm u}_{\mathcal{M}, p}^{\delta} \\ {\bm u}_{\mathcal{M}} / {\bm u}_{\mathcal{M}, p}^{\delta} \end{pmatrix}.
\end{equation}

\subsection{Multilevel smoothed Schwarz preconditioner}
\label{subsec: v-cycle}
The class of multilevel Schwarz algorithms is widely used for accelerating convergence and enhancing scalability in massively parallel scenarios. As shown in Figure~\ref{fig: v-cycle}, we consider the V-cycle multilevel smoothed Schwarz framework for solving the discretized system~\eqref{equ: discretized}, where the RAS preconditioner serves as the smoother on each level. Initially, a mesh hierarchy from fine to coarse $\mathcal{M}_{i},\ i = 0, \dots, L-1$ is constructed. Then, pre-smoothing, post-smoothing, and error correction are conducted on each level of the mesh hierarchy. Between different mesh levels, interpolation operators $\mathcal{I}_{\mathcal{M}_{i}}^{\mathcal{M}_{i+1}}$ and $\mathcal{I}_{\mathcal{M}_{i+1}}^{\mathcal{M}_{i}}$ are used as the fine-to-coarse transfer and coarse-to-fine transfer, respectively.

\begin{figure}[H]
    \centering
    \begin{tikzpicture}[scale=1.25, every node/.style={scale=1}]
    \draw[thick] (0.5,4) -- (0.5,3) -- (0.5,2) -- (0.5,1);
    \filldraw[fill=white] (0.5,4) circle (2pt);
    \filldraw[fill=white] (0.5,3) circle (2pt);
    \filldraw[fill=white] (0.5,2) circle (2pt);
    \filldraw[fill=white] (0.5,1) circle (2pt);

    \coordinate (M0_1) at (0.5,3.9);
    \coordinate (M1_0) at (0.5,3.0);
    \coordinate (M1_1) at (0.5,2.9);
    \coordinate (M2_0) at (0.5,2.0);
    \coordinate (M2_1) at (0.5,1.9);
    \coordinate (M3_0) at (0.5,1.0);
    \draw[->, >=latex, line width=1.0pt] (M0_1) -- (M1_0);
    \draw[->, >=latex, line width=1.0pt] (M1_1) -- (M2_0);
    \draw[->, >=latex, line width=1.0pt] (M2_1) -- (M3_0);

    \node at (0.8, 4.25) {$\mathcal{M}_{0}$};
    \node at (0.8, 3.25) {$\mathcal{M}_{1}$};
    \node at (0.8, 2.25) {$\mathcal{M}_{2}$};
    \node at (0.8, 1.25) {$\mathcal{M}_{3}$};
    
    \foreach \y [count=\n] in {3, 2, 1, 0} {
        \draw[dashed] (0, \n) -- (12, \n);
        \node[right] at (12.1, \n) {Level \y};
    }

    \draw[thick] (1.5,4)--(2.5,4)--(3.0,3)--(4.0,3)--(4.5,2)--(5.5,2)--(6.0,1)--(7.0,1)--(7.5,2)--(8.5,2)--(9.0,3)--(10.0,3)--(10.5,4)--(11.5,4);
    \filldraw[fill=black] (1.5,4) circle (2pt);
    \filldraw[fill=black] (2.5,4) circle (2pt);
    \filldraw[fill=black] (3.0,3) circle (2pt);
    \filldraw[fill=black] (4.0,3) circle (2pt);
    \filldraw[fill=black] (4.5,2) circle (2pt);
    \filldraw[fill=black] (5.5,2) circle (2pt);
    \filldraw[fill=black] (6.0,1) circle (2pt);
    \filldraw[fill=black] (7.0,1) circle (2pt);
    \filldraw[fill=black] (7.5,2) circle (2pt);
    \filldraw[fill=black] (8.5,2) circle (2pt);
    \filldraw[fill=black] (9.0,3) circle (2pt);
    \filldraw[fill=black] (10.0,3) circle (2pt);
    \filldraw[fill=black] (10.5,4) circle (2pt);
    \filldraw[fill=black] (11.5,4) circle (2pt);

    \node at (2.0, 4.25) {$\mathcal{S}(B^{-1}_{\mathcal{M}_{0}})$};
    \node at (3.5, 3.25) {$\mathcal{S}(B^{-1}_{\mathcal{M}_{1}})$};
    \node at (5.0, 2.25) {$\mathcal{S}(B^{-1}_{\mathcal{M}_{2}})$};
    \node at (6.5, 1.25) {$\mathcal{S}(B^{-1}_{\mathcal{M}_{3}})$};
    \node at (8.0, 2.25) {$\mathcal{S}(B^{-1}_{\mathcal{M}_{2}})$};
    \node at (9.5, 3.25) {$\mathcal{S}(B^{-1}_{\mathcal{M}_{1}})$};
    \node at (11.0, 4.25) {$\mathcal{S}(B^{-1}_{\mathcal{M}_{0}})$};

    \node at (2.4, 3.5) {$\mathcal{I}_{\mathcal{M}_{0}}^{\mathcal{M}_{1}}$}; \node at (10.6, 3.5) {$\mathcal{I}_{\mathcal{M}_{1}}^{\mathcal{M}_{0}}$};
    \node at (3.9, 2.5) {$\mathcal{I}_{\mathcal{M}_{1}}^{\mathcal{M}_{2}}$}; \node at (9.1, 2.5) {$\mathcal{I}_{\mathcal{M}_{2}}^{\mathcal{M}_{1}}$};
    \node at (5.4, 1.5) {$\mathcal{I}_{\mathcal{M}_{2}}^{\mathcal{M}_{3}}$}; \node at (7.6, 1.5) {$\mathcal{I}_{\mathcal{M}_{3}}^{\mathcal{M}_{2}}$};

    \coordinate (A_0) at (3.0,4.0); \coordinate (B_0) at (10.0,4.0);
    \coordinate (A_1) at (4.5,3.0); \coordinate (B_1) at (8.5,3.0);
    \coordinate (A_2) at (6.0,2.0); \coordinate (B_2) at (7.0,2.0);
    \draw[->, >=latex, line width=1.0pt] (A_0) -- (B_0);
    \draw[->, >=latex, line width=1.0pt] (A_1) -- (B_1);
    \draw[->, >=latex, line width=1.0pt] (A_2) -- (B_2);

    \node at (6.5, 4.25) {$\mathcal{C}(\mathcal{M}_{0})$}; 
    \node at (6.5, 3.25) {$\mathcal{C}(\mathcal{M}_{1})$};
    \node at (6.5, 2.25) {$\mathcal{C}(\mathcal{M}_{2})$};
    
    \end{tikzpicture}
    \caption{Schematic description of V-cycle multilevel smoothed Schwarz framework. Here, $\mathcal{S}(B^{-1}_{\mathcal{M}_{i}})$ represents the smoother (both pre-smoothing and post-smoothing) based on RAS preconditioner $B^{-1}_{\mathcal{M}_{i}}$, and $\mathcal{C}(\mathcal{M}_{i})$ represents the error correction, for $i = 0, \dots, L-1$. }
    \label{fig: v-cycle}
\end{figure}
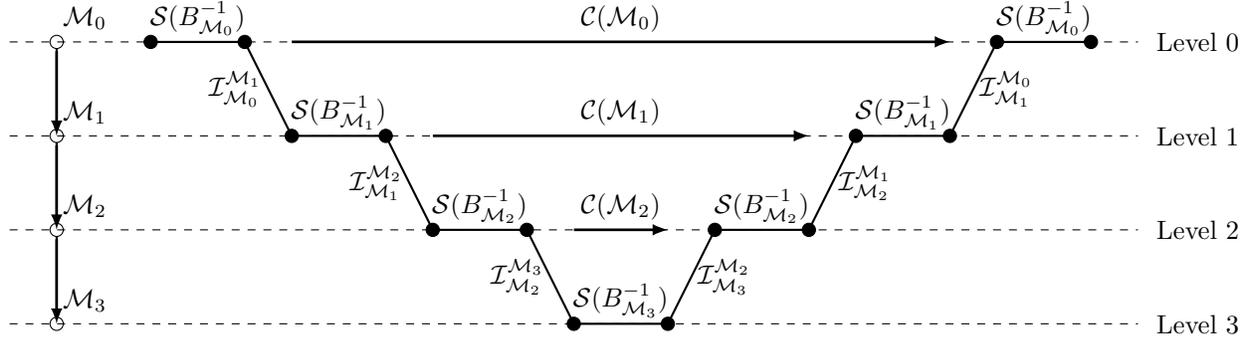

In practical applications, we typically only have access to the finest mesh $\mathcal{M}_0$. Therefore, the first step of the multilevel algorithm always involves constructing the coarse mesh hierarchy $\mathcal{M}_{i}$ for $i = 1, \dots, L-1$ from the finest mesh $\mathcal{M}_0$. To ensure that $\mathcal{M}_{i}$ inherits the boundary discretization from $\mathcal{M}_0$, it is essential to preserve the geometric boundaries during the mesh coarsening process. Since nested mesh hierarchies are less flexible for accommodating complex geometric boundaries,  considering a non-nested coarse hierarchy within the same computational domain $\bar{\Omega}_{h}$ can be beneficial. In parallel scenarios, as $\mathcal{M}_0$ is partitioned into $np$ subdomains, the construction of a non-nested mesh hierarchy presents significant challenges, which is the first major difficulty encountered. Additionally, the non-nested hierarchy typically leads to mismatched subdomain boundaries, complicating the design of mesh-to-mesh interpolation methods. These two main difficulties are summarized in Remark~\ref{challenges}.



\ 

\begin{remark}[\textbf{Main difficulties}]
In non-nested scenarios, the large-scale parallel implementation of the multilevel smoothed Schwarz preconditioner encounters two main difficulties:
\begin{itemize}
\item How can one rapidly construct a sequence of geometry-preserving non-nested coarse meshes in parallel based on the finest mesh? 
\item How can one design a flexible interpolation operator on domain-decomposed non-nested meshes when the subdomain boundaries of different levels are mismatched?
\end{itemize}
\label{challenges}
\end{remark}

\section{The parallel non-nested multilevel smoothed Schwarz preconditioner}
\label{sec: algorithm}
In this section, to address the challenges arising from the non-nested hierarchy, we introduce a new parallel geometry-preserving coarsening algorithm and a parallel interpolation method suitable for multiphysics problems on non-nested meshes. By combining these two techniques, a novel parallel non-nested multilevel smoothed Schwarz preconditioner is proposed.

\subsection{A parallel geometry-preserving mesh coarsening algorithm}
Coarsening from the finest mesh to generate a sequence of coarse mesh $\mathcal{M}_{i},\ i = 1, \dots, L-1$ always serves as the initial operation in multilevel algorithms. Without loss of generality, this subsection focuses on the process of coarsening from $\mathcal{M}_{i}$ to $\mathcal{M}_{i+1}$. To begin, we construct a graph $\hat{\mathcal{M}}_i:=(\mathcal{V}_{i}, \mathcal{E}_{i})$ based on the element set $\mathcal{K}_i$ and the interior vertex set $\mathcal{V}_i$ in mesh $\mathcal{M}_i$. Two vertices $v_1, v_2 \in \mathcal{V}_{i}$ are connected by an edge $e \in \mathcal{E}_i$ if and only if vertices $v_1, v_2$ both belong to the same element $k \in \mathcal{K}_i$. In Figure~\ref{fig: coarsening schematic}(a), as the mesh $\mathcal{M}_i$ is partitioned into $np$ subdomains, the graph $\hat{\mathcal{M}}_i$ is also distributed and stored across $np$ processors. Subsequently, we design an iterative algorithm based on $\hat{\mathcal{M}}_i$ to successively remove the interior vertices in parallel from the set $\mathcal{V}_i$ while preserving those located on the boundaries.

Above all, we choose the maximal independent set (MIS) \citep{luby1985simple} of $\hat{\mathcal{M}}_i$ as the interior vertices to be removed. That is, find a subset $\mathcal{S} \subset \mathcal{V}_{i}$ such that: 
\begin{itemize}
    \item (1) For every pair of vertices $v_1, v_2 \in \mathcal{S}$, there is no edge in $\mathcal{E}_i$ connecting them;
    \item (2) Every vertex $v_1 \in \mathcal{V}_{i} \setminus \mathcal{S}$ is connected to at least one vertex $v_2 \in \mathcal{S}$.
\end{itemize}
In ref. \citep{ma2024cpaft}, a highly parallel MIS algorithm is applicable for solving this problem. Figure~\ref{fig: coarsening schematic}(a) provides a schematic illustration of the selection of $\mathcal{S}$. If $\mathcal{S}$ is empty, it indicates that the interior vertex set $\mathcal{V}_{i}$ is empty, thereby terminating the algorithm. Otherwise, for every vertex $v \in \mathcal{S}$, select the elements in $\mathcal{K}_i$ containing $v$, forming $|\mathcal{S}|$ disjoint subsets of elements $\mathcal D_{k} \subset \mathcal{K}_i,\ k = 0, \dots, |\mathcal{S}|-1$, which represent the elements that need to be deleted. As shown in Figure~\ref{fig: coarsening schematic}(b), for $k = 0, \dots, |\mathcal{S}|-1$, every $\mathcal{D}_{k}$ forms a small polyhedron as ref. \citep{liu2017new}, with its surface denoted as $\mathcal{F}_{k}$. Since the mesh $\mathcal{M}_i$ has been decomposed into $np$ subdomains, the sets of surfaces $\mathcal{F}_{k},\ k = 0, \dots, |\mathcal{S}|-1$ are naturally distributed evenly across $np$ processors. For the surface $\mathcal{F} := \bigcup_{k = 0}^{|\mathcal{S}|-1} \mathcal{F}_{k}$, we introduce the consistent parallel advancing front technique (CPAFT) algorithm \citep{ma2024cpaft} with a coarse scale for generating new coarse elements to fill the small polyhedra formed by $\mathcal{D}_{k}$ as shown in Figure~\ref{fig: coarsening schematic}(c), benefiting from its boundary-preserving properties and excellent parallel scalability. Finally, each processor  independently merges the newly generated elements with the remaining elements to construct the coarse mesh $\mathcal{M}_{i+1}$. These designs guarantee that each step of the algorithm is highly parallel.

\begin{figure}[H]
    \centering
    \subfigure[\textbf{Original mesh} $NE=14, NV=12$]{
    \begin{minipage}[]{0.31\linewidth}
    \centering
    \begin{tikzpicture} [scale=1.6, every node/.style={scale=1}]
        \coordinate (A1) at (2, 1);
        \coordinate (A2) at (1, 2);
        \coordinate (A3) at (0, 1);
        \coordinate (A4) at (1, 0);
        \coordinate (A5) at (1.70711,1.70711);
        \coordinate (A6) at (1.70711,0.292893);
        \coordinate (A7) at (0.292893,0.292893);
        \coordinate (A8) at (0.292893,1.70711);
        \coordinate (A9) at (1.11133,0.731234);
        \coordinate (A10) at (0.849525,1.36328);
        \coordinate (A11) at (1.44768,1.19024);
        \coordinate (A12) at (0.54892,0.81796);

        \foreach \point in {A1, A2, A3, A4, A5, A6, A7, A8, A9, A10, A11, A12} {
            \filldraw[fill=black] (\point) circle (0.5pt);
        }

        \draw[thick] (A1)--(A5); \draw[thick] (A1)--(A6); \draw[thick] (A1)--(A11);
        \draw[thick] (A5)--(A11); \draw[thick] (A6)--(A11); \draw[thick] (A2)--(A5);
        \draw[thick] (A2)--(A11); \draw[thick] (A2)--(A10); \draw[thick] (A2)--(A8);
        \draw[thick] (A8)--(A10); \draw[thick] (A3)--(A8); \draw[thick] (A3)--(A12);
        \draw[thick] (A8)--(A12); \draw[thick] (A10)--(A12); \draw[thick] (A9)--(A10);
        \draw[thick] (A10)--(A11); \draw[thick] (A9)--(A11); \draw[thick] (A9)--(A12);
        \draw[thick] (A3)--(A7); \draw[thick] (A7)--(A12); \draw[thick] (A4)--(A12);
        \draw[thick] (A4)--(A7); \draw[thick] (A4)--(A9); \draw[thick] (A4)--(A6);
        \draw[thick] (A6)--(A9);

        \filldraw[fill=red] (A11) circle (1.5pt);
        \filldraw[fill=blue] (A12) circle (1.5pt);

        \draw[dashed, yellow, line width=1.5pt] (A2)--(A10)--(A9)--(A4);
    \end{tikzpicture}
    \end{minipage}
    }
    \subfigure[Delete elements]{
    \begin{minipage}[]{0.31\linewidth}
    \centering
    \begin{tikzpicture} [scale=1.6, every node/.style={scale=1}]
        \coordinate (A1) at (2, 1);
        \coordinate (A2) at (1, 2);
        \coordinate (A3) at (0, 1);
        \coordinate (A4) at (1, 0);
        \coordinate (A5) at (1.70711,1.70711);
        \coordinate (A6) at (1.70711,0.292893);
        \coordinate (A7) at (0.292893,0.292893);
        \coordinate (A8) at (0.292893,1.70711);
        \coordinate (A9) at (1.11133,0.731234);
        \coordinate (A10) at (0.849525,1.36328);
        \coordinate (A11) at (1.44768,1.19024);
        \coordinate (A12) at (0.54892,0.81796);

        \foreach \point in {A1, A2, A3, A4, A5, A6, A7, A8, A9, A10, A11, A12} {
            \filldraw[fill=black] (\point) circle (0.5pt);
        }

        \draw[thick] (A4)--(A6); \draw[thick] (A8)--(A2);

        \fill[red!50] (A5) -- (A1) -- (A6) -- (A9) -- (A10) -- (A2) -- cycle;
        \draw[thick, red, line width = 1.5pt] (A5) -- (A1) -- (A6) -- (A9) -- (A10) -- (A2) -- (A5);

        \fill[blue!50] (A10) -- (A9) -- (A4) -- (A7) -- (A3) -- (A8) -- cycle;
        \draw[thick, blue, line width = 1.5pt] (A10) -- (A9) -- (A4) -- (A7) -- (A3) -- (A8) -- (A10);

        
        \draw[dashed, yellow, line width=1.5pt] (A2)--(A10)--(A9)--(A4);
    \end{tikzpicture}
    \end{minipage}
    }
    \subfigure[$1_{st}$ Coarse mesh $NE=10, NV=10$]{
    \begin{minipage}[]{0.31\linewidth}
    \centering
    \begin{tikzpicture} [scale=1.6, every node/.style={scale=1}]
        \coordinate (A1) at (1, 0);
        \coordinate (A2) at (0.292893, 0.292893);
        \coordinate (A3) at (1.07457, 0.819961);
        \coordinate (A4) at (0.814915, 1.44683);
        \coordinate (A5) at (1, 2);
        \coordinate (A6) at (1.70711, 1.70711);
        \coordinate (A7) at (0, 1);
        \coordinate (A8) at (2, 1);
        \coordinate (A9) at (0.292893, 1.70711);
        \coordinate (A10) at (1.70711, 0.292893);

        \foreach \point in {A1, A2, A3, A4, A5, A6, A7, A8, A9, A10} {
            \filldraw[fill=black] (\point) circle (0.5pt);
        }

        \fill[blue!50] (A4) -- (A7) -- (A9) -- cycle;
        \fill[blue!50] (A4) -- (A3) -- (A7) -- cycle;
        \fill[blue!50] (A3) -- (A2) -- (A7) -- cycle;
        \fill[blue!50] (A3) -- (A1) -- (A2) -- cycle;

        \fill[red!50] (A5) -- (A4) -- (A6) -- cycle;
        \fill[red!50] (A4) -- (A3) -- (A6) -- cycle;
        \fill[red!50] (A6) -- (A3) -- (A8) -- cycle;
        \fill[red!50] (A8) -- (A3) -- (A10) -- cycle;

        \draw[thick] (A1)--(A2); \draw[thick] (A2)--(A3); \draw[thick] (A3)--(A1);
        \draw[thick] (A2)--(A7); \draw[thick] (A7)--(A3); \draw[thick] (A7)--(A4);
        \draw[thick] (A4)--(A3); \draw[thick] (A7)--(A9); \draw[thick] (A9)--(A4);
        \draw[thick] (A9)--(A5); \draw[thick] (A5)--(A4); \draw[thick] (A5)--(A6);
        \draw[thick] (A6)--(A4); \draw[thick] (A6)--(A3); \draw[thick] (A6)--(A8);
        \draw[thick] (A8)--(A3); \draw[thick] (A8)--(A10); \draw[thick] (A10)--(A3);
        \draw[thick] (A10)--(A1); \draw[thick] (A1)--(A3);

        \draw[dashed, yellow, line width=1.5pt] (A5)--(A4)--(A3)--(A1);
    \end{tikzpicture}
    \end{minipage}
    }
    \caption{Schematic description of the geometry-preserving mesh coarsening using $np = 2$ processors in one iteration. Yellow dashed lines represent the boundaries between two subdomains. Red and blue parts  respectively represent the domains processed by the two processors with CPAFT algorithm. The white parts represent the elements that are not processed in this iteration. }
    \label{fig: coarsening schematic}
\end{figure}
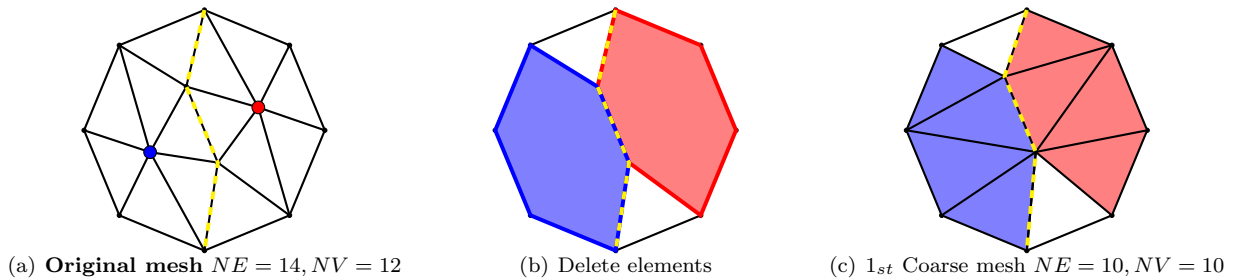

By repeating the above process $M$ times, we propose a parallel geometry-preserving coarsening algorithm in Algorithm~\ref{alg: coarsening}. Here, the maximal number of iterations $M$ is a user-adjustable parameter which can be used to control the coarsening ratio. Typically, we set $M=4$ for 2D problems and $M=8$ for 3D problems.


\begin{algorithm} [H]
\caption{The parallel geometry-preserving coarsening algorithm}
\label{alg: coarsening}
\begin{algorithmic}[1]
\State \textbf{Input:} The fine mesh $\mathcal{M}_{i}$.
\State Partition $\mathcal{M}_{i}$ into $np$ subdomains. Then, construct graph $\hat{\mathcal{M}}_{i}^{(1)}$ from $\mathcal{M}_{i}$ in parallel.
\For{\(its = 1 \) \textbf{to} \( M \)}
\State Find the MIS $\mathcal{S}$ of graph $\hat{\mathcal{M}}_{i}^{(its)}$ using the parallel MIS algorithm. If $|\mathcal{S}|=0$, break.
\State Delete all elements connected with $\mathcal{S}$ to obtain remaining elements $\mathcal{R}$ and internal surface $\mathcal{F}$.
\State Conduct CPAFT iterations using $\mathcal{F}$. The  generated elements are denoted as $\mathcal{G}$.
\State Construct $\hat{\mathcal{M}}_{i}^{(its+1)}$ based on elements $\mathcal{K}_{i}^{(its+1)}=\mathcal{R} \bigcup \mathcal{G}$ in parallel.
\EndFor
\State Construct $\mathcal{M}_{i+1}$ from graph $\hat{\mathcal{M}}_{i}^{(M)}$ in parallel.
\State \textbf{Output:} A coarse mesh $\mathcal{M}_{i+1}$.
\end{algorithmic}
\end{algorithm}

\begin{figure}[H]
    \subfigure[$its = 1: NE=10, NV = 10$]{
    \begin{minipage}[]{0.31\linewidth}
    \centering
    \begin{tikzpicture} [scale=1.6, every node/.style={scale=1}]
        \coordinate (A1) at (1, 0);
        \coordinate (A2) at (0.292893, 0.292893);
        \coordinate (A3) at (1.07457, 0.819961);
        \coordinate (A4) at (0.814915, 1.44683);
        \coordinate (A5) at (1, 2);
        \coordinate (A6) at (1.70711, 1.70711);
        \coordinate (A7) at (0, 1);
        \coordinate (A8) at (2, 1);
        \coordinate (A9) at (0.292893, 1.70711);
        \coordinate (A10) at (1.70711, 0.292893);

        \foreach \point in {A1, A2, A3, A4, A5, A6, A7, A8, A9, A10} {
            \filldraw[fill=black] (\point) circle (0.5pt);
        }

        \draw[thick] (A1)--(A2); \draw[thick] (A2)--(A3); \draw[thick] (A3)--(A1);
        \draw[thick] (A2)--(A7); \draw[thick] (A7)--(A3); \draw[thick] (A7)--(A4);
        \draw[thick] (A4)--(A3); \draw[thick] (A7)--(A9); \draw[thick] (A9)--(A4);
        \draw[thick] (A9)--(A5); \draw[thick] (A5)--(A4); \draw[thick] (A5)--(A6);
        \draw[thick] (A6)--(A4); \draw[thick] (A6)--(A3); \draw[thick] (A6)--(A8);
        \draw[thick] (A8)--(A3); \draw[thick] (A8)--(A10); \draw[thick] (A10)--(A3);
        \draw[thick] (A10)--(A1); \draw[thick] (A1)--(A3);

        \draw[dashed, yellow, line width=1.5pt] (A5)--(A4)--(A3)--(A1);
    \end{tikzpicture}
    \end{minipage}
    }\subfigure[$its = 2: NE=8, NV = 9$]{
    \begin{minipage}[]{0.31\linewidth}
    \centering
    \begin{tikzpicture} [scale=1.6, every node/.style={scale=1}]
        \coordinate (A1) at (1.10102, 0.898985);
        \coordinate (A2) at (0, 1);
        \coordinate (A3) at (1, 2);
        \coordinate (A4) at (0.292893, 1.70711);
        \coordinate (A5) at (1.70711, 1.70711);
        \coordinate (A6) at (2, 1);
        \coordinate (A7) at (1.70711, 0.292893);
        \coordinate (A8) at (1, 0);
        \coordinate (A9) at (0.292893, 0.292893);

        \foreach \point in {A1, A2, A3, A4, A5, A6, A7, A8, A9} {
            \filldraw[fill=black] (\point) circle (0.5pt);
        }

        \draw[thick] (A2)--(A4); \draw[thick] (A4)--(A3); \draw[thick] (A3)--(A2);
        \draw[thick] (A3)--(A1); \draw[thick] (A1)--(A2); \draw[thick] (A3)--(A5);
        \draw[thick] (A5)--(A1); \draw[thick] (A1)--(A9); \draw[thick] (A9)--(A2);
        \draw[thick] (A5)--(A6); \draw[thick] (A6)--(A1); \draw[thick] (A6)--(A7);
        \draw[thick] (A7)--(A1); \draw[thick] (A7)--(A8); \draw[thick] (A8)--(A1);
        \draw[thick] (A8)--(A9); \draw[thick] (A9)--(A1); \draw[thick] (A9)--(A2);

        \draw[dashed, yellow, line width=1.5pt] (A3)--(A1)--(A8);
    \end{tikzpicture}
    \end{minipage}
    } \subfigure[$its = 3: NE=6, NV = 8$ (\textbf{Output})]{
    \begin{minipage}[]{0.31\linewidth}
    \centering
    \begin{tikzpicture} [scale=1.6, every node/.style={scale=1}]
        \coordinate (A1) at (0, 1);
        \coordinate (A2) at (1, 2);
        \coordinate (A3) at (0.292893, 1.70711);
        \coordinate (A4) at (1.70711, 1.70711);
        \coordinate (A5) at (2, 1);
        \coordinate (A6) at (1.70711, 0.292893);
        \coordinate (A7) at (1, 0);
        \coordinate (A8) at (0.292893, 0.292893);

        \foreach \point in {A1, A2, A3, A4, A5, A6, A7, A8} {
            \filldraw[fill=black] (\point) circle (0.5pt);
        }

        \draw[thick] (A1)--(A3)--(A2)--(A4)--(A5)--(A6)--(A7)--(A8)--(A1);
        \draw[thick] (A7)--(A5); \draw[thick] (A8)--(A5); \draw[thick] (A2)--(A5); \draw[thick] (A8)--(A2); \draw[thick] (A1)--(A2);

        \draw[dashed, yellow, line width=1.5pt] (A2)--(A8);
    \end{tikzpicture}
    \end{minipage}
    }
    \caption{The 3 iterations of Algorithm~\ref{alg: coarsening} using $np = 2$ processors. Here, yellow dashed lines represent the boundaries between two subdomains. The algorithm terminates after three iterations because all interior vertices are removed. }
    \label{fig: coarsening iter}
\end{figure}
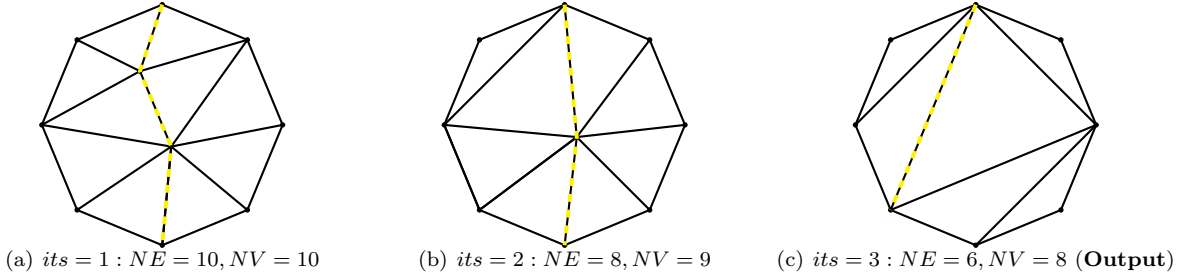

Figure~\ref{fig: coarsening iter} displays the results of three iterations applied to the original mesh $\mathcal{M}_{i}$ shown in Figure~\ref{fig: coarsening schematic}(a). As all interior vertices are removed, Algorithm~\ref{alg: coarsening} terminates after three iterations, with the final output $\mathcal{M}_{i+1}$ presented in Figure~\ref{fig: coarsening iter}(c). As the iteration progresses, interior vertices are progressively removed and the elements are continuously modified, resulting in gradual adjustments to the boundaries of the subdomains. Adjustments to the subdomain boundaries are implemented to promote a roughly balanced load of tasks across each subdomain, thereby facilitating satisfactory parallel efficiency. Finally, all coarse meshes generated by the coarsening algorithm are distributed and stored across $np$ processors, with each mesh level representing a domain decomposition of the computational domain $\Omega_h$,
\begin{equation}
    \label{eq: domain decomposition}
    \bar{\Omega}_h = \bigcup_{p = 1}^{np} \bar{\Omega}_{\mathcal{M}_{i}, p},\ \ \ i = 0, \dots, L-1.
\end{equation}


\subsection{A multiphysics-oriented parallel non-nested interpolation method}
Using the parallel geometry-preserving mesh coarsening algorithm, we obtain a non-nested mesh hierarchy $\mathcal{M}_{i},\ i = 0, \dots, L-1$. However, it can be observed from Figure~\ref{fig: coarsening iter} that the subdomain boundary in the output mesh $\mathcal{M}_{i+1}$ typically does not match that of the input mesh $\mathcal{M}_{i}$. For meshes at different levels, the mismatch of subdomain boundaries necessitates cross-processor processing to obtain nearby intersecting elements from other mesh levels, which presents a major challenge in designing the interpolation operator. Additionally, in multiphysics scenarios, the distribution of DoFs becomes increasingly intricate. In these cases, vertices, edges, and faces on the subdomain boundaries necessitate specialized handling, which further exacerbates the challenge associated with cross-processor processing. As far as we know, existing methods have not systematically addressed this challenge. In this work, we employ overlapping domain decomposition to address the challenge arising from mismatched subdomain boundaries. Subsequently, the field-splitting strategy \citep{calandrini2020field} and the MLS method \citep{levin1998approximation} are introduced for interpolation under varying distributions of DoFs. 


Without loss of generality, we consider interpolation from a mesh level $\tilde{\mathcal{M}}$ to another mesh level $\mathcal{M}$. Subsection~\ref{subsec: discretization} defines the set of unknowns $U_{\mathcal{M}}$ on a mesh level $\mathcal{M}$, implying that a node ${\bm x}_{*}$ may correspond to multiple unknowns. To begin, we decompose the DoFs at each node ${\bm x}_{*}$ using a splitting strategy based on physical variables. This implies that the distribution of DoFs, denoted as $\mathcal{N}$, is split as
\begin{equation}
    \left\{
    \begin{aligned}
        &\mathcal{N}=\sum_{f=1}^{F}\mathcal{N}^{f}, \ \mathcal{N}^{f} = (z^{f}_{0},\dots,z^{f}_{d})\in \mathbb{Z}_{\geq 0}^{d+1}, \\
        &z_{s}^{f}\leq 1,\ \ \ \ \ \ \ \ \ s = 0, \dots, d;\ \ f = 1, \dots, F.
    \end{aligned}
    \right.
\end{equation}
Then, with the splitting of $\mathcal{N}$, the unstructured mesh $\mathcal{M}$ can be treated in the splitting format shown in Figure~\ref{fig: field decomposition},
\begin{equation}
    \mathcal{M} = \sum_{f=1}^{F}\mathcal{M}^{f},\ \ \ \mathcal{M}^{f} = (\mathcal{K},\ \mathcal{N}^{f}).
    \label{equ: field-split}
\end{equation}
The goal of this splitting strategy is to ensure that there is at most one unknown per geometric node on the mesh. For each $f$, we consider geometric nodes $X_{\mathcal{M}^{f}} := \{{\bm x}_{n} \in X_{\mathcal{M}}\ |\ z_{s(n)}^{f} = 1\} \subset X_{\mathcal{M}}$ and corresponding unknowns $U_{\mathcal{M}^{f}} := \{u\ \texttt{at}\ (\bm x_n, 0)\ |\ {\bm x}_n \in X_{\mathcal{M}^{f}}\} \subset U_{\mathcal{M}}$. 
As a consequence, the complexities arising from the intricate distribution of DoFs are simplified to the handling of nodes in $X_{\mathcal{M}^{f}}, f = 1, \dots, F$, which facilitates easier parallel processing. Let us define an operator ${\bm R}_{f}$ that extracts vector $\bm u_{\mathcal{M}^{f}}$ associated with the $f_{th}$ physical variable from $\bm u_{\mathcal{M}}$. The treatment of mesh $\tilde{\mathcal{M}}$ is entirely analogous. Then, the unknowns on $\mathcal{M}$ can be interpolated in a splitting format based on $F$ physical variables,
\begin{equation}
    \label{equ: interpolation split}
    \mathcal{I}_{\tilde{\mathcal{M}}}^{\mathcal{M}} = \sum_{f=1}^{F} {\bm R}_{f}^{\top} \mathcal{I}_{\tilde{\mathcal{M}}^{f}}^{\mathcal{M}^{f}} \tilde{\bm R}_{f}.
\end{equation}

\begin{figure}[H]
    \centering
    \begin{tikzpicture} [xscale=2.4, yscale = 1.8, every node/.style={scale=1}]
        \coordinate (AA1) at (0.4854+0.15, 1.4158+0.75);
        \coordinate (AA2) at (0.3762+0.15, 1.1632+0.75);
        \coordinate (AA3) at (0.0905+0.15, 0.9947+0.75);
        \coordinate (AA4) at (0.9476+0.15, 1.3316+0.75);
        \coordinate (AA5) at (0.7571+0.15, 1.0789+0.75);
        \coordinate (AA6) at (0.4714+0.15, 0.8263+0.75);
        \coordinate (AA7) at (0.0905+0.15, 0.6579+0.75);
        \coordinate (AA8) at (1.3286+0.15, 1.3316+0.75);
        \coordinate (AA9) at (1.1381+0.15, 0.9947+0.75);
        \coordinate (AA10) at (0.8724+0.15, 0.7421+0.75);
        \coordinate (AA11) at (0.5467+0.15, 0.5037+0.75);
        \coordinate (AA12) at (0.1857+0.15, 0.3053+0.75);

        \draw[thick] (AA1) -- (AA2); \coordinate (Mid_1) at ($(AA1)!0.5!(AA2)$); \filldraw[fill=white] (Mid_1) circle (0.8pt);
        \draw[thick] (AA1) -- (AA4); \coordinate (Mid_2) at ($(AA1)!0.5!(AA4)$); \filldraw[fill=white] (Mid_2) circle (0.8pt);
        \draw[thick] (AA1) -- (AA5); \coordinate (Mid_3) at ($(AA1)!0.5!(AA5)$); \filldraw[fill=white] (Mid_3) circle (0.8pt);
        \draw[thick] (AA2) -- (AA3); \coordinate (Mid_4) at ($(AA2)!0.5!(AA3)$); \filldraw[fill=white] (Mid_4) circle (0.8pt);
        \draw[thick] (AA2) -- (AA5); \coordinate (Mid_5) at ($(AA2)!0.5!(AA5)$); \filldraw[fill=white] (Mid_5) circle (0.8pt);
        \draw[thick] (AA2) -- (AA6); \coordinate (Mid_6) at ($(AA2)!0.5!(AA6)$); \filldraw[fill=white] (Mid_6) circle (0.8pt);
        \draw[thick] (AA3) -- (AA6); \coordinate (Mid_7) at ($(AA3)!0.5!(AA6)$); \filldraw[fill=white] (Mid_7) circle (0.8pt);
        \draw[thick] (AA3) -- (AA7); \coordinate (Mid_8) at ($(AA3)!0.5!(AA7)$); \filldraw[fill=white] (Mid_8) circle (0.8pt);
        \draw[thick] (AA4) -- (AA5); \coordinate (Mid_9) at ($(AA4)!0.5!(AA5)$); \filldraw[fill=white] (Mid_9) circle (0.8pt);
        \draw[thick] (AA4) -- (AA8); \coordinate (Mid_10) at ($(AA4)!0.5!(AA8)$); \filldraw[fill=white] (Mid_10) circle (0.8pt);
        \draw[thick] (AA4) -- (AA9); \coordinate (Mid_11) at ($(AA4)!0.5!(AA9)$); \filldraw[fill=white] (Mid_11) circle (0.8pt);
        \draw[thick] (AA5) -- (AA6); \coordinate (Mid_12) at ($(AA5)!0.5!(AA6)$); \filldraw[fill=white] (Mid_12) circle (0.8pt);
        \draw[thick] (AA5) -- (AA9); \coordinate (Mid_13) at ($(AA5)!0.5!(AA9)$); \filldraw[fill=white] (Mid_13) circle (0.8pt);
        \draw[thick] (AA5) -- (AA10);\coordinate (Mid_14) at ($(AA5)!0.5!(AA10)$); \filldraw[fill=white] (Mid_14) circle (0.8pt);
        \draw[thick] (AA6) -- (AA7); \coordinate (Mid_15) at ($(AA6)!0.5!(AA7)$); \filldraw[fill=white] (Mid_15) circle (0.8pt);
        \draw[thick] (AA6) -- (AA10); \coordinate (Mid_16) at ($(AA6)!0.5!(AA10)$); \filldraw[fill=white] (Mid_16) circle (0.8pt);
        \draw[thick] (AA6) -- (AA11); \coordinate (Mid_17) at ($(AA6)!0.5!(AA11)$); \filldraw[fill=white] (Mid_17) circle (0.8pt);
        \draw[thick] (AA7) -- (AA11); \coordinate (Mid_18) at ($(AA7)!0.5!(AA11)$); \filldraw[fill=white] (Mid_18) circle (0.8pt);
        \draw[thick] (AA7) -- (AA12); \coordinate (Mid_19) at ($(AA7)!0.5!(AA12)$); \filldraw[fill=white] (Mid_19) circle (0.8pt);
        \draw[thick] (AA8) -- (AA9); \coordinate (Mid_20) at ($(AA8)!0.5!(AA9)$); \filldraw[fill=white] (Mid_20) circle (0.8pt); 
        \draw[thick] (AA9) -- (AA10); \coordinate (Mid_21) at ($(AA9)!0.5!(AA10)$); \filldraw[fill=white] (Mid_21) circle (0.8pt);
        \draw[thick] (AA10) -- (AA11); \coordinate (Mid_22) at ($(AA10)!0.5!(AA11)$); \filldraw[fill=white] (Mid_22) circle (0.8pt);
        \draw[thick] (AA11) -- (AA12); \coordinate (Mid_23) at ($(AA11)!0.5!(AA12)$); \filldraw[fill=white] (Mid_23) circle (0.8pt);

        \foreach \point in {AA1, AA2, AA3, AA4, AA5, AA6, AA7, AA8, AA9, AA10, AA11, AA12} {
            \filldraw[fill=gray] (\point) circle (0.8pt);
        }

        \draw[thick] (-0.9, 0.1+0.75) -- (1.4, 0.1+0.75) -- (2.4, 1.6+0.75) -- (0.1, 1.6+0.75) -- (-0.9, 0.1+0.75);

        \node at (1.0, 0.3+0.75) {$\mathcal{N} = (3, 2, 0)$};
        \node at (-0.45, 0.3+0.75) {$\mathcal{M}$};

        \filldraw [fill = black] (-0.2, 0.3) circle (0.8pt);
        \filldraw [fill = white] (0.4, 0.3) circle (0.8pt);
        \filldraw [fill = gray]  (1.0, 0.3) circle (0.8pt);

        \node at (-0.4, 0.3) [left] {legend};
        \node at (-0.1, 0.3) [right] {$1$};
        \node at (0.5, 0.3) [right] {$2$};
        \node at (1.1, 0.3) [right] {$3$};

        \coordinate (A1) at (0.4854+3.5, 1.4158);
        \coordinate (A2) at (0.3762+3.5, 1.1632);
        \coordinate (A3) at (0.0905+3.5, 0.9947);
        \coordinate (A4) at (0.9476+3.5, 1.3316);
        \coordinate (A5) at (0.7571+3.5, 1.0789);
        \coordinate (A6) at (0.4714+3.5, 0.8263);
        \coordinate (A7) at (0.0905+3.5, 0.6579);
        \coordinate (A8) at (1.3286+3.5, 1.3316);
        \coordinate (A9) at (1.1381+3.5, 0.9947);
        \coordinate (A10) at (0.8724+3.5, 0.7421);
        \coordinate (A11) at (0.5467+3.5, 0.5037);
        \coordinate (A12) at (0.1857+3.5, 0.3053);

        \fill[white] (-0.9+3.35, 0.1) -- (1.4+3.35, 0.1) -- (2.4+3.35, 1.6) -- (0.1+3.35, 1.6) -- (-0.9+3.35, 0.1)--cycle;

        \draw[thick] (-0.9+3.35, 0.1) -- (1.4+3.35, 0.1) -- (2.4+3.35, 1.6) -- (0.1+3.35, 1.6) -- (-0.9+3.35, 0.1);

        \draw[thick] (A1) -- (A2); \draw[thick] (A1) -- (A4); \draw[thick] (A1) -- (A5);
        \draw[thick] (A2) -- (A3); \draw[thick] (A2) -- (A5); \draw[thick] (A2) -- (A6);
        \draw[thick] (A3) -- (A6); \draw[thick] (A3) -- (A7);
        \draw[thick] (A4) -- (A5); \draw[thick] (A4) -- (A8); \draw[thick] (A4) -- (A9);
        \draw[thick] (A5) -- (A6); \draw[thick] (A5) -- (A9); \draw[thick] (A5) -- (A10);
        \draw[thick] (A6) -- (A7); \draw[thick] (A6) -- (A10); \draw[thick] (A6) -- (A11);
        \draw[thick] (A7) -- (A11); \draw[thick] (A7) -- (A12);
        \draw[thick] (A8) -- (A9); \draw[thick] (A9) -- (A10); \draw[thick] (A10) -- (A11); \draw[thick] (A11) -- (A12);

        \coordinate (MM_1) at ($(A3)!0.5!(A7)$); \filldraw[fill=black] (MM_1) circle (0.8pt);
        \coordinate (MM_2) at ($(A7)!0.5!(A12)$); \filldraw[fill=black] (MM_2) circle (0.8pt);
        \coordinate (MM_3) at ($(A7)!0.5!(A6)$); \filldraw[fill=black] (MM_3) circle (0.8pt);
        \coordinate (MM_4) at ($(A7)!0.5!(A11)$); \filldraw[fill=black] (MM_4) circle (0.8pt);
        \coordinate (MM_5) at ($(A6)!0.5!(A11)$); \filldraw[fill=black] (MM_5) circle (0.8pt);
        \coordinate (MM_6) at ($(A11)!0.5!(A12)$); \filldraw[fill=black] (MM_6) circle (0.8pt);
        \coordinate (MM_7) at ($(A11)!0.5!(A10)$); \filldraw[fill=black] (MM_7) circle (0.8pt);
        \coordinate (MM_8) at ($(A6)!0.5!(A10)$); \filldraw[fill=black] (MM_8) circle (0.8pt);

        \foreach \point in {A1, A2, A3, A4, A5, A6, A7, A8, A9, A10, A11, A12} {
            \filldraw[fill=black] (\point) circle (0.8pt);
        }

        \coordinate (B1) at (0.4854+3.5, 1.4158+0.75);
        \coordinate (B2) at (0.3762+3.5, 1.1632+0.75);
        \coordinate (B3) at (0.0905+3.5, 0.9947+0.75);
        \coordinate (B4) at (0.9476+3.5, 1.3316+0.75);
        \coordinate (B5) at (0.7571+3.5, 1.0789+0.75);
        \coordinate (B6) at (0.4714+3.5, 0.8263+0.75);
        \coordinate (B7) at (0.0905+3.5, 0.6579+0.75);
        \coordinate (B8) at (1.3286+3.5, 1.3316+0.75);
        \coordinate (B9) at (1.1381+3.5, 0.9947+0.75);
        \coordinate (B10) at (0.8724+3.5, 0.7421+0.75);
        \coordinate (B11) at (0.5467+3.5, 0.5037+0.75);
        \coordinate (B12) at (0.1857+3.5, 0.3053+0.75);

        \fill[white] (-0.9+3.35, 0.1+0.75) -- (1.4+3.35, 0.1+0.75) -- (2.4+3.35, 1.6+0.75) -- (0.1+3.35, 1.6+0.75) -- (-0.9+3.35, 0.1+0.75)--cycle;

        \draw[thick] (-0.9+3.35, 0.1+0.75) -- (1.4+3.35, 0.1+0.75) -- (2.4+3.35, 1.6+0.75) -- (0.1+3.35, 1.6+0.75) -- (-0.9+3.35, 0.1+0.75);

        \draw[thick] (B1) -- (B2); \draw[thick] (B1) -- (B4); \draw[thick] (B1) -- (B5);
        \draw[thick] (B2) -- (B3); \draw[thick] (B2) -- (B5); \draw[thick] (B2) -- (B6);
        \draw[thick] (B3) -- (B6); \draw[thick] (B3) -- (B7);
        \draw[thick] (B4) -- (B5); \draw[thick] (B4) -- (B8); \draw[thick] (B4) -- (B9);
        \draw[thick] (B5) -- (B6); \draw[thick] (B5) -- (B9); \draw[thick] (B5) -- (B10);
        \draw[thick] (B6) -- (B7); \draw[thick] (B6) -- (B10); \draw[thick] (B6) -- (B11);
        \draw[thick] (B7) -- (B11); \draw[thick] (B7) -- (B12);
        \draw[thick] (B8) -- (B9); \draw[thick] (B9) -- (B10); \draw[thick] (B10) -- (B11); \draw[thick] (B11) -- (B12);

        \coordinate (M_1) at ($(B3)!0.5!(B7)$); \filldraw[fill=black] (M_1) circle (0.8pt);
        \coordinate (M_2) at ($(B7)!0.5!(B12)$); \filldraw[fill=black] (M_2) circle (0.8pt);
        \coordinate (M_3) at ($(B7)!0.5!(B6)$); \filldraw[fill=black] (M_3) circle (0.8pt);
        \coordinate (M_4) at ($(B7)!0.5!(B11)$); \filldraw[fill=black] (M_4) circle (0.8pt);
        \coordinate (M_5) at ($(B6)!0.5!(B11)$); \filldraw[fill=black] (M_5) circle (0.8pt);
        \coordinate (M_6) at ($(B11)!0.5!(B12)$); \filldraw[fill=black] (M_6) circle (0.8pt);
        \coordinate (M_7) at ($(B11)!0.5!(B10)$); \filldraw[fill=black] (M_7) circle (0.8pt);
        \coordinate (M_8) at ($(B6)!0.5!(B10)$); \filldraw[fill=black] (M_8) circle (0.8pt);

        \foreach \point in {B1, B2, B3, B4, B5, B6, B7, B8, B9, B10, B11, B12} {
            \filldraw[fill=black] (\point) circle (0.8pt);
        }

        \coordinate (C1) at (0.4854+3.5, 1.4158+1.5);
        \coordinate (C2) at (0.3762+3.5, 1.1632+1.5);
        \coordinate (C3) at (0.0905+3.5, 0.9947+1.5);
        \coordinate (C4) at (0.9476+3.5, 1.3316+1.5);
        \coordinate (C5) at (0.7571+3.5, 1.0789+1.5);
        \coordinate (C6) at (0.4714+3.5, 0.8263+1.5);
        \coordinate (C7) at (0.0905+3.5, 0.6579+1.5);
        \coordinate (C8) at (1.3286+3.5, 1.3316+1.5);
        \coordinate (C9) at (1.1381+3.5, 0.9947+1.5);
        \coordinate (C10) at (0.8724+3.5, 0.7421+1.5);
        \coordinate (C11) at (0.5467+3.5, 0.5037+1.5);
        \coordinate (C12) at (0.1857+3.5, 0.3053+1.5);

        \fill[white] (-0.9+3.35, 0.1+1.5) -- (1.4+3.35, 0.1+1.5) -- (2.4+3.35, 1.6+1.5) -- (0.1+3.35, 1.6+1.5) -- (-0.9+3.35, 0.1+1.5) -- cycle;

        \draw[thick] (-0.9+3.35, 0.1+1.5) -- (1.4+3.35, 0.1+1.5) -- (2.4+3.35, 1.6+1.5) -- (0.1+3.35, 1.6+1.5) -- (-0.9+3.35, 0.1+1.5);

        \draw[thick] (C1) -- (C2); \draw[thick] (C1) -- (C4); \draw[thick] (C1) -- (C5);
        \draw[thick] (C2) -- (C3); \draw[thick] (C2) -- (C5); \draw[thick] (C2) -- (C6);
        \draw[thick] (C3) -- (C6); \draw[thick] (C3) -- (C7);
        \draw[thick] (C4) -- (C5); \draw[thick] (C4) -- (C8); \draw[thick] (C4) -- (C9);
        \draw[thick] (C5) -- (C6); \draw[thick] (C5) -- (C9); \draw[thick] (C5) -- (C10);
        \draw[thick] (C6) -- (C7); \draw[thick] (C6) -- (C10); \draw[thick] (C6) -- (C11);
        \draw[thick] (C7) -- (C11); \draw[thick] (C7) -- (C12);
        \draw[thick] (C8) -- (C9); \draw[thick] (C9) -- (C10); \draw[thick] (C10) -- (C11); \draw[thick] (C11) -- (C12);

        \foreach \point in {C1, C2, C3, C4, C5, C6, C7, C8, C9, C10, C11, C12} {
            \filldraw[fill=black] (\point) circle (0.8pt);
        }

        \node at (1.0+3.35, 0.3) {$\mathcal{N}^{1} = (1, 1, 0)$};
        \node at (-0.45+3.35, 0.3) {$\mathcal{M}^{1}$};
        \node at (1.0+3.35, 0.3+0.75) {$\mathcal{N}^{2} = (1, 1, 0)$};
        \node at (-0.45+3.35, 0.3+0.75) {$\mathcal{M}^{2}$};
        \node at (1.0+3.35, 0.3+1.5) {$\mathcal{N}^{3} = (1, 0, 0)$};
        \node at (-0.45+3.35, 0.3+1.5) {$\mathcal{M}^{3}$};

        \node at (-0.7+3.35, 0.65) {$+$}; \node at (-0.7+3.35, 0.65+0.75) {$+$};
        \node at (2.2, 1.5) {$=$};
    \end{tikzpicture}
    \caption{Illustration of the splitting of $\mathcal{M}$. Here, we use Figure~\ref{multiphysics}(c) as an example. $\mathcal{N} = (3,2,0) = (1, 1, 0) + (1, 1, 0) + (1, 0, 0)$ represents the split into two velocity components and one pressure component. Then, the unknowns on $\mathcal{M}^{1}$, $\mathcal{M}^{2}$, and $\mathcal{M}^{3}$ correspond to the three physical variables are interpolated separately.}
    \label{fig: field decomposition}
\end{figure}
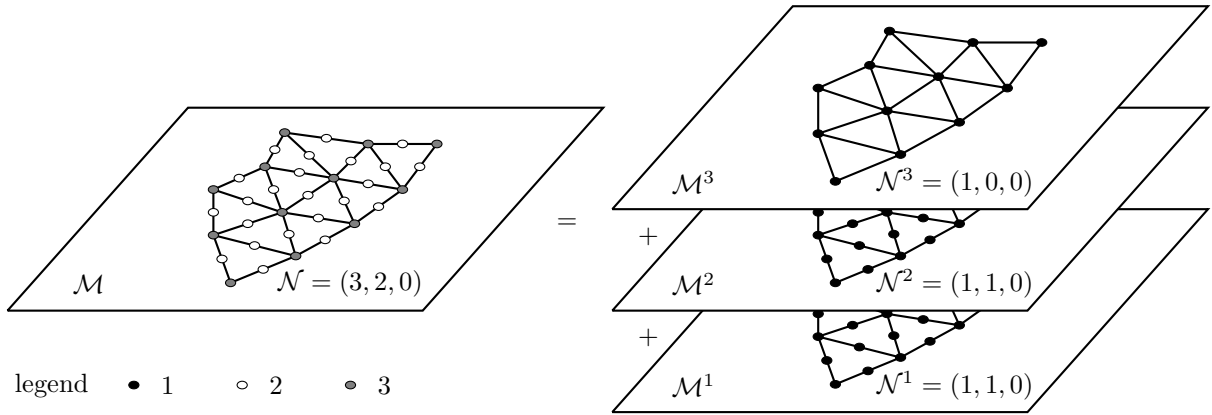

Our problem reduces to designing the interpolation operator $\mathcal{I}_{{\tilde{\mathcal{M}}}^{f}}^{{\mathcal{M}}^{f}}$ in \eqref{equ: interpolation split} associated with the $f_{th}$ physical variable for $f = 1, \dots, F$. 
In accordance with equation~\eqref{eq: domain decomposition}, as the domain $\Omega_{h}$ has been decomposed into $np$ subdomains on each mesh level, the unknowns and geometric nodes have been correspondingly allocated to $np$ processors. 
Specifically, the unknowns in each subdomain are denoted as $U_{{\mathcal{M}}^{f}, p}$ and $U_{{\tilde{\mathcal{M}}^{f}, p}}$, while the sets of geometric nodes are denoted by $X_{\mathcal{M}^{f}, p}$ and $X_{\tilde{\mathcal{M}}^{f}, p}$, where $p = 1, \dots, np$.
In our multilevel smoothed Schwarz algorithm, we employ the RAS preconditioners as smoothers on each mesh level. Consequently, the subdomains $\Omega_{\tilde{\mathcal{M}}, p}$ are extended to $\Omega^{\delta}_{\tilde{\mathcal{M}}, p}$ as described in Subsection~\ref{subsec: RAS} for $p = 1, \dots, np$, where the overlap is typically set as $\delta = 1$. 
Then, the local overlapping vectors $\bm u^{\delta}_{\mathcal{M}^{f},p}$, $\bm u^{\delta}_{\tilde{\mathcal{M}}^{f},p}$ and restriction operators ${\bm R}_{\mathcal{M}^{f}, p}^{\delta}$, ${\bm R}_{\tilde{\mathcal{M}}^{f}, p}^{\delta}$ are provided similarly to Subsection~\ref{subsec: RAS}.
As shown in Figure~\ref{fig: interpolation}, to interpolate an unknown $u_{*} \in U_{{\mathcal{M}}^{f}, p}$ located at a geometric node $\bm{x}_{*} \in X_{\mathcal{M}^{f}, p}$ in the $p_{th}$ subdomain, we introduce an overlapping MLS method based on the overlapping domain decomposition. 
We define $\tilde{N}(\bm x_{*}) := \{\tilde{{\bm x}}_{n} \in X_{\tilde{\mathcal{M}}^{f}}|\ \tilde{{\bm x}}_{n}\ \texttt{in}\ \Omega_{\tilde{\mathcal{M}}, p}^{\delta},  \|\tilde{{\bm x}}_{n}-{\bm x_{*}}\| < \rho_{\tilde{\mathcal{M}}}({\bm x}_{*})\} \subset X_{\tilde{\mathcal{M}}^{f}}$ as the neighborhood of $\bm x_{*}$ using an appropriate scale $\rho_{\tilde{\mathcal{M}}}({\bm x}_{*})$. With this definition, for every $\tilde{\bm x}_{n} \in \tilde{N}(\bm x_{*})$, there exists an unknown $\tilde{u}_{n} \in U_{\tilde{\mathcal{M}}^{f}}$ located at $\tilde{\bm x}_{n}$. Given a polynomial space $\mathcal{P}$ which bases can be presented as a column vector $\bm p(\bm x) = [p_1(\bm x), \dots, p_{|\mathcal{P}|}(\bm x)]^{\top}$, the MLS approximation from $\{ \tilde{u}_{n} \}$ to $u_{*} \in U_{{\mathcal{M}}^{f}, p}$ is achieved by finding $\bm a = (a_1, \dots, a_{|\mathcal{P}|})^{\top} \in \mathbb{R}^{|\mathcal{P}|}$ through the following minimization problem,
\begin{equation}
    \label{equ: min}
    \underset{\bm a \in \mathbb{R}^{|\mathcal{P}|}}{\texttt{arg min}} \ \mathcal{L}_{\tilde{\mathcal{M}}^{f}}(\bm a),
\end{equation}
where the loss $\mathcal{L}_{\tilde{\mathcal{M}}^{f}}(\bm a )$ with a positive weight function $w_{\tilde{\mathcal{M}}}$ \citep{gu2019generalized} is defined as
\begin{align}
    \label{equ: loss}
    &\mathcal{L}_{\tilde{\mathcal{M}}^{f}}(\bm a )\ \ \ = \sum_{\tilde{\bm x}_{n} \in \tilde{N}(\bm x_{*})} w_{\tilde{\mathcal{M}}}(\bm x_{*} - \tilde{\bm x}_{n}) [\bm p(\tilde{\bm x}_{n})^{\top}{\bm a} - \tilde{u}_{n}]^{2}, \\
    &w_{\tilde{\mathcal{M}}}(\bm x_{*} - \tilde{\bm x}_{n}) = 1 - 6(\|\bm x_{*} - \tilde{\bm x}_{n}\|/\rho_{\tilde{\mathcal{M}}}({\bm x}_{*}))^2 + 8 (\| \bm x_{*} - \tilde{\bm x}_{n}\|/\rho_{\tilde{\mathcal{M}}}({\bm x}_{*}))^3 - 3 (\|\bm x_{*} - \tilde{\bm x}_{n}\|/\rho_{\tilde{\mathcal{M}}}({\bm x}_{*}))^4.
\end{align}

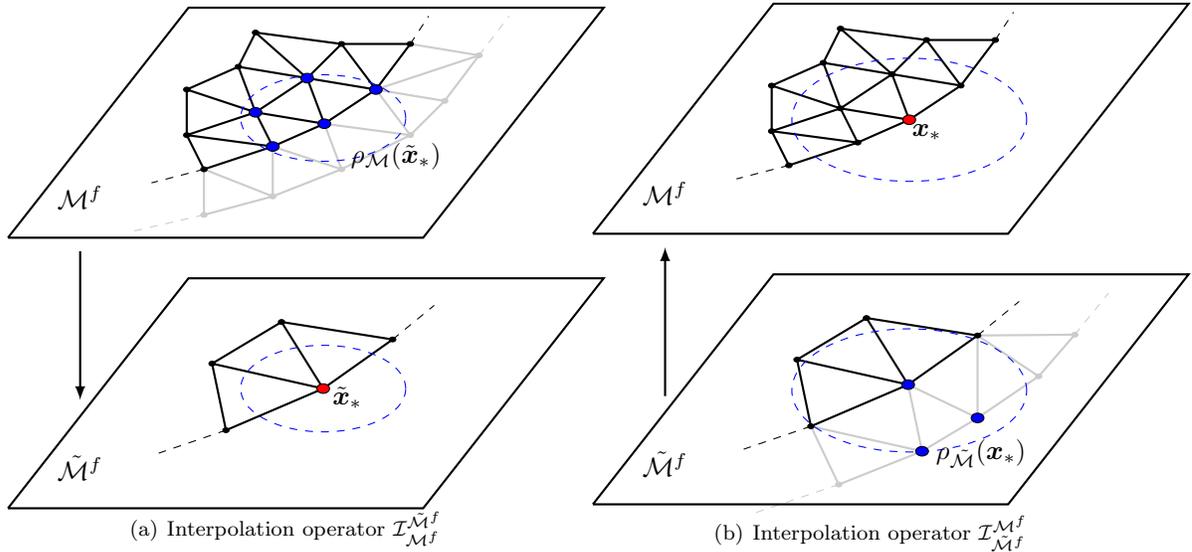
\begin{figure}[H]
    \centering
    \subfigure[Interpolation operator $\mathcal{I}_{\mathcal{M}^{f}}^{\tilde{\mathcal{M}}^{f}}$]{
    \begin{minipage}[]{0.45\linewidth}
    \centering
    \begin{tikzpicture} [xscale=2.4, yscale = 1.8, every node/.style={scale=1}]
        \coordinate (A1) at (0.4714, 1.4158+2.0);
        \coordinate (A2) at (0.3762, 1.1632+2.0);
        \coordinate (A3) at (0.0905, 0.9947+2.0);
        \coordinate (A4) at (0.9476, 1.3316+2.0);
        \coordinate (A5) at (0.7571, 1.0789+2.0);
        \coordinate (A6) at (0.4714, 0.8263+2.0);
        \coordinate (A7) at (0.0905, 0.6579+2.0);
        \coordinate (A8) at (1.3286, 1.3316+2.0);
        \coordinate (A9) at (1.1381, 0.9947+2.0);
        \coordinate (A10) at (0.8524, 0.7421+2.0);
        \coordinate (A11) at (0.5667, 0.5737+2.0);
        \coordinate (A12) at (0.1857, 0.4053+2.0);
        \coordinate (A13) at (1.7095, 1.2474+2.0);
        \coordinate (A14) at (1.5190, 0.9105+2.0);
        \coordinate (A15) at (1.3286, 0.6579+2.0);
        \coordinate (A16) at (0.9476, 0.4053+2.0);
        \coordinate (A17) at (0.5667, 0.2052+2.0);
        \coordinate (A18) at (0.1857, 0.0684+2.0);

        \draw[thick] (A1) -- (A2); \draw[thick] (A1) -- (A4); \draw[thick] (A1) -- (A5);
        \draw[thick] (A2) -- (A3); \draw[thick] (A2) -- (A5); \draw[thick] (A2) -- (A6);
        \draw[thick] (A3) -- (A6); \draw[thick] (A3) -- (A7);
        \draw[thick] (A4) -- (A5); \draw[thick] (A4) -- (A8); \draw[thick] (A4) -- (A9);
        \draw[thick] (A5) -- (A6); \draw[thick] (A5) -- (A9); \draw[thick] (A5) -- (A10);
        \draw[thick] (A6) -- (A7); \draw[thick] (A6) -- (A10); \draw[thick] (A6) -- (A11);
        \draw[thick] (A7) -- (A11); \draw[thick] (A7) -- (A12);
        \draw[thick] (A8) -- (A9); \draw[thick] (A9) -- (A10); \draw[thick] (A10) -- (A11); \draw[thick] (A11) -- (A12);

        \draw[gray!40, thick] (A8) -- (A13); \draw[gray!40, thick] (A9) -- (A13);
        \draw[gray!40, thick] (A9) -- (A14); \draw[gray!40, thick] (A13) -- (A14);
        \draw[gray!40, thick] (A9) -- (A15); \draw[gray!40, thick] (A14) -- (A15);
        \draw[gray!40, thick] (A10) -- (A15); \draw[gray!40, thick] (A10) -- (A16); \draw[gray!40, thick] (A15) -- (A16);
        \draw[gray!40, thick] (A11) -- (A16); \draw[gray!40, thick] (A11) -- (A17); \draw[gray!40, thick] (A16) -- (A17);
        \draw[gray!40, thick] (A12) -- (A17); \draw[gray!40, thick] (A12) -- (A18); \draw[gray!40, thick] (A17) -- (A18);

        \foreach \point in {A1, A2, A3, A4, A5, A6, A7, A8, A9, A10, A11, A12} {
            \filldraw[fill=black] (\point) circle (0.5pt);
        }
        \foreach \point in {A13, A14, A15, A16, A17, A18} {
            \filldraw[gray!40] (\point) circle (0.5pt);
        }

        \draw[dashed] (A12) -- (-0.1048, 0.3053+2.0); \draw[dashed] (A8) -- (1.4299, 1.5376+2.0);
        \draw[gray!40, dashed] (A18) -- (-0.2, -0.05+2.0); \draw[gray!40, dashed] (A13) -- (1.8764, 1.5389+2.0);

        \draw[thick] (-0.9, 1.9) -- (1.4, 1.9) -- (2.4, 3.6) -- (0.1, 3.6) -- (-0.9, 1.9);

        \coordinate (B1) at (0.6154, 1.2769);
        \coordinate (B2) at (0.2308, 0.9692);
        \coordinate (B3) at (1.2308, 1.1477);
        \coordinate (B4) at (0.8462, 0.7846);
        \coordinate (B5) at (0.3077, 0.4769);
        \coordinate (B6) at (1.7693, 1.1539);
        \coordinate (B7) at (1.5692, 0.8462);
        \coordinate (B8) at (1.2308, 0.5385);
        \coordinate (B9) at (0.9231, 0.2923);
        \coordinate (B10) at (0.4616, 0.0462);

        \draw[thick] (B1) -- (B2); \draw[thick] (B1) -- (B3); \draw[thick] (B1) -- (B4);
        \draw[thick] (B2) -- (B4); \draw[thick] (B2) -- (B5);
        \draw[thick] (B3) -- (B4); \draw[thick] (B4) -- (B5);

        
        \foreach \point in {B1, B2, B3, B4, B5} {
            \filldraw[fill=black] (\point) circle (0.5pt);
        }

        \filldraw[fill=red] (B4) circle (1.0pt);
        \foreach \point in {A5, A6, A9, A10, A11} {
            \filldraw[fill=blue] (\point) circle (1.0pt);
        }
        \draw[blue, dashed] (B4) ellipse (13pt and 9.1pt);
        \draw[blue, dashed] (0.8462, 0.7846+2.0) ellipse (13pt and 9.1pt);

        \draw[dashed] (B5) -- (-0.0648, 0.3075); \draw[dashed] (B3) -- (1.4615, 1.4);

        \draw[thick] (-0.9, -0.1) -- (1.4, -0.1) -- (2.4, 1.6) -- (0.1, 1.6) -- (-0.9, -0.1);

        \draw[->, >=latex, line width=1.0pt, thick, black] (-0.5, 1.8) -- (-0.5, 0.7);

        \node at (-0.5, 0.2) {$\tilde{\mathcal{M}}^{f}$};
        \node at (-0.5, 2.2) {$\mathcal{M}^{f}$};

        \node at (1.25, 2.5) {$\rho_{\mathcal{M}}(\tilde{\bm x}_{*})$};
        \node at (0.98, 0.72) {$\tilde{\bm x}_{*}$};
    \end{tikzpicture}
    \end{minipage}
    }
    \subfigure[Interpolation operator $\mathcal{I}_{\tilde{\mathcal{M}}^{f}}^{\mathcal{M}^{f}}$]{
    \begin{minipage}[]{0.45\linewidth}
    \centering
    \begin{tikzpicture} [xscale=2.4, yscale = 1.8, every node/.style={scale=1}]
        \coordinate (A1) at (0.4714, 1.4158+2.0);
        \coordinate (A2) at (0.3762, 1.1632+2.0);
        \coordinate (A3) at (0.0905, 0.9947+2.0);
        \coordinate (A4) at (0.9476, 1.3316+2.0);
        \coordinate (A5) at (0.7571, 1.0789+2.0);
        \coordinate (A6) at (0.4714, 0.8263+2.0);
        \coordinate (A7) at (0.0905, 0.6579+2.0);
        \coordinate (A8) at (1.3286, 1.3316+2.0);
        \coordinate (A9) at (1.1381, 0.9947+2.0);
        \coordinate (A10) at (0.8524, 0.7421+2.0);
        \coordinate (A11) at (0.5667, 0.5737+2.0);
        \coordinate (A12) at (0.1857, 0.4053+2.0);
        \coordinate (A13) at (1.7095, 1.2474+2.0);
        \coordinate (A14) at (1.5190, 0.9105+2.0);
        \coordinate (A15) at (1.3286, 0.6579+2.0);
        \coordinate (A16) at (0.9476, 0.4053+2.0);
        \coordinate (A17) at (0.5667, 0.2052+2.0);
        \coordinate (A18) at (0.1857, 0.0684+2.0);

        \draw[thick] (A1) -- (A2); \draw[thick] (A1) -- (A4); \draw[thick] (A1) -- (A5);
        \draw[thick] (A2) -- (A3); \draw[thick] (A2) -- (A5); \draw[thick] (A2) -- (A6);
        \draw[thick] (A3) -- (A6); \draw[thick] (A3) -- (A7);
        \draw[thick] (A4) -- (A5); \draw[thick] (A4) -- (A8); \draw[thick] (A4) -- (A9);
        \draw[thick] (A5) -- (A6); \draw[thick] (A5) -- (A9); \draw[thick] (A5) -- (A10);
        \draw[thick] (A6) -- (A7); \draw[thick] (A6) -- (A10); \draw[thick] (A6) -- (A11);
        \draw[thick] (A7) -- (A11); \draw[thick] (A7) -- (A12);
        \draw[thick] (A8) -- (A9); \draw[thick] (A9) -- (A10); \draw[thick] (A10) -- (A11); \draw[thick] (A11) -- (A12);


        \foreach \point in {A1, A2, A3, A4, A5, A6, A7, A8, A9, A10, A11, A12} {
            \filldraw[fill=black] (\point) circle (0.5pt);
        }

        \draw[dashed] (A12) -- (-0.1048, 0.3053+2.0); \draw[dashed] (A8) -- (1.4299, 1.5376+2.0);

        \draw[thick] (-0.9, 1.9) -- (1.4, 1.9) -- (2.4, 3.6) -- (0.1, 3.6) -- (-0.9, 1.9);

        \coordinate (B1) at (0.6154, 1.2769);
        \coordinate (B2) at (0.2308, 0.9692);
        \coordinate (B3) at (1.2308, 1.1477);
        \coordinate (B4) at (0.8462, 0.7846);
        \coordinate (B5) at (0.3077, 0.4769);
        \coordinate (B6) at (1.7693, 1.1539);
        \coordinate (B7) at (1.5692, 0.8462);
        \coordinate (B8) at (1.2308, 0.5385);
        \coordinate (B9) at (0.9231, 0.2923);
        \coordinate (B10) at (0.4616, 0.0462);

        \draw[thick] (B1) -- (B2); \draw[thick] (B1) -- (B3); \draw[thick] (B1) -- (B4);
        \draw[thick] (B2) -- (B4); \draw[thick] (B2) -- (B5);
        \draw[thick] (B3) -- (B4); \draw[thick] (B4) -- (B5);

        \draw[gray!40, thick] (B3) -- (B6); \draw[gray!40, thick] (B3) -- (B7); \draw[gray!40, thick] (B6) -- (B7);
        \draw[gray!40, thick] (B3) -- (B8); \draw[gray!40, thick] (B7) -- (B8);
        \draw[gray!40, thick] (B4) -- (B8); \draw[gray!40, thick] (B4) -- (B9); \draw[gray!40, thick] (B8) -- (B9);
        \draw[gray!40, thick] (B5) -- (B9); \draw[gray!40, thick] (B5) -- (B10); \draw[gray!40, thick] (B9) -- (B10);
        
        \foreach \point in {B1, B2, B3, B4, B5} {
            \filldraw[fill=black] (\point) circle (0.5pt);
        }
        \foreach \point in {B6, B7, B8, B9, B10} {
            \filldraw[gray!40] (\point) circle (0.5pt);
        }

        \filldraw[fill=red] (A10) circle (1.0pt);
        \draw[blue, dashed] (A10) ellipse (18.5pt and 12.95pt);
        \draw[blue, dashed] (0.8524, 0.7421) ellipse (18.5pt and 12.95pt);
        \foreach \point in {B4, B8, B9} {
            \filldraw[fill=blue] (\point) circle (1.0pt);
        }

        \draw[dashed] (B5) -- (-0.0648, 0.3075); \draw[dashed] (B3) -- (1.4615, 1.4);
        \draw[gray!40, dashed] (B10) -- (0, -0.1623); \draw[gray!40, dashed] (B6) -- (1.9556, 1.46);

        \draw[thick] (-0.9, -0.1) -- (1.4, -0.1) -- (2.4, 1.6) -- (0.1, 1.6) -- (-0.9, -0.1);

        \draw[->, >=latex, line width=1.0pt, thick, black] (-0.5, 0.7) -- (-0.5, 1.8);

        \node at (-0.5, 0.2) {$\tilde{\mathcal{M}}^{f}$};
        \node at (-0.5, 2.2) {$\mathcal{M}^{f}$};
        \node at (1.25, 0.28) {$\rho_{\tilde{\mathcal{M}}}(\bm x_{*})$};
        \node at (0.95, 2.65) {$\bm x_{*}$};
    \end{tikzpicture}
    \end{minipage}
    }
    \caption{Schematic description of non-nested interpolation operators (a) $I_{\mathcal{M}^{f}}^{\tilde{\mathcal{M}}^{f}}$ and (b) $I^{\mathcal{M}^{f}}_{\tilde{\mathcal{M}}^{f}}$. Here, we take $\mathcal{M}^{3}$ in Figure~\ref{fig: field decomposition} as an example. The black parts represent subdomains, while the gray parts denote $\delta = 1$ overlapping layers. The red nodes represent the target nodes ${\bm x}_{*}$, while the blue nodes denote $\tilde{N}({\bm x}_{*})$.}
    \label{fig: interpolation}
\end{figure}

Let us now present the unique solvability of Problem~\eqref{equ: min}, which demonstrates the feasibility of designing interpolation operator using the MLS approximation $u_{*} = {\bm p}(\bm x_{*})^{\top}\bm a$.
\begin{lemma}
    \label{lem: convex}
    \(\mathcal{L}_{\tilde{\mathcal{M}}^{f}}(\bm a)\) is convex with respect to \(\bm a \in \mathbb{R}^{|\mathcal{P}|}\).
\end{lemma}
\begin{proof}
    Let us consider \({\bm a}_1, \bm a_2 \in \mathbb{R}^{|\mathcal{P}|}\), and \(0 \leq t \leq 1\). Then, for all \(\tilde{\bm x}_{n} \in \tilde{N}({\bm x}_{*})\), we have
    \begin{equation}
        \label{equ: convex}
        \begin{aligned}
        &t[\bm p(\tilde{\bm x}_{n})^{\top}{\bm a}_{1}-\tilde{u}_{n}]^2 + (1-t)[\bm p(\tilde{\bm x}_{n})^{\top} {\bm a}_{2}-\tilde{u}_{n}]^2 \\
        &\quad - [\bm p(\tilde{\bm x}_{n})^{\top} (t{\bm a}_{1}+(1-t) {\bm a}_2) - \tilde{u}_{n}]^2\\
        &= t(1-t)[\bm p(\tilde{\bm x}_{n})^{\top}({\bm a}_{1}-{\bm a}_{2})]^2 \geq 0.
        \end{aligned}
    \end{equation}
    Then, we multiply \eqref{equ: convex} by the weight \(w_{\tilde{\mathcal{M}}}({\bm x}_{*} - \tilde{\bm x}_{n})\) and sum the results over \(\tilde{\bm x}_{n} \in \tilde{N}({\bm x}_{*})\). Since \(w_{\tilde{\mathcal{M}}}({\bm x}_{*} - \tilde{\bm x}_{n}) > 0\) for \(\|{\bm x}_{*} - \tilde{\bm x}_{n}\| < \rho_{\tilde{\mathcal{M}}}(\bm x_{*})\), it follows that
    \begin{equation}
        \mathcal{L}_{\tilde{\mathcal{M}}^{f}}(t\bm a_1 + (1-t)\bm a_2) \leq t \mathcal{L}_{\tilde{\mathcal{M}}^{f}}(\bm a_1) + (1-t) \mathcal{L}_{\tilde{\mathcal{M}}^{f}}(\bm a_2),
    \end{equation} 
    which completes the proof.
\end{proof}
Given a matrix ${\bm P} := [{\bm p}({\tilde{\bm x}}_{1}),\dots,{\bm p}({\tilde{\bm x}}_{|\tilde{N}(\bm x_{*})|})]\in \mathbb{R}^{|\mathcal{P}| \times |\tilde{N}(\bm x_{*})|}$, the following Theorem~\ref{thm: solvable} demonstrates that Problem~\eqref{equ: min} admits a unique solution under certain conditions.
\begin{theorem}
    \label{thm: solvable}
    If $\mathrm{rank} (\bm P) = |\mathcal{P}|$, Problem~\eqref{equ: min} admits a unique solution.
\end{theorem}
\begin{proof}
    By Lemma~\ref{lem: convex},  it suffices to prove  $\mathcal{L}_{\tilde{\mathcal{M}}^{f}}(\bm a)$ has a unique stationary point. Our task now is to calculate $\nabla_{\bm a} \mathcal{L}_{\tilde{\mathcal{M}}^{f}}(\bm a) = 0$, which is equivalent to the following linear system
    \begin{equation}
        \label{equ: stationary point}
        {\bm P}_{1}({\bm x}_{*}) {\bm a} = {\bm P}_{2}({\bm x}_{*}) \tilde{\bm u}(\bm x_{*}).
    \end{equation}
    Here, ${\bm P}_{1}({\bm x}_{*}) \in \mathbb{R}^{|\mathcal{P}|\times|\mathcal{P}|}$, ${\bm P}_{2}({\bm x}_{*}) \in \mathbb{R}^{|\mathcal{P}|\times|\tilde{N}(\bm x_{*})|}$, and $\tilde{\bm u}({\bm x}_{*}) \in \mathbb{R}^{|\tilde{N}(\bm x_{*})|}$ are given by
    \begin{align}
        &({\bm P}_{1}({\bm x}_{*}))_{j, k}\ =\ \sum_{n=1}^{|\tilde{N}(\bm x_{*})|} w_{\tilde{\mathcal{M}}}(\bm x_{*} - \tilde{\bm x}_{n}) p_j(\tilde{\bm x}_n) p_k(\tilde{\bm x}_n),\ \ \ \ \ \ \ j, k = 1, \dots, |\mathcal{P}|,\\
        &({\bm P}_{2}({\bm x}_{*}))_{k, n}\ =\ w_{\tilde{\mathcal{M}}}(\bm x_{*} - \tilde{\bm x}_{n})p_k(\tilde{\bm x}_{n}),\ \ k = 1, \dots, |\mathcal{P}|, n = 1, \dots, |\tilde{N}(\bm x_{*})|,\\
        &\ \ (\tilde{\bm u}(\bm x_{*}))_{n} \ \ \ =\ \tilde{u}_n,\ \ \ \ \ \ \ \ \ \ \ \ \ \ \ \ \ \ \ \ \ \ \ \ \ \ \ \ \ \ \ \ \ \ \ \ \ \ \ \ \ \ \ \ \ n = 1, \dots, |\tilde{N}(\bm x_{*})|.
    \end{align}
    Defining a positive-definite matrix $\bm D : = \texttt{diag} (w_{\tilde{\mathcal{M}}}(\bm x_{*} - \tilde{\bm x}_{1}), \dots, w_{\tilde{\mathcal{M}}}(\bm x_{*} - \tilde{\bm x}_{|\tilde{N}(\bm x_{*})|})) \in \mathbb{R}^{|\tilde{N}(\bm x_{*})| \times |\tilde{N}(\bm x_{*})|}$, we have
    \begin{equation}
        \mathrm{rank}(\bm P_{1} (\bm x_{*})) = \mathrm{rank}({\bm P}{\bm D}{\bm P}^{\top})
        = \mathrm{rank}(\sqrt{\bm D} {\bm P}^{\top}) = \mathrm{rank}({\bm P}) = |\mathcal{P}|, 
    \end{equation}
    which indicates that Equation~\eqref{equ: stationary point} has a unique solution $\bm a = {\bm P}_{1}({\bm x}_{*})^{-1}{\bm P}_{2}({\bm x}_{*}) \tilde{\bm u}({\bm x}_{*})$. The proof is completed.
\end{proof}
Theorem~\ref{thm: solvable} establishes the feasibility of designing $\mathcal{I}_{{\tilde{\mathcal{M}}}^{f}}^{{\mathcal{M}}^{f}}$ using the MLS method. It suffices to gradually increase $\rho_{\mathcal{M}}(\bm x_{*})$ at $\bm x_{*}$ such that $\mathrm{rank} (\bm P) = |\mathcal{P}|$, which is readily achievable. Then, let us define an operator ${\bm R}_{\tilde{N}(\bm x_{*})}$ that extracts $\tilde{\bm u}(\bm x_{*})$ from the local overlapping vector ${\bm u}_{\tilde{\mathcal{M}}^{f}, p}^{\delta}$, we have the MLS approximation
\begin{equation}
    \begin{aligned}
    &u_{*} = {\bm p}(\bm x_{*})^{\top} \bm a = {\bm p}(\bm x_{*})^{\top} {\bm P}_{1}({\bm x}_{*})^{-1}{\bm P}_{2}({\bm x}_{*}) \tilde{\bm u}({\bm x}_{*}) \\
    &\quad = {\bm p}(\bm x_{*})^{\top} {\bm P}_{1}({\bm x}_{*})^{-1}{\bm P}_{2}({\bm x}_{*}) {\bm R}_{\tilde{N}(\bm x_{*})} {\bm u}_{\tilde{\mathcal{M}}^{f}, p}^{\delta}.
    \end{aligned}
\end{equation}
Then, let $(*)$ range over $1, 2, \dots, N: = |U_{\mathcal{M}^{f}, p}|$, the operator $\mathcal{I}_{\tilde{\mathcal{M}}^{f}}^{\mathcal{M}^{f}}$ is given by
\begin{equation}
    \label{equ: interpolation f}
    \mathcal{I}_{\tilde{\mathcal{M}}^{f}}^{\mathcal{M}^{f}} = \sum_{p=1}^{np} ({\bm R}_{\mathcal{M}^{f}, p}^{0})^{\top}
    \begin{bmatrix}
        {\bm p}(\bm x_{1})^{\top} {\bm P}_{1}({\bm x}_{1})^{-1}{\bm P}_{2}({\bm x}_{1})  {\bm R}_{\tilde{N}(\bm x_{1})} \\
        {\bm p}(\bm x_{2})^{\top} {\bm P}_{1}({\bm x}_{2})^{-1}{\bm P}_{2}({\bm x}_{2})  {\bm R}_{\tilde{N}(\bm x_{2})} \\
        \vdots \\
        {\bm p}(\bm x_{N})^{\top} {\bm P}_{1}({\bm x}_{N})^{-1}{\bm P}_{2}({\bm x}_{N})  {\bm R}_{\tilde{N}(\bm x_{N})}
    \end{bmatrix}
    {\bm R}_{\tilde{\mathcal{M}}^{f}, p}^{\delta}.
\end{equation}
From Equation~\eqref{equ: interpolation f}, the operator $\mathcal{I}_{\tilde{\mathcal{M}}}^{\mathcal{M}}$ is immediately obtained by using Equation~\eqref{equ: interpolation split}.

\subsection{Algorithm workflow}
We incorporate the above coarsening Algorithm~\ref{alg: coarsening} and the interpolation operator $\mathcal{I}_{\tilde{\mathcal{M}}}^{\mathcal{M}}$ into the V-cycle multilevel smoothed Schwarz framework, then propose the following Algorithm~\ref{alg: workflow}.

\begin{algorithm} [H]
\caption{The parallel non-nested multilevel smoothed Schwarz preconditioner}
\label{alg: workflow}
\begin{algorithmic}[1]
\State \textbf{Input:} Finest mesh $\mathcal{M}_{0}$ and a vector $ {\bm b}_{\mathcal{M}_{0}}$ on $\mathcal{M}_{0}$.
\State Construct non-nested mesh hierarchy $\mathcal{M}_{i},\ i=0, \dots, L-1$ via Algorithm~\ref{alg: coarsening}, and provide the corresponding Jacobian matrix $\mathcal{J}_{\mathcal{M}_{i}}$.
\For{\(i = 0\) \textbf{to} \( L-1 \)}
    \State \( {\bm z}_{\mathcal{M}_{i}}^{(0)} = 0 \)
    \For{\(n = 0\) \textbf{to} \textit{its}}
        \State \( \bm{z}_{\mathcal{M}_{i}}^{(n+1)} = \bm{z}_{\mathcal{M}_{i}}^{(n)} + B_{\mathcal{M}_{i}}^{-1} (\bm{r}_{\mathcal{M}_{i}} - \mathcal{J}_{\mathcal{M}_{i}} \bm{z}_{\mathcal{M}_{i}}^{(n)}) \)
    \EndFor
    \State \( \bm{z}_{\mathcal{M}_{i}} = \bm{z}_{\mathcal{M}_{i}}^{(\textit{its})} \)
    \If{\(i < L-1 \)}
        \State \( \bm{r}_{\mathcal{M}_{i}+1} = \mathcal{I}_{\mathcal{M}_{i}}^{\mathcal{M}_{i+1}}(\bm {r}_{\mathcal{M}_{i}} - \mathcal{J}_{\mathcal{M}_{i}} \bm{z}_{\mathcal{M}_{i}}) \)
    \EndIf
\EndFor
\For{\(i = L-2 \) \textbf{to} \( 0 \)}
    \State \( \bm{z}_{\mathcal{M}_{i}}^{(0)} = \bm{z}_{\mathcal{M}_{i}} + \mathcal{I}_{\mathcal{M}_{i+1}}^{\mathcal{M}_{i}} \bm{z}_{\mathcal{M}_{i+1}} \)
    \For{\(n = 0\) \textbf{to} \textit{its}}
        \State \( \bm{z}_{\mathcal{M}_{i}}^{(n+1)} = \bm {z}_{\mathcal{M}_{i}}^{(n)} + B_{\mathcal{M}_{i}}^{-1} (\bm {r}_{\mathcal{M}_{i}} - \mathcal{J}_{\mathcal{M}_{i}} \bm{z}_{\mathcal{M}_{i}}^{(n)}) \)
    \EndFor
    \State \( \bm {z}_{\mathcal{M}_{i}} = \bm {z}_{\mathcal{M}_{i}}^{(\textit{its})} \)
\EndFor
\State \textbf{Output:} A vector \( {\bm z}_{\mathcal{M}_{0}} \) to the outer solver.
\end{algorithmic}
\end{algorithm}

\ 

\section{Experiment results}
\label{sec: experiments}
To evaluate the performance of the newly proposed parallel preconditioner, we conduct numerical experiments on a series of 2D and 3D linear parametric PDEs. In this work, we fix both the number of pre-smoothing and post-smoothing steps $its$ to $3$ in Algorithm~\ref{alg: workflow}. All simulations are performed on a supercomputer with multiple nodes, each equipped with two AMD EPYC 7452 32-Core CPUs and 256 GB of local memory. The nodes are interconnected via an InfiniBand high-performance network. In all experiments, we employ the generalized minimal residual algorithm (GMRES) \citep{saad1986gmres} for solving the linear systems iteratively. In this section, we primarily focus on: (1) the number of GMRES iterations to convergence, (2) the computing time, and (3) the parallel efficiency, using the one-level case (i.e., RAS with overlap $\delta=1$ in PETSc \citep{petsc-user-ref}, which is one of the state-of-the-art parallel preconditioners) as a baseline for comparison.


\ 

\subsection{Comparison under multiple parameters}
We first investigate the use of the same coarse mesh to accelerate the solution of problems under multiple parameters, to evaluate the overall the overall benefits of the two-level preconditioner compared to the one-level preconditioner. The computing time for a two-level experiment includes both the time required for the coarsening process and the time to solve all parametric problems. In this subsection, the restart value of GMRES is fixed at 50, and the absolute and relative tolerances for the GMRES iterative solver are set to $10^{-12}$. Here, the RAS preconditioner with $\delta = 1$ is used as the smoother on each mesh level, and LU factorization in SuperLU \citep{li2005overview} is used as the subdomain solver in the RAS method. Experiments are conducted using $np = 3, 6, 9$, and $12$ processors.

\subsubsection{2D Case}
We start by solving the convection-diffusion equations with a zero Dirichlet boundary condition in a 2D Lake-Superior-like domain, as described in refs. \citep{shewchuk2002delaunay, persson2005mesh},
\begin{equation}
    \left \{
    \begin{aligned}
    \nabla \cdot (\beta {\bm v}_{0} u + \nabla u) &= f, \ \texttt{in}\ \Omega, \\
    u &= 0, \ \texttt{on}\ \partial\Omega, 
    \end{aligned}
    \right.
\end{equation}
where we set $f = 1$ and velocity field ${\bm v}_{0} = (1, 0)$ as ref. \citep{pan2022agglomeration}. The computational domain $\Omega$ is shown in Figure~\ref{fig: superior}(a). Figures~\ref{fig: superior}(b) and~\ref{fig: superior}(c) show the original fine mesh $\mathcal{M}_{0}$ and the geometry-preserving coarse mesh $\mathcal{M}_{1}$ obtained after one coarsening process using Algorithm~\ref{alg: coarsening}, respectively. Graphically, there is no observable difference between the boundaries in $\mathcal{M}_{0}$ and $\mathcal{M}_{1}$. In this case, the first-order finite volume method is utilized to discretize this problem, and DoFs are distributed as shown in Figure~\ref{multiphysics}(a). We then investigate the performance results of the newly proposed two-level preconditioner by simultaneously simulating the cases with parameters $\beta = 1, 10, 100$.

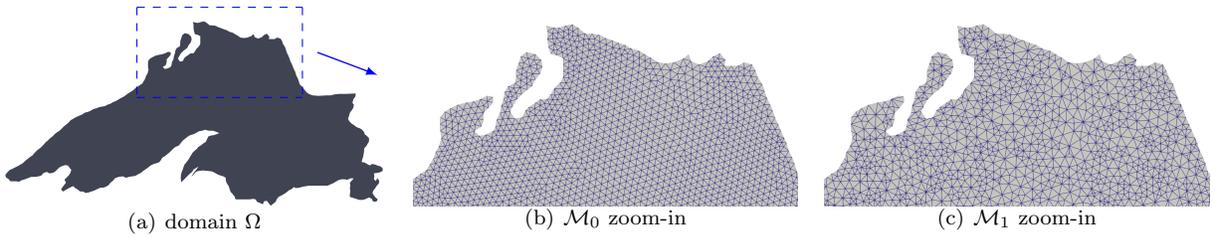
\begin{figure}[H]
    \centering
    \subfigure[domain $\Omega$]{
		\begin{minipage}[]{0.31\linewidth}
		  \centering
            \begin{tikzpicture}
                \node[anchor=south west,inner sep=0] (image) at (0,0.5) {\includegraphics[width=1.0\textwidth]{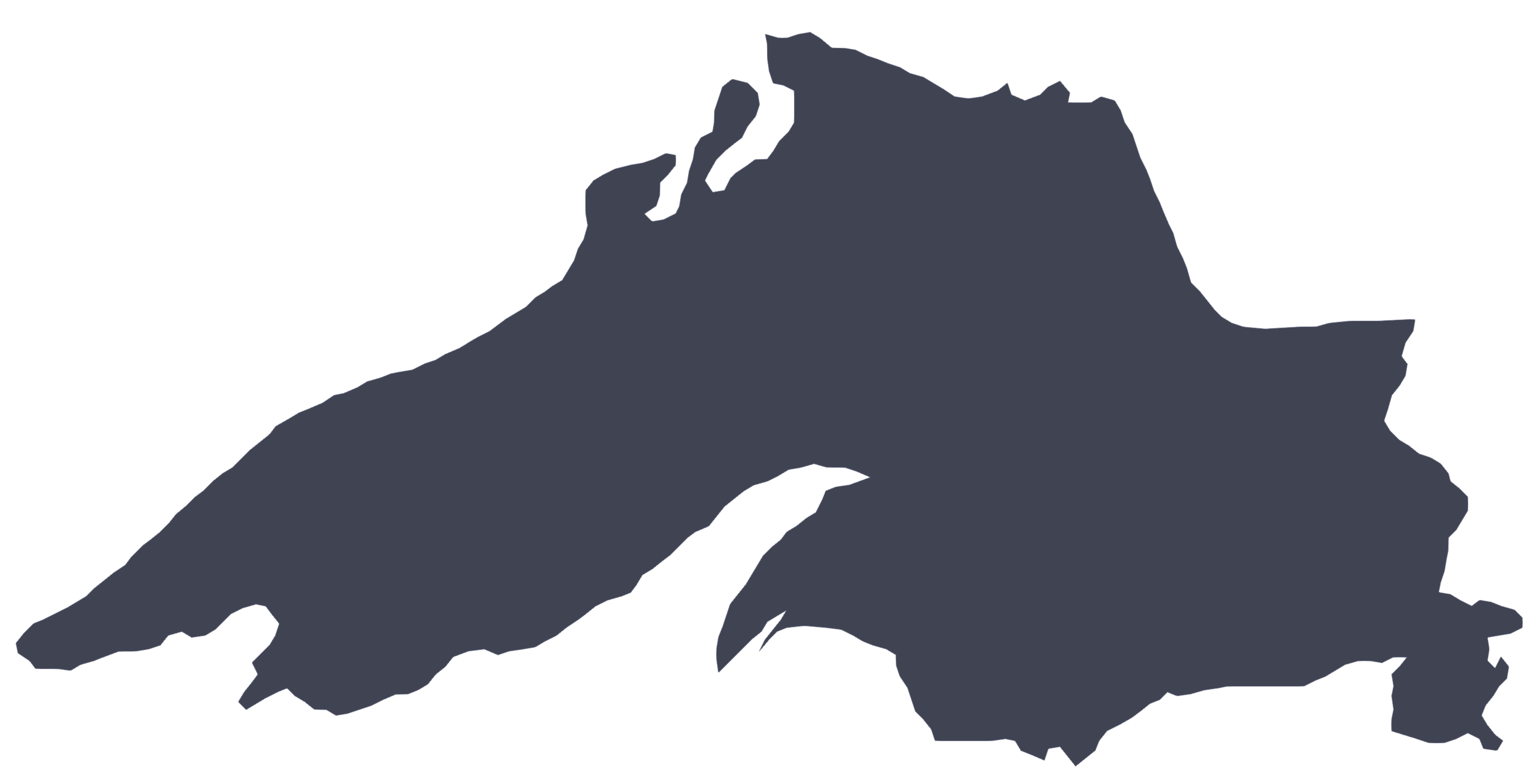}};  
                \draw[dashed, blue] (1.8, 2.7+0.5) -- (4.0, 2.7+0.5);
                \draw[dashed, blue] (1.8, 1.5+0.5) -- (4.0, 1.5+0.5);
                \draw[dashed, blue] (1.8, 2.7+0.5) -- (1.8, 1.5+0.5);
                \draw[dashed, blue] (4.0, 2.7+0.5) -- (4.0, 1.5+0.5);
                \draw[->, >=latex, blue, line width=0.5pt] (4.2, 2.1+0.5) -- (5.0, 1.8+0.5);

            \end{tikzpicture}
		\end{minipage}
    } \subfigure[$\mathcal{M}_{0}$ zoom-in]{
		\begin{minipage}[]{0.31\linewidth}
		  \centering
            \begin{tikzpicture}
                \node[anchor=south west,inner sep=0] (image) at (0,0.5) {\includegraphics[width=1.0\textwidth]{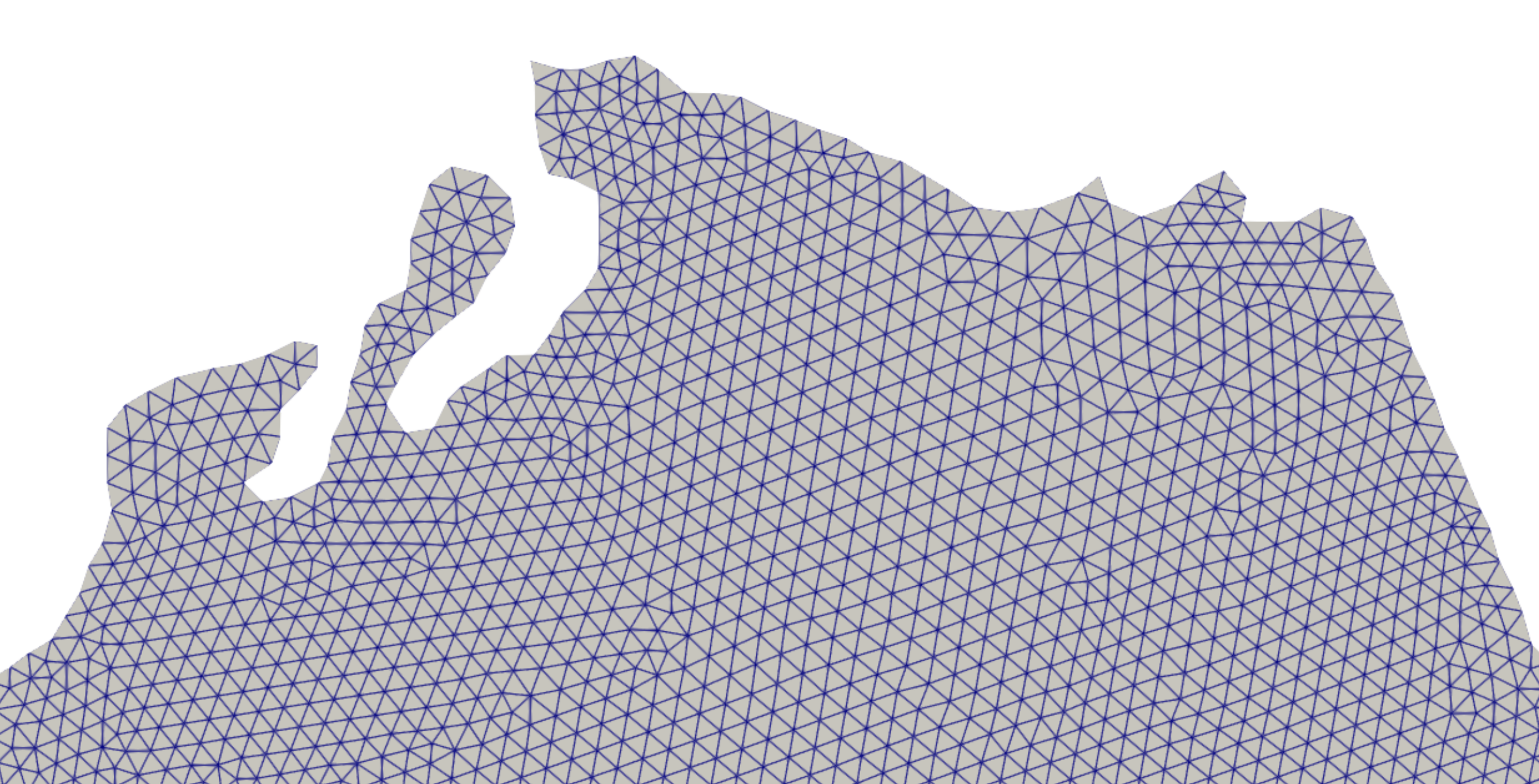}};  
            \end{tikzpicture}
		\end{minipage}
    } \subfigure[$\mathcal{M}_{1}$ zoom-in]{
		\begin{minipage}[]{0.31\linewidth}
		  \centering
            \begin{tikzpicture}
                \node[anchor=south west,inner sep=0] (image) at (0,0.5) {\includegraphics[width=1.0\textwidth]{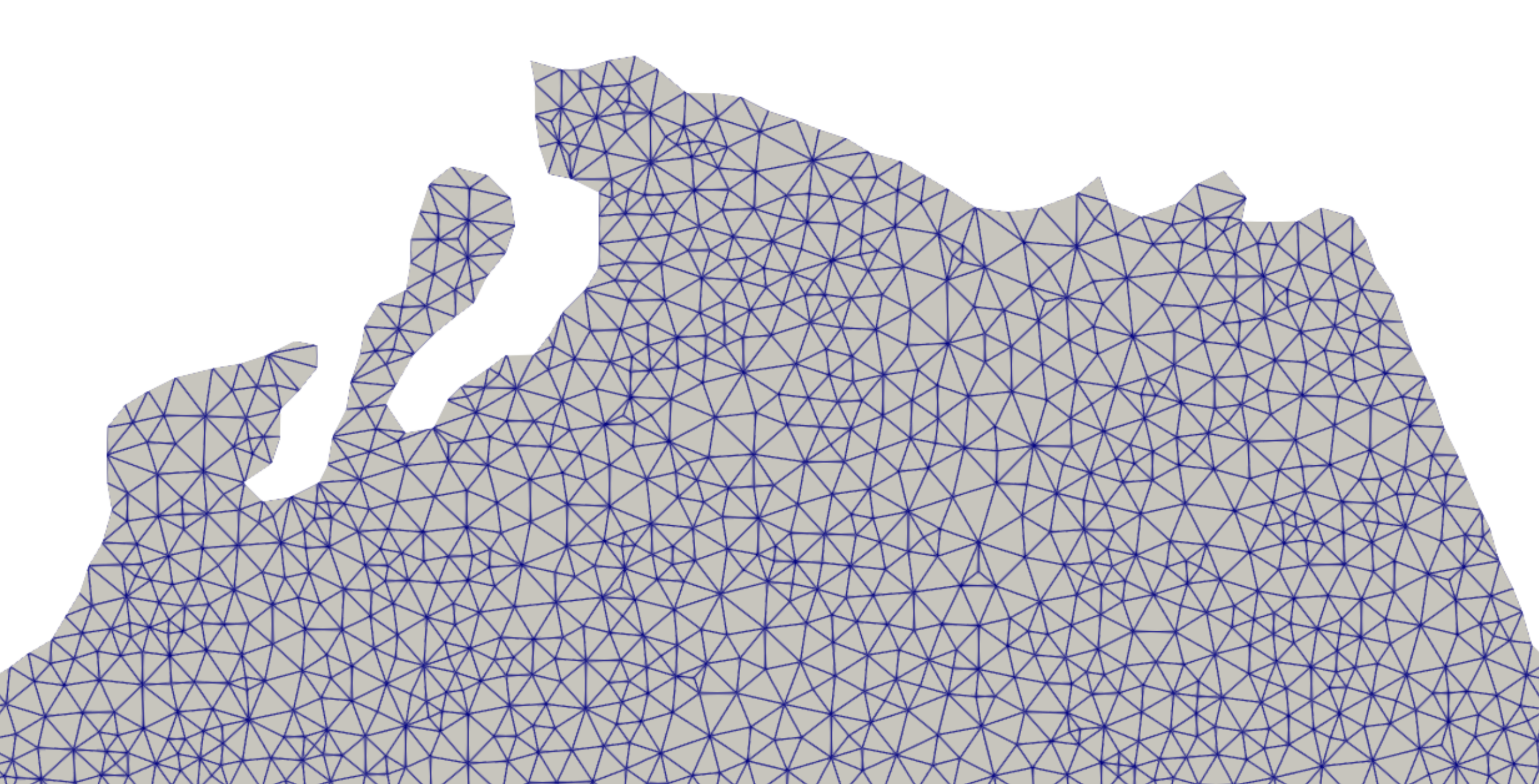}};  
            \end{tikzpicture}
		\end{minipage}
    }
    \caption{(a) The computational domain $\Omega$; (b) $\mathcal{M}_{0}$ with 10,336 elements and 5,393 vertices; (c) $\mathcal{M}_{1}$ with 7,766 elements and 4,108 vertices;}
    \label{fig: superior}
\end{figure}

Table~\ref{tab: convection-diffusion-multiple} presents the total GMRES iterations, computing time, and parallel efficiency for both one-level and two-level cases. The two-level preconditioner achieves a parallel efficiency of $75.53\%$, while the one-level baseline demonstrates $48.13\%$ parallel efficiency with $12$ processors, illustrating the superior parallel efficiency of the new two-level preconditioner. Furthermore, the two-level preconditioner exhibits a significant advantage in terms of GMRES iterations and total computing time at 12 processors.

\begin{table}[H]
    \centering
    \caption{For $\beta = 1, 10, 100$, comparison of total GMRES iterations (denoted as ``${\hbox{N.It}}_{total}$''), computing time (``Time(s)''), and parallel efficiency (``Eff $\%$'') for the proposed multilevel Schwarz preconditioners, specifically for one-level and two-level cases. Here, ``Eff $\%$ = $(m T_m)/(n T_n)$" denotes the parallel efficiency when scaling from $m$ processors to $n$ processors ($m \leq n$). }
    \setlength{\tabcolsep}{8.0mm}{
    \center
    \begin{tabular} {|c|l|c|c|c|c|}
        \hline
            &$np$ &3 &6 &9 &12 \\\cline{1-6}
        \multirow{3}{*}{One-level} &${\hbox{N.It}}_{total}$ &38 &47 &52 &56 \\\cline{2-6}
        &Time &1.211  &0.796 &0.687 &0.629 \\\cline{2-6}
        &Eff($\%$) &100\% &76.07\% &58.76\% &48.13\%
        \\\cline{1-6}
        \multirow{3}{*}{Two-level} &${\hbox{N.It}}_{total}$ &21 &28 &29 &33 \\\cline{2-6}
		&Time &1.571  &0.952 &0.673 &0.520 \\\cline{2-6}
        &Eff($\%$) &100\% &82.51\% &77.81\% &75.53\%
        \\\hline
    \end{tabular}
    }
    \label{tab: convection-diffusion-multiple}
\end{table}

\ 

\subsubsection{3D Case}
We then consider the linear elasticity equations \citep{kong2016highly, howell2009applied, atallah2021shifted} to calculate the displacement $\bm u = (u, v, w)$ of a 3D drive wheel with a geometry as described in ref. \citep{wheelmodel}. In this problem, we define the fixed bottom side surface of the model as $\Gamma_W$, while the remaining surface is defined as $\Gamma_S = \partial \Omega / \Gamma_W$. The governing equations are:
\begin{equation}
    \left \{
    \begin{aligned}
    -\nabla \cdot  \bm \sigma &=  \bm 0,\ \ \ \ \ \texttt{in}\ \Omega,\\
    \bm \sigma \cdot \bm n &= -p \bm n,\ \texttt{on}\ \Gamma_S,\\
    \bm u &= \bm 0,\ \ \ \ \ \texttt{on}\ \Gamma_W.
    \end{aligned}
    \right.
\end{equation}
Here $p=$ 1 MPa is the constant surface pressure on $\Gamma_S$. The Cauchy stress tensor is formulated as follows: 
\begin{equation}
    \bm \epsilon = 1/2(\nabla \bm u + (\nabla \bm u)^{\top}),\ \ \ \bm \sigma = \lambda \texttt{trace}(\bm \epsilon) \bm I + 2 \mu \bm \epsilon, 
\end{equation}
where $\mu$ and $\lambda$ are coefficients defined by Young’s modulus $E$ and Poisson’s ratio $\nu$,
\begin{equation}
    \mu = \frac{E}{2(1+\nu)},\ \ \ \ \ \lambda = \frac{E\nu}{(1+\nu)(1-2\nu)}.
\end{equation}
In this experiment, Young’s modulus $E$ is $2.1\times10^3$ MPa and Poisson’s ratio $\nu = 0.43, 0.45$, and $0.47$. The finite element $[P_2(\mathcal{K})]^{3}$ is utilized to discretize this problem, and DoFs are distributed as shown in Figure~\ref{multiphysics}(b). Then, Figure~\ref{fig: drive-wheel} presents the computational results for the displacement $\bm u$, verifying that the numerical algorithm yields reasonable results. In the context of a two-level experiment, Figures~\ref{fig: wheel-coarsening}(a-b) and~\ref{fig: wheel-coarsening}(c-d) display the original fine mesh $\mathcal{M}_{0}$ and the geometry-preserving coarse mesh $\mathcal{M}_{1}$ obtained after one coarsening process using Algorithm~\ref{alg: coarsening}, respectively. Here, cross-sectional views of $\mathcal{M}_{0}$ and $\mathcal{M}_{1}$ clearly demonstrate that while the mesh becomes coarser, the domain boundaries are preserved.

\ 

\begin{figure}[H]
    \centering
    \subfigure[displacement $u$]{
		\begin{minipage}[]{0.23\linewidth}
		  \centering
		  \includegraphics[width=1.0\linewidth]{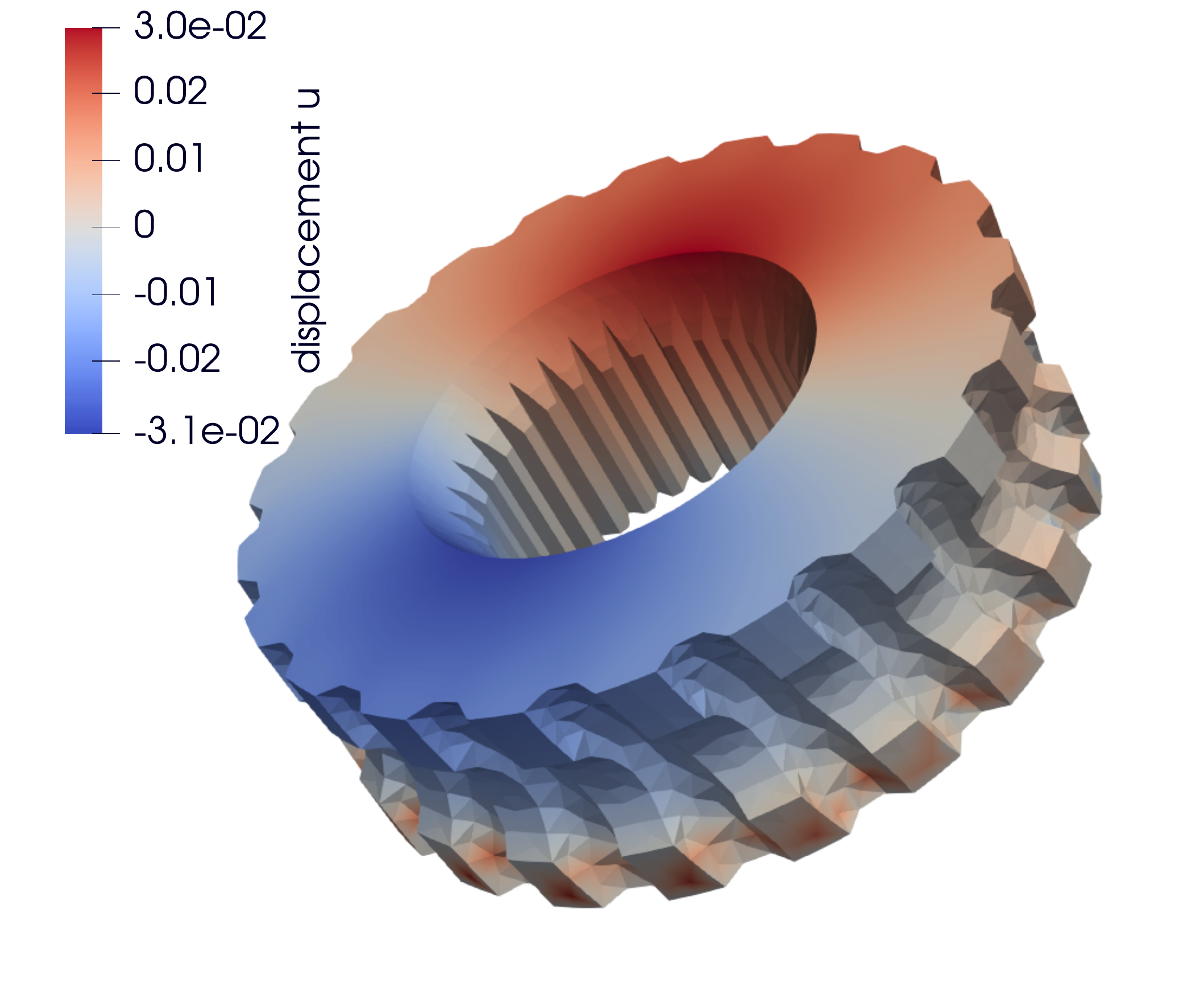}
		\end{minipage}
    } \subfigure[displacement $v$]{
		\begin{minipage}[]{0.23\linewidth}
		  \centering
		  \includegraphics[width=1.0\linewidth]{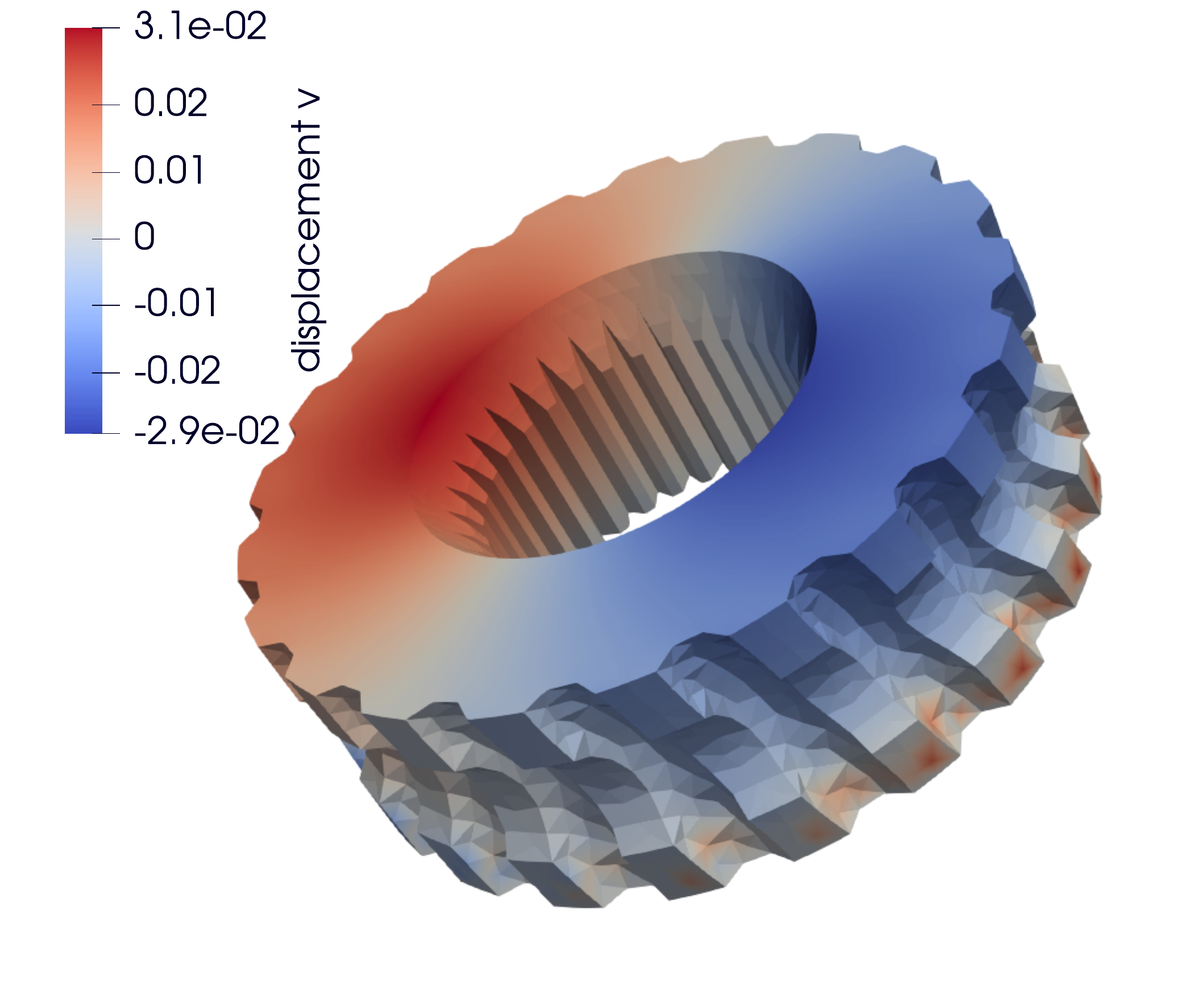}
		\end{minipage}
    } \subfigure[displacement $w$]{
		\begin{minipage}[]{0.23\linewidth}
		  \centering
		  \includegraphics[width=1.0\linewidth]{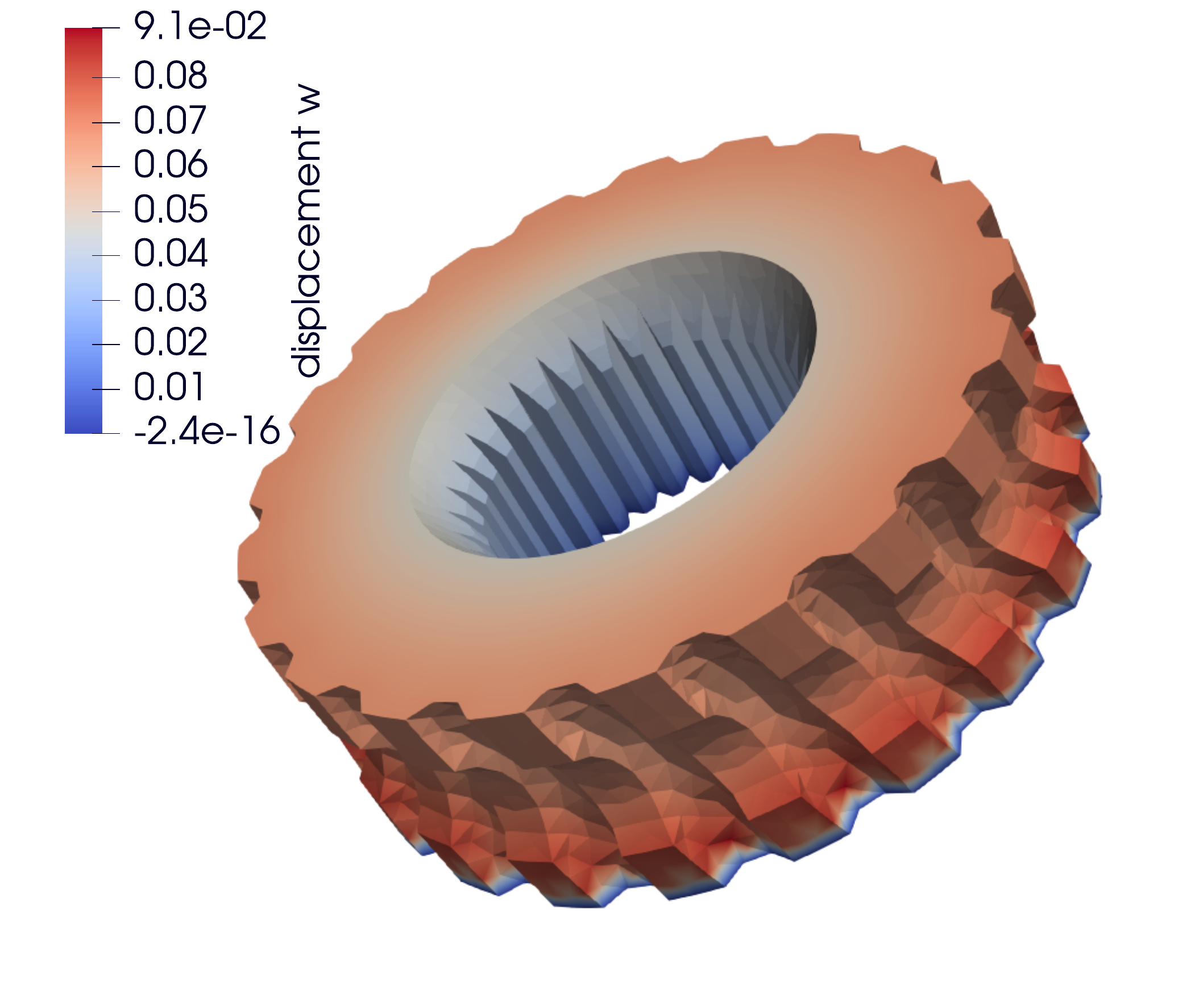}
		\end{minipage}
    } \subfigure[displacement magnitude]{
		\begin{minipage}[]{0.23\linewidth}
		  \centering
		  \includegraphics[width=1.0\linewidth]{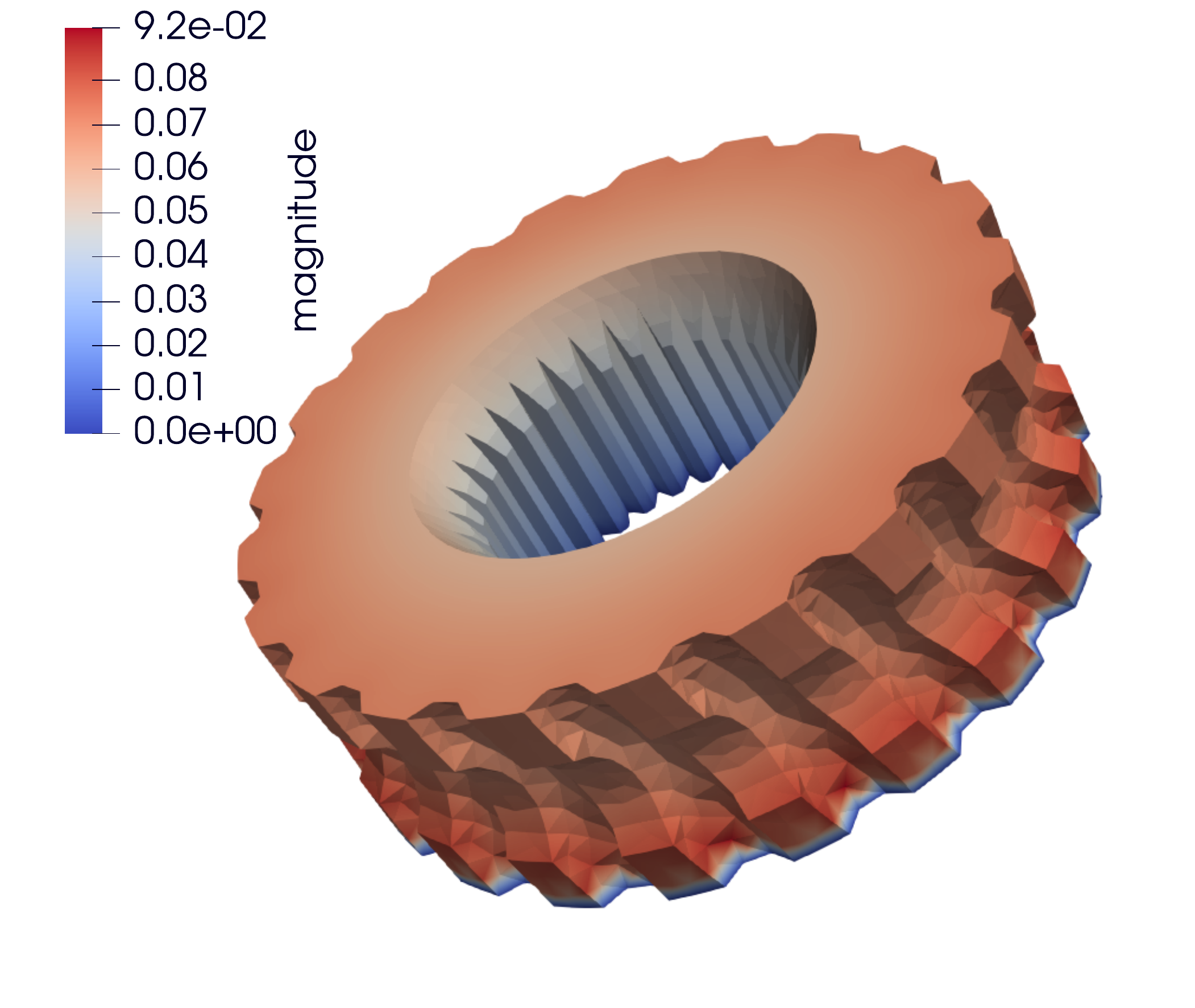}
		\end{minipage}
    } 
    \caption{The deformation of 3D drive wheel for $\nu = 0.45$ case; (a) The displacement $u$ in the x-direction; (b) The displacement $v$ in the y-direction; (c) The displacement $w$ in the z-direction; (d) The magnitude $|\bm u|$.}
    \label{fig: drive-wheel}
\end{figure}
\begin{figure}[H]
    \centering
    \subfigure[$\mathcal{M}_{0} (x < 0)$]{
		\begin{minipage}[]{0.23\linewidth}
		  \centering
		  \includegraphics[width=1.0\linewidth]{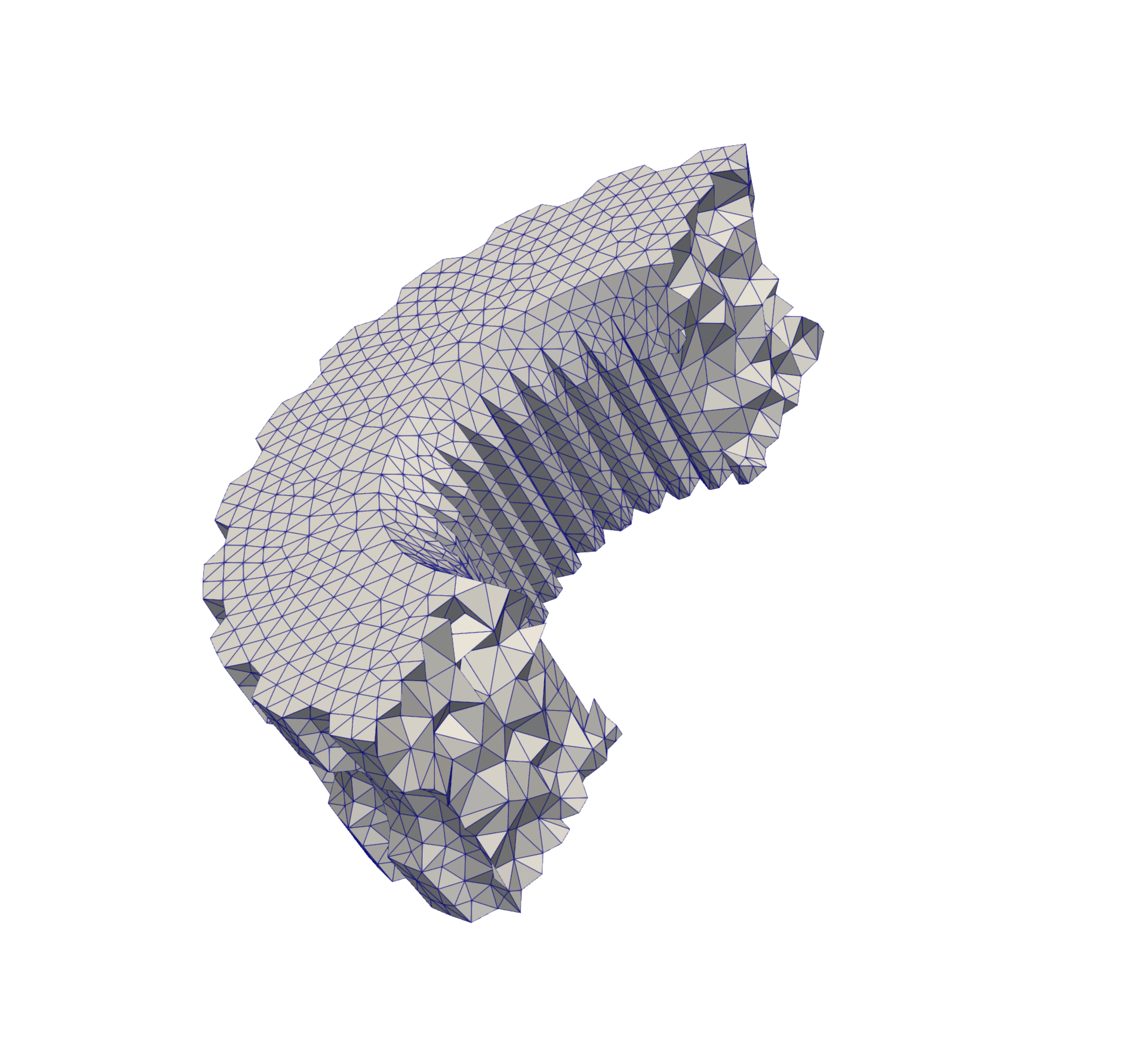}
		\end{minipage}
    } \subfigure[$\mathcal{M}_{0} (y < 0)$]{
		\begin{minipage}[]{0.23\linewidth}
		  \centering
		  \includegraphics[width=1.0\linewidth]{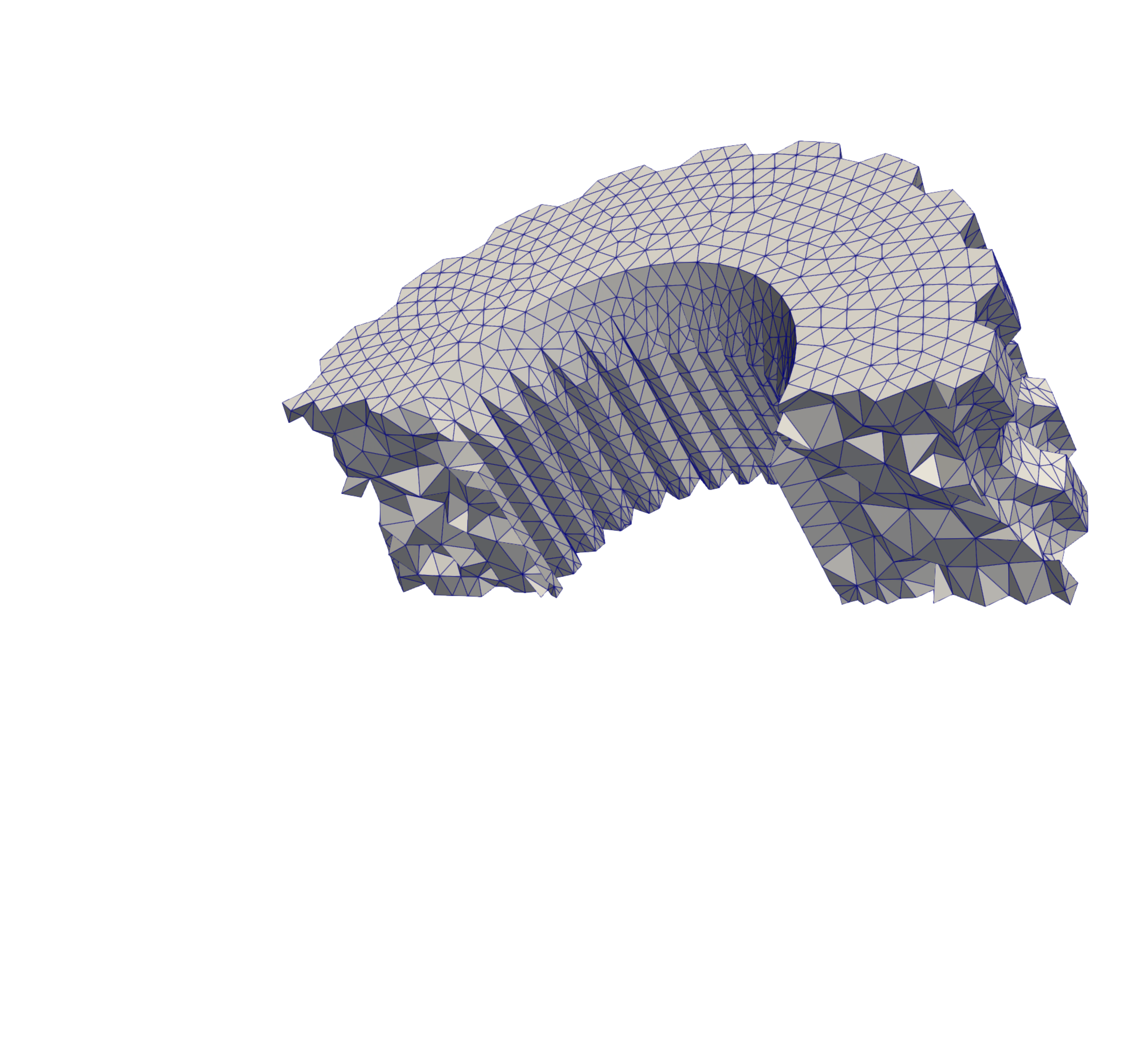}
		\end{minipage}
    } \subfigure[$\mathcal{M}_{1} (x < 0)$]{
		\begin{minipage}[]{0.23\linewidth}
		  \centering
		  \includegraphics[width=1.0\linewidth]{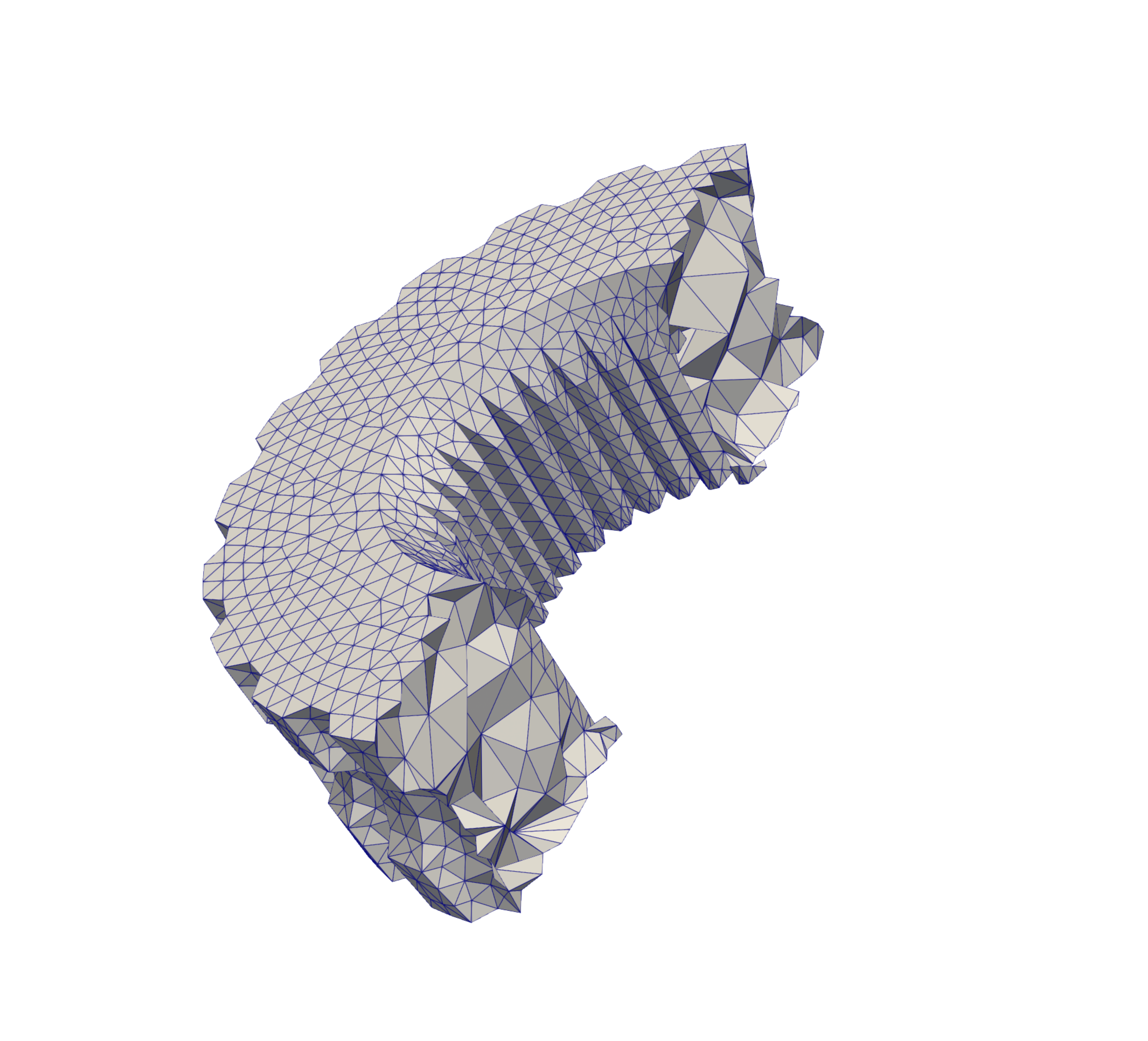}
		\end{minipage}
    } \subfigure[$\mathcal{M}_{1} (y < 0)$]{
		\begin{minipage}[]{0.23\linewidth}
		  \centering
		  \includegraphics[width=1.0\linewidth]{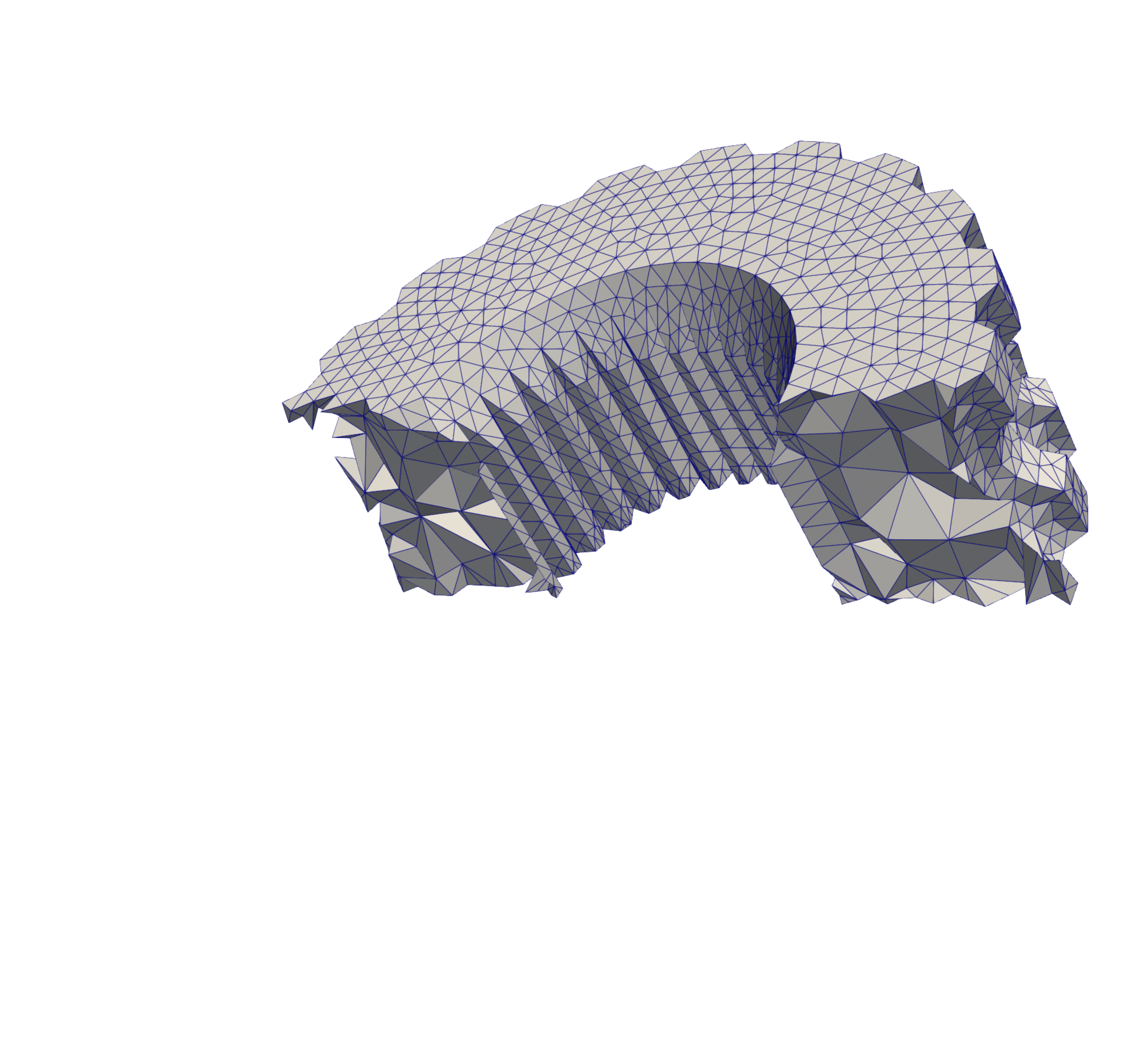}
		\end{minipage}
    } 
    \caption{Cross-sectional views of the meshes $\mathcal{M}_{0}$ and $\mathcal{M}_{1}$ in the x- and y-directions. (a-b) $\mathcal{M}_{0}$: 27,286 elements and 6,912 vertices; (c-d) $\mathcal{M}_{1}$: 19,999 elements and 5,675 vertices.}
    \label{fig: wheel-coarsening}
\end{figure}

For this 3D case, Table~\ref{tab: elasticity-multiple} presents the total iterations, computing time, and parallel efficiency for both one-level and two-level cases. Similar to the 2D case, the two-level preconditioner demonstrates a significant advantage in terms of GMRES iterations and total computing time when using 12 processors. Notably, the two-level preconditioner achieves superlinear parallel efficiency of $146.0\%$ with 12 processors, primarily due to the significant decrease in the number of GMRES iterations as the number of processors increases from 3 to 12.

\begin{table}[H]
    \centering
    \caption{For $\nu = 0.43, 0.45, 0.47$, comparison of total GMRES iterations (denoted as ${\hbox{N.It}}_{total}$), computing time (``Time (s)''), and parallel efficiency (Eff $\%$) for the proposed multilevel Schwarz preconditioners, specifically for one-level and two-level cases. Here, ``Eff $\%$ = $(m T_m)/(n T_n)$" denotes the parallel efficiency when scaling from $m$ processors to $n$ processors ($m \leq n$).}
    \setlength{\tabcolsep}{7.9mm}{
    \center
    \begin{tabular} {|c|l|c|c|c|c|}
        \hline
        &$np$ &3 &6 &9 &12 \\\cline{1-6}
        \multirow{3}{*}{One-level} &${\hbox{N.It}}_{total}$ &683 &596 &671 &453 \\\cline{2-6}
        &Time &555.35 &289.82 &249.56 &176.69 \\\cline{2-6}
        &Eff($\%$) &100\% &95.81\% &74.18\% &78.58\%
        \\\cline{1-6}
        \multirow{3}{*}{Two-level} &${\hbox{N.It}}_{total}$ &473 &279 &249 &147 \\\cline{2-6}
		&Time &873.81 &347.91 &233.07 &149.60 \\\cline{2-6}
        &Eff($\%$) &100\% &125.6\% &125.0\% &146.0\%
        \\\hline
    \end{tabular}
    }
    \label{tab: elasticity-multiple}
\end{table}

\ 

\subsection{Comparison of preconditioners with different levels}
To investigate the performance of the proposed preconditioner with different levels, we then consider the Stokes flow in a 2D/3D pipe governed by Stokes equations \citep{logg2012automated, atallah2022high},  
\begin{equation}
    \left \{
    \begin{aligned}
    -\nabla p + \nu \Delta  \bm u &= \bm 0,\  \texttt{in}\  \Omega, \\
    \nabla \cdot  \bm u &= 0,\  \texttt{in}\  \Omega.
    \end{aligned}
    \right.
    \label{equ: pipe-flow}
\end{equation}
In Equation~\eqref{equ: pipe-flow}, $ \bm u $ ($\bm u = (u, v)$ for 2D case and $\bm u = (u, v, w)$ for 3D case) denotes the velocity, and $p$ represents the pressure. We define the Cauchy stress tensor $ \bm \sigma = -p \bm I + \nu (\nabla  \bm u + (\nabla  \bm u)^{\top})$ with the viscosity coefficient $\nu \in \{ 1/500, 1/1000, 1/2000\}$. The boundary conditions on the inlets $\Gamma_{I}$, outlets $\Gamma_{O}$, and the wall $\Gamma_{W}$ are specified separately for the 2D and 3D cases. 

In this experiment, we mainly focus on the cases where level $L = 1, 2, 3$. For each parameter $\nu$, the coarsening time and simulation time using one-level, two-level, and three-level preconditioners are evaluated separately. In this subsection, the restart value of GMRES is fixed at 50 for the 2D case and 200 for the 3D case. The absolute tolerance and relative tolerance for the GMRES iterative solver are set to $10^{-12}$. Here, the RAS preconditioner with $\delta = 1$ is used as the smoother on each mesh level. For higher efficiency, incomplete LU factorization (ILU(p)) in PETSc \citep{petsc-user-ref} is used as the approximate subdomain solver. All experiments are conducted using $np$ = 3, 6, 9, and 12 processors.

\subsubsection{2D Case}
We first consider the 2D problem, in which $\Omega$ is a 2D pipe as shown in Figure~\ref{fig: Stokes-Flow}(a). In this case, we focus on the following boundary conditions, 
\begin{equation}
    \left \{
    \begin{aligned}
    &\bm u = (4(1-y)y,\ 0),\  \texttt{on}\ \ \Gamma_{I},\\
    &\bm u =  \bm 0,\ \ \ \ \ \ \ \ \ \ \ \ \ \ \ \ \   \texttt{on}\ \ \Gamma_{W},\\
    &\bm \sigma \cdot \bm n = \bm 0,\ \ \ \ \ \ \ \ \ \ \ \ \ \texttt{on}\ \ \Gamma_{O}.
    \end{aligned}
    \right.
\end{equation}
The 2D Taylor-Hood element $P_{2}-P_{1}$ \citep{taylor1973numerical} is utilized to discretize this problem, and the DoFs are distributed as shown in Figure~\ref{multiphysics}(c). Figures~\ref{fig: Stokes-Flow}(b-c) display the computational results for the case $\nu = 1/1000$. For the multilevel experiments, we begin with the finest mesh $\mathcal{M}_{0}$ shown in Figure~\ref{fig: pipe-coarsening}(a) and apply Algorithm~\ref{alg: coarsening} to construct the coarse meshes $\mathcal{M}_{1}$ and $\mathcal{M}_{2}$ as depicted in Figures~\ref{fig: pipe-coarsening}(b-c). As shown, meshes $\mathcal{M}_{0}$, $\mathcal{M}_{1}$, and $\mathcal{M}_{2}$ share the same geometric boundary.
\begin{figure}[H]
    \centering
    \subfigure[domain $\Omega$]{
		\begin{minipage}[]{0.31\linewidth}
		  \centering
            \begin{tikzpicture}
                \node[anchor=south west,inner sep=0] (image) at (0,0) {\includegraphics[width=1.0\textwidth]{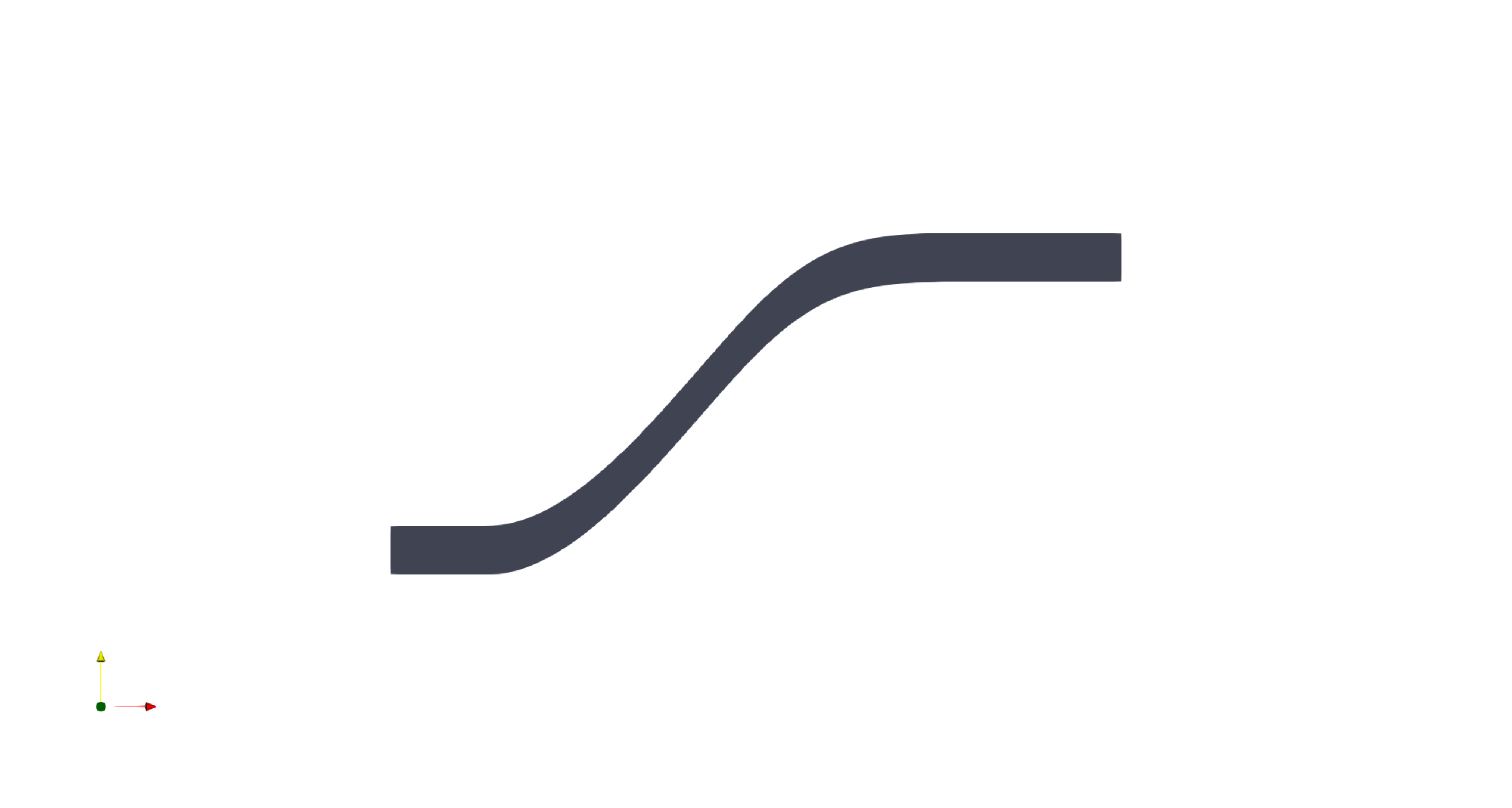}};
                \node at (0.4,0.8) {$\Gamma_{I}$};
                \node at (1.65,1.6) {$\Gamma_{W}$};
                \node at (2.8,1.4) {$\Gamma_{W}$};
                \node at (4.8,1.9) {$\Gamma_{O}$};

                \draw[dashed, blue] (0.1, 1.25) -- (2.15, 1.25);
                \draw[dashed, blue] (0.1, 0.2) -- (2.15, 0.2);
                \draw[dashed, blue] (0.1, 0.2) -- (0.1, 1.25);
                \draw[dashed, blue] (2.15, 0.2) -- (2.15, 1.25);
            \end{tikzpicture}
		\end{minipage}
    } \subfigure[pressure]{
		\begin{minipage}[]{0.31\linewidth}
		  \centering
		  \includegraphics[width=1.0\linewidth]{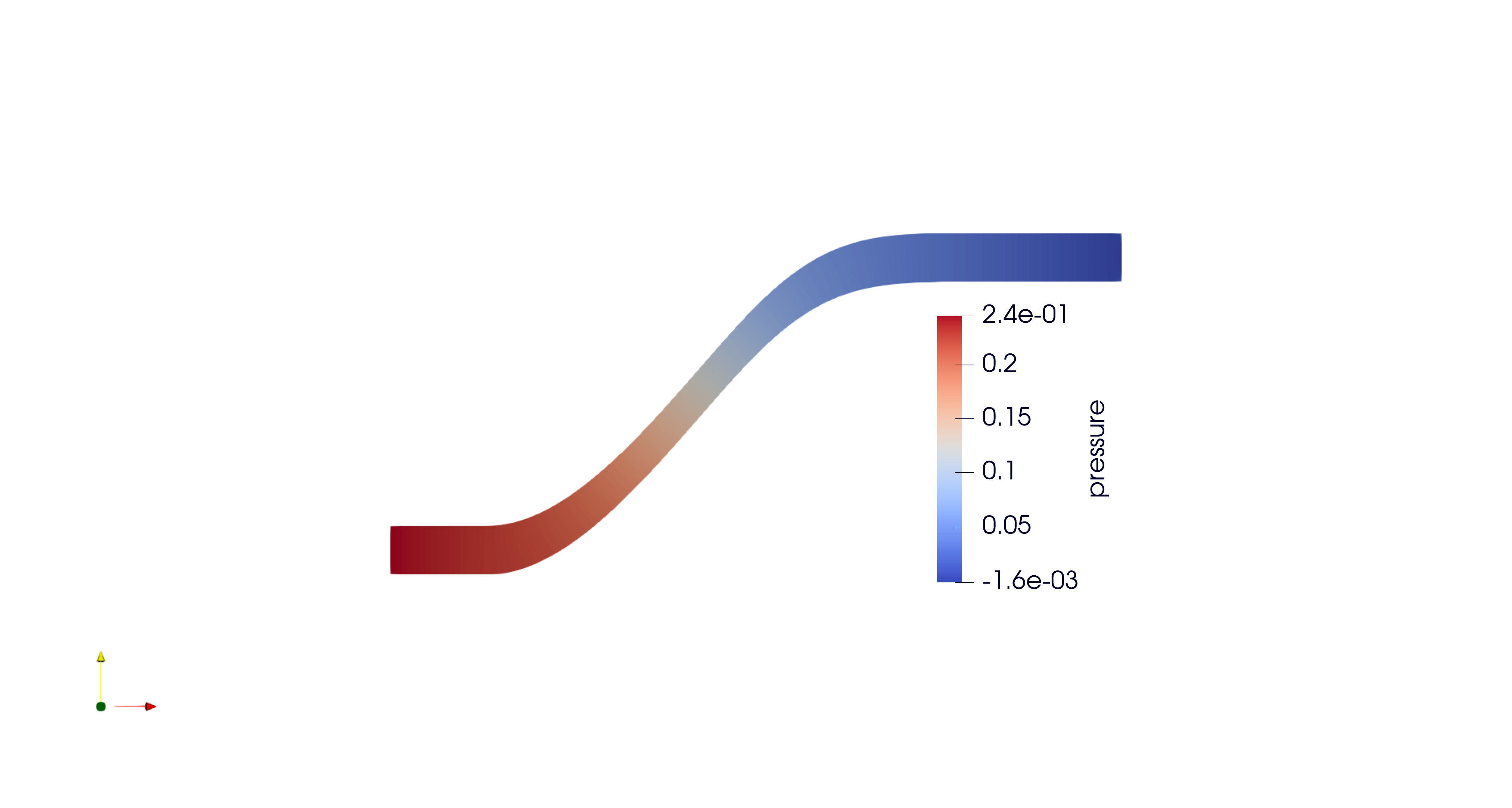}
		\end{minipage}
    } \subfigure[velocity magnitude $|{\bm u}|$]{
		\begin{minipage}[]{0.31\linewidth}
		  \centering
		  \includegraphics[width=1.0\linewidth]{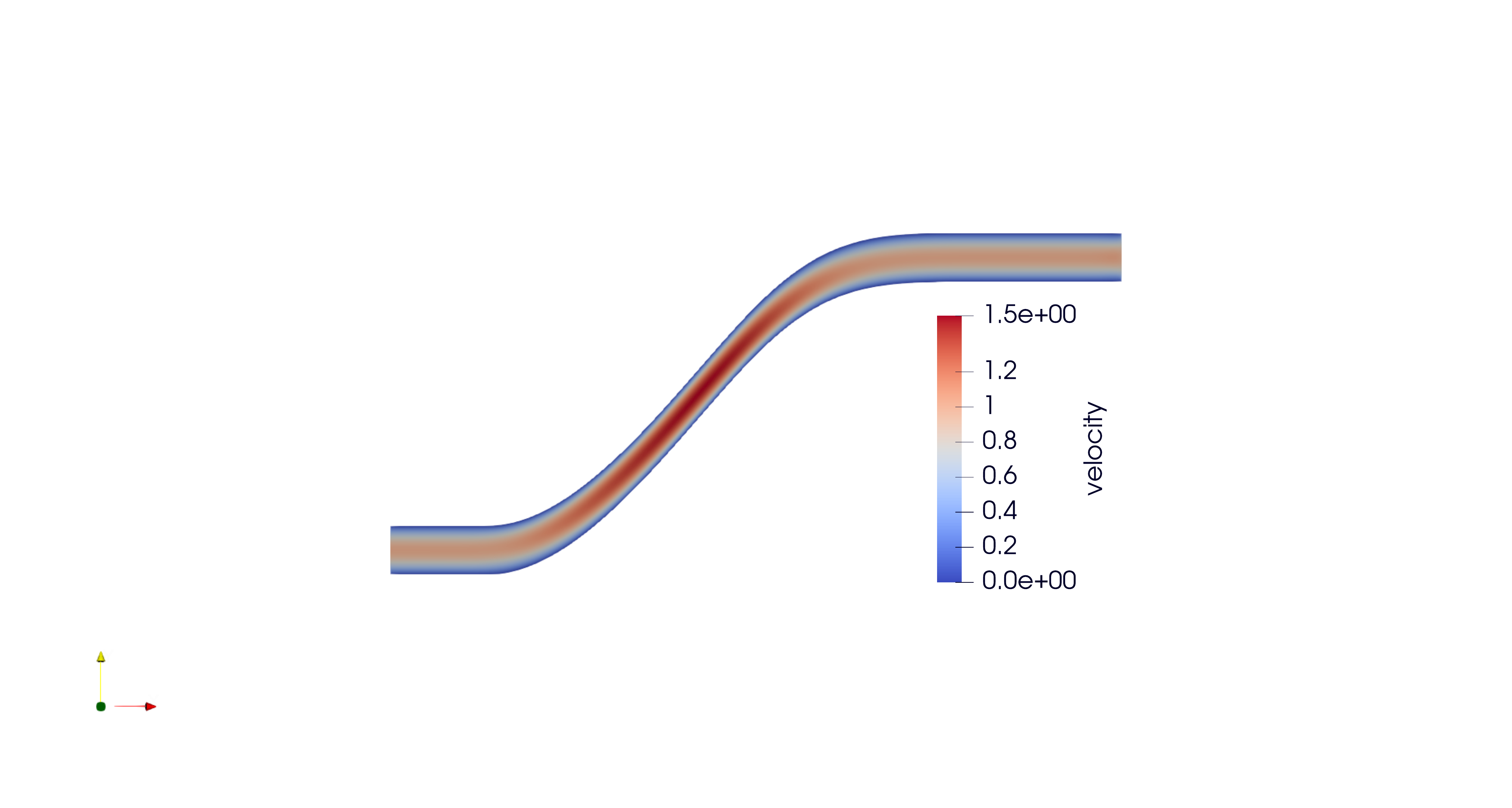}
		\end{minipage}
    }
    \caption{(a) The 2D pipe for the computation of Stokes flow; (b) The pressure in the computed results for $\nu = 1/1000$ case; (c) The magnitude of velocity for $\nu = 1/1000$ case. }
    \label{fig: Stokes-Flow}
\end{figure}
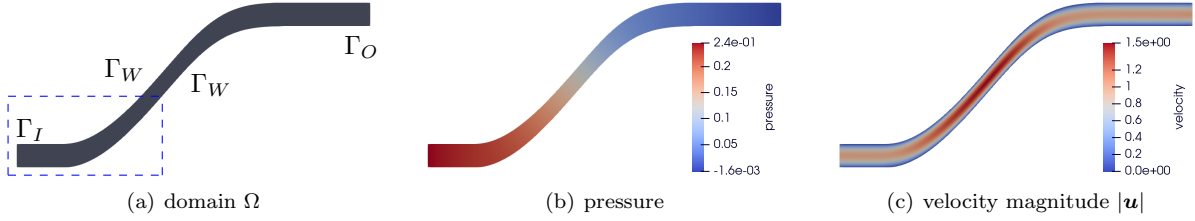

\begin{figure}[H]
    \centering
    \subfigure[$\mathcal{M}_{0}$ zoom-in]{
		\begin{minipage}[]{0.31\linewidth}
		  \centering
            \begin{tikzpicture}
                \node[anchor=south west,inner sep=0] (image) at (0,0) {\includegraphics[width=1.0\textwidth]{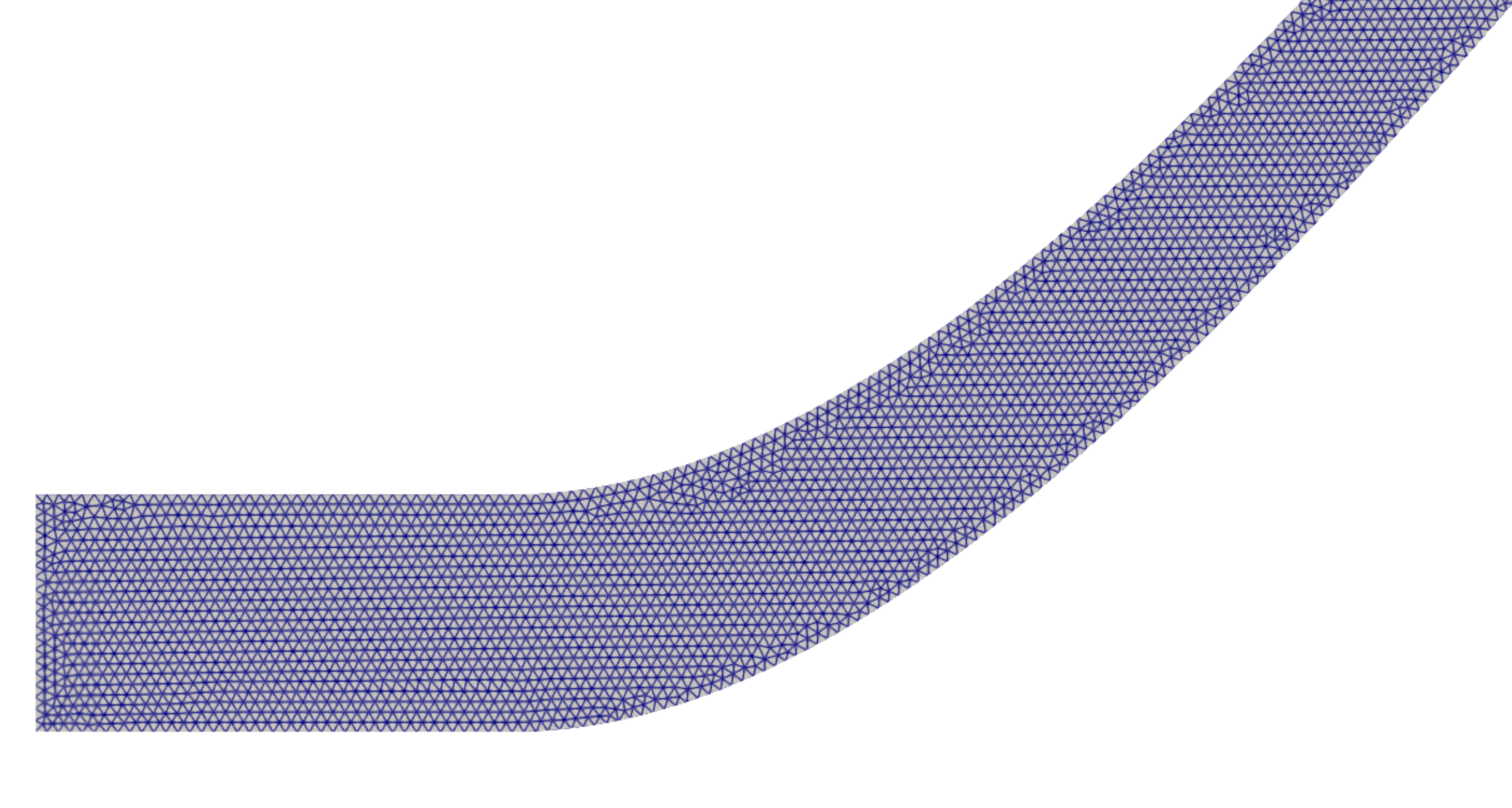}};
            \end{tikzpicture}
		\end{minipage}
    } \subfigure[$\mathcal{M}_{1}$ zoom-in]{
		\begin{minipage}[]{0.31\linewidth}
		  \centering
            \begin{tikzpicture}
                \node[anchor=south west,inner sep=0] (image) at (0,0) {\includegraphics[width=1.0\textwidth]{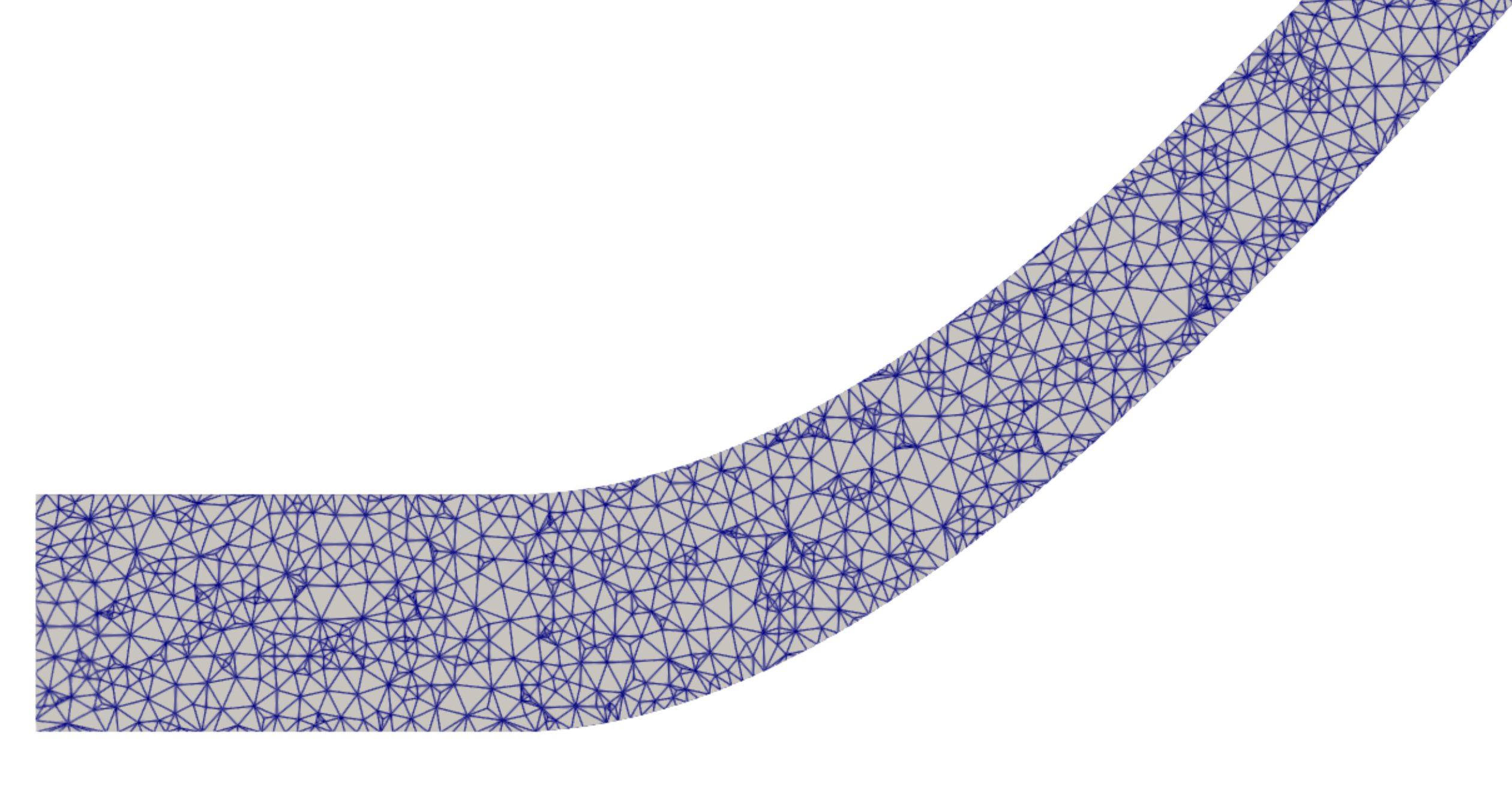}};
            \end{tikzpicture}
		\end{minipage}
    } \subfigure[$\mathcal{M}_{2}$ zoom-in]{
		\begin{minipage}[]{0.31\linewidth}
		  \centering
            \begin{tikzpicture}
                \node[anchor=south west,inner sep=0] (image) at (0,0) {\includegraphics[width=1.0\textwidth]{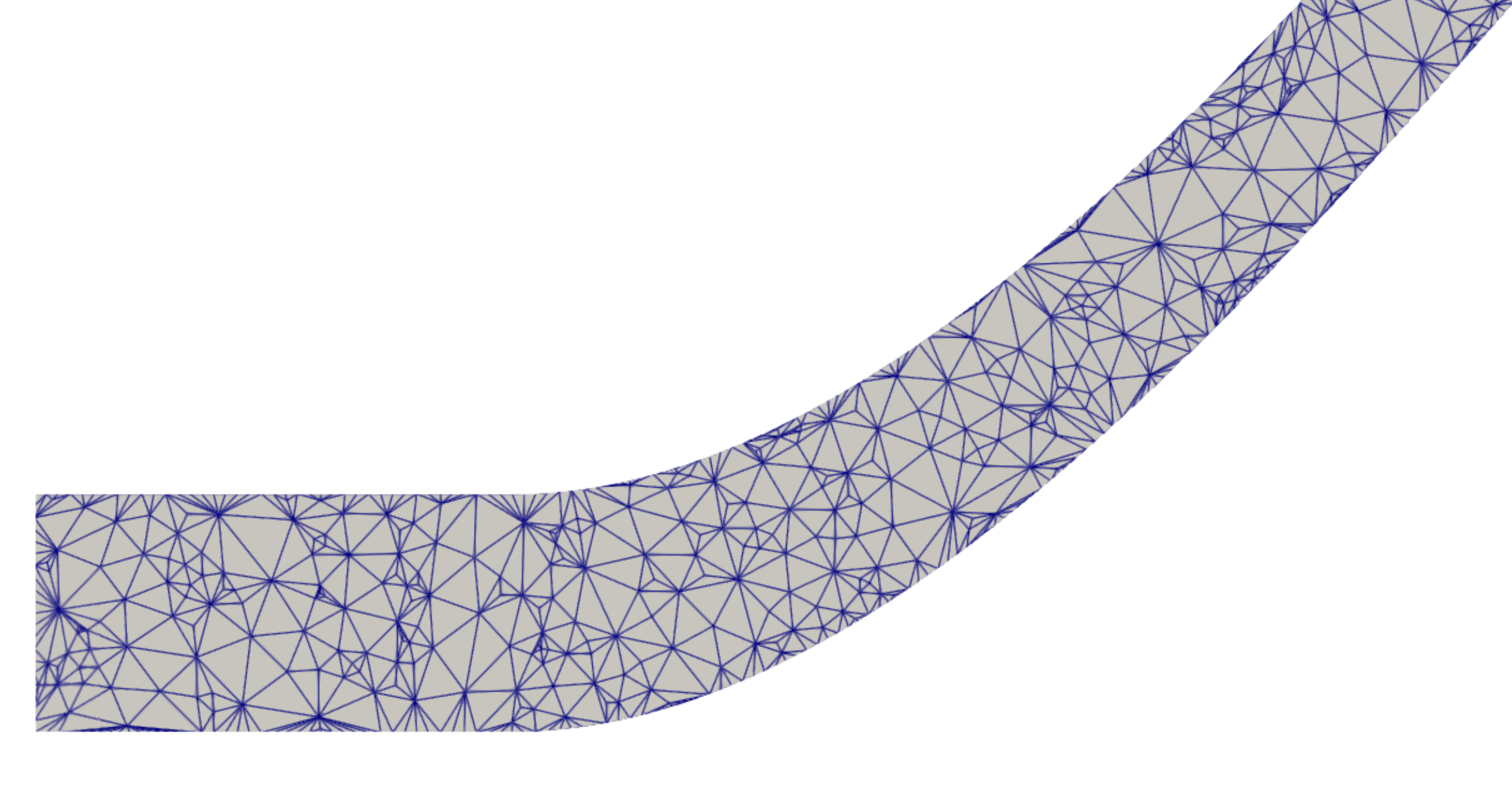}};
            \end{tikzpicture}
		\end{minipage}
    }
    \caption{The zoom-in views of the meshes $\mathcal{M}_{0}$, $\mathcal{M}_{1}$, and $\mathcal{M}_{2}$; (a) $\mathcal{M}_{0}$: 14,198 elements and 7,465 vertices; (b)  $\mathcal{M}_{1}$: 5,956 elements and 3,344 vertices; (c) $\mathcal{M}_{2}$: 2,544 elements and 1,638 vertices. }
    \label{fig: pipe-coarsening}
\end{figure}

We then compare the convergence history of the GMRES iterative solver with and without the proposed multilevel preconditioner. We test the cases of $\nu = 1/500$, $\nu = 1/1000$, and $\nu = 1/2000$ separately, utilizing 12 processors. The first 500 iterations are depicted in Figure~\ref{fig: Stokes-Flow-Iteration}. This figure demonstrates that the proposed multilevel preconditioner significantly enhances the convergence rate of the GMRES iterative solver. The three-level preconditioner outperforms the two-level preconditioner, both of which are substantially more effective than the one-level case. This is seen more clearly in Table~\ref{tab: convergence}, which shows the number of GMRES iterations, computing time, and parallel efficiency. 

As shown in Table~\ref{tab: convergence}, the coarsening time is negligible compared to the simulation time. The newly proposed multilevel preconditioner significantly decreases the computing time and enhances parallel efficiency compared to the one-level baseline. Specifically, the three-level preconditioner decreases the number of iterations to approximately $10\%$ of that required by the one-level case, while reducing the computing time to about $30\%$. From the one-level case to the three-level case, the parallel efficiency increased from approximately $70\%$ to over $90\%$ when using $12$ processors. The primary reason for this improvement is that the three-level preconditioner decreases the number of iterations and prevents the number of iterations from escalating as the number of processes increases. 

\begin{figure}[H]
    \centering
    \subfigure[$\nu = 1/500$]{
		\begin{minipage}[]{0.313\linewidth}
		  \centering
            \includegraphics[width=1.0\linewidth]{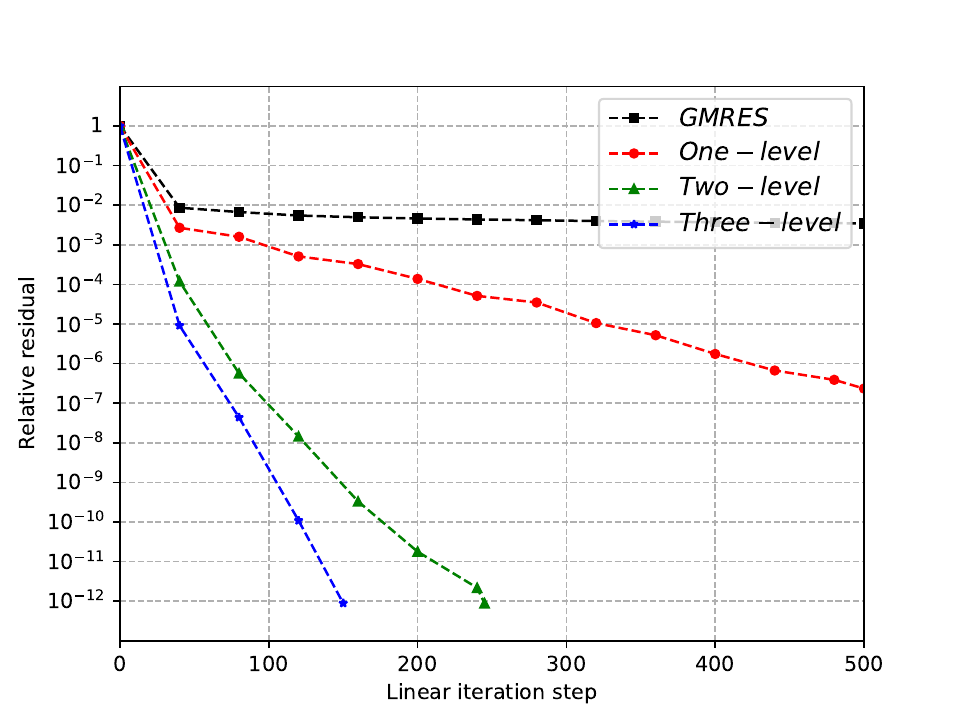}
		\end{minipage}
    } \subfigure[$\nu = 1/1000$]{
		\begin{minipage}[]{0.313\linewidth}
		  \centering
            \includegraphics[width=1.0\linewidth]{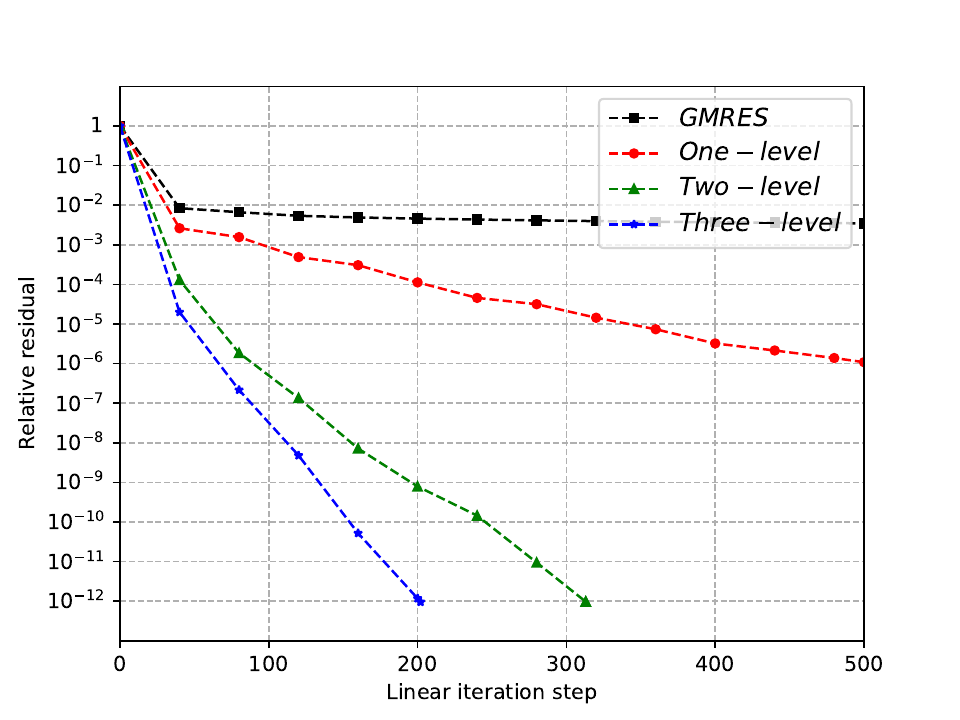}
		\end{minipage}
    } \subfigure[$\nu = 1/2000$]{
		\begin{minipage}[]{0.313\linewidth}
		  \centering
            \includegraphics[width=1.0\linewidth]{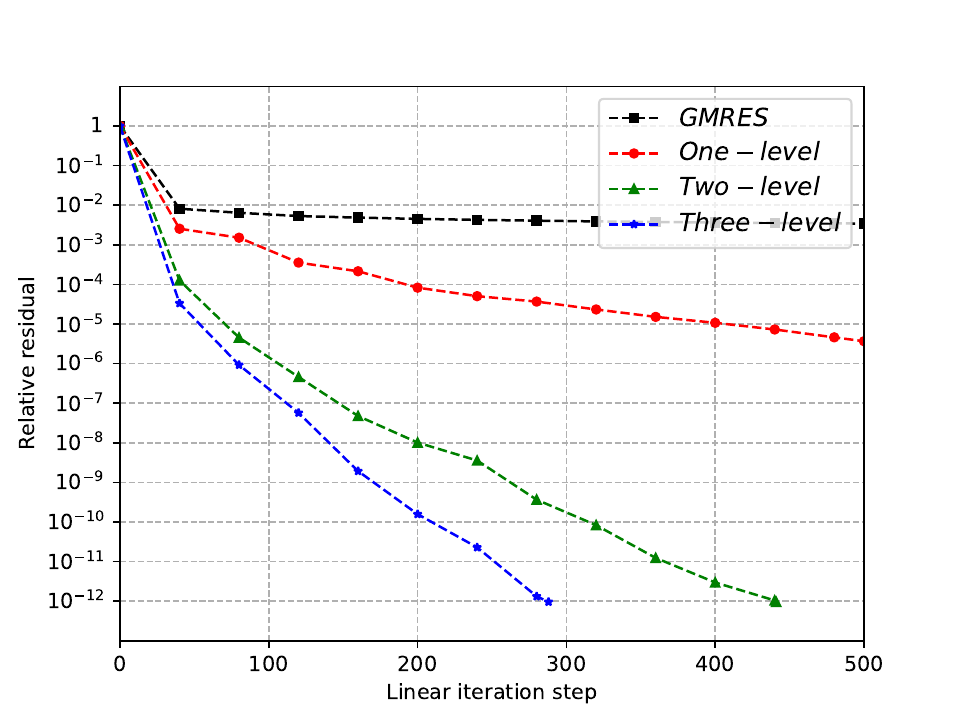}
		\end{minipage}
    }
    \caption{The first 500 iterations of GMRES iterative solver for the $\nu = 1/500$, $\nu = 1/1000$, and $\nu = 1/2000$ cases, respectively. The above results are all computed using 12 processors.}
    \label{fig: Stokes-Flow-Iteration}
\end{figure}

\begin{table}[H]
    \centering
    \caption{The coarsening time using $np$ =  3, 6, 9, and 12 processers, respectively; Comparison of GMRES iterations (denoted as ``$\hbox{N.It}$''), computing time (``Time(s)''), and parallel efficiency of the proposed Schwarz preconditioners with different levels.}
    \setlength{\tabcolsep}{3.8mm}{
    \center
    \begin{tabular} {|c|l|c|c|c|c|}
        \hline
            &$np$ &3 &6 &9 &12 \\\cline{1-6}
        $1_{st}$ coarsening &Time &0.394($100\%$) &0.258($76.36\%$) &0.239($54.95\%$) &0.213($46.24\%$) \\ \hline
        $2_{nd}$ coarsening &Time &0.207($100\%$) &0.151($68.54\%$) &0.139($49.64\%$) &0.138($37.50\%$) \\ \hline
        \multicolumn{6}{l}{$\nu = 1/500$: }\\ \hline
        \multirow{2}{*}{One-level} &$\hbox{N.It}$ &1364 &1324 &1385 &1438 \\\cline{2-6}
		&Time &30.67(100\%)  &16.81(91.23\%) &13.06(78.28\%) &11.35(67.56\%) \\\cline{1-6}
        \multirow{2}{*}{Two-level} &$\hbox{N.It}$ &245 &245 &246 &245 \\\cline{2-6}
		&Time &15.59(100\%)  &8.30(93.92\%) &6.60(78.74\%) &4.97(78.42\%) \\\hline
        \multirow{2}{*}{Three-level} &$\hbox{N.It}$ &160 &154 &151 &150 \\\cline{2-6}
		&Time &13.43(100\%)  &6.93(96.90\%) &4.92(90.99\%) &3.76(89.30\%) \\\hline
        
        \multicolumn{6}{l}{$\nu = 1/1000$: }\\ \hline
        \multirow{2}{*}{One-level} &$\hbox{N.It}$ &2077 &2042 &2140 &2054 \\\cline{2-6}
		&Time(s) &48.94(100\%)  &24.91(98.23\%) &20.89(78.09\%) &15.81(77.39\%) \\\cline{1-6}
        \multirow{2}{*}{Two-level} &$\hbox{N.It}$ &311 &306 &321 &313 \\\cline{2-6}
		&Time(s) &20.27(100\%)  &9.87(102.7\%) &8.48(79.68\%) &5.70(88.90\%) \\\hline
        \multirow{2}{*}{Three-level} &$\hbox{N.It}$ &211 &206 &203 &202 \\\cline{2-6}
		&Time(s) &18.77(100\%)  &8.66(108.4\%) &7.12(87.87\%) &4.84(96.95\%) \\\hline
        
        \multicolumn{6}{l}{$\nu = 1/2000$: }\\ \hline
        \multirow{2}{*}{One-level} &$\hbox{N.It}$ &3306 &3420 &3215 &3236 \\\cline{2-6}
		&Time(s) &77.81(100\%)  &39.57(98.32\%) &31.02(83.61\%) &24.73(78.66\%) \\\cline{1-6}
        \multirow{2}{*}{Two-level} &$\hbox{N.It}$ &426 &429 &431 &441 \\\cline{2-6}
		&Time(s) &28.44(100\%)  &13.52(105.2\%) &11.29(83.97\%) &8.34(85.25\%) \\\hline
        \multirow{2}{*}{Three-level} &$\hbox{N.It}$ &293 &288 &290 &288 \\\cline{2-6}
		&Time(s) &24.82(100\%)  &11.35(109.3\%) &9.51(87.00\%) &6.51(95.31\%) \\\hline
    \end{tabular}
    }
    \label{tab: convergence}
\end{table}

\ 

\subsubsection{3D Case}
Analogously, we consider the 3D case, in which $\Omega$ is the 3D vessel \citep{arterymodel} as shown in Figure~\ref{fig: Stokes-Flow 3D}(a).  In this case, we focus on the following boundary conditions, 
\begin{equation}
    \left \{
    \begin{aligned}
    & \bm u = -10\bm n,\  \texttt{on}\ \ \Gamma_{I},\\
    & \bm u =  \bm 0,\ \ \ \ \ \  \texttt{on}\ \ \Gamma_{W},\\
    & \bm \sigma \cdot \bm n = \bm 0,\ \ \texttt{on}\ \ \Gamma_{O}.
    \end{aligned}
    \right.
\end{equation}
Analogous to the 2D case, the Taylor-Hood element $P_{2}-P_{1}$ \citep{taylor1973numerical} is utilized to discretize this problem, and the DoFs are distributed as shown in Figure~\ref{multiphysics}(d).
Figures~\ref{fig: Stokes-Flow 3D}(b-d) present the computational results for the pressure $p$ and the velocity magnitude $|\bm u|$ of the Stokes flow in the 3D vessel, verifying that the numerical algorithm yields reasonable results. 
Figures~\ref{fig: vessel-coarsening}(a-b), (c-d), and (e-f) represent the fine to coarse meshes $\mathcal{M}_0$, $\mathcal{M}_1$, and $\mathcal{M}_2$, respectively. 
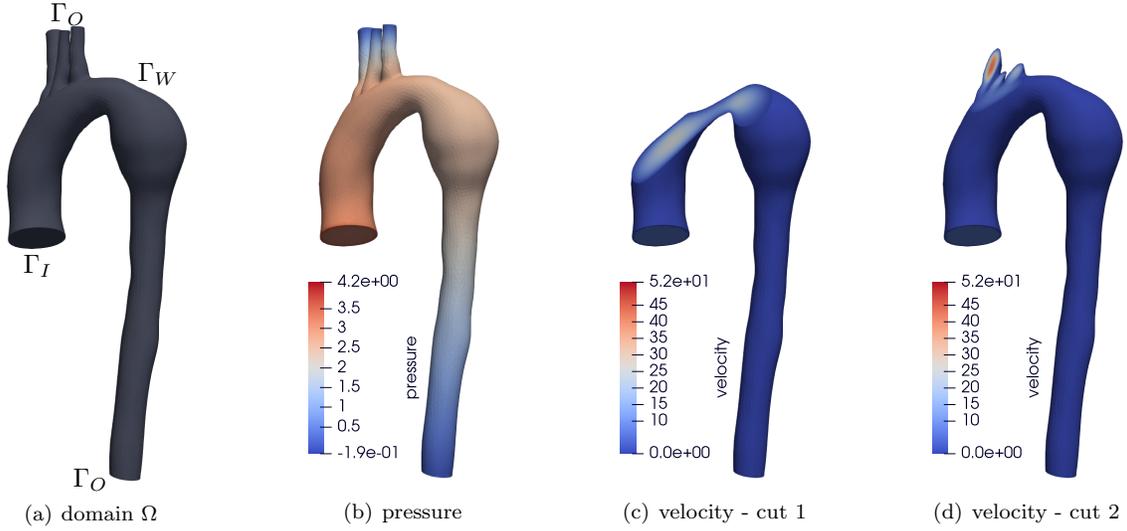
\begin{figure}[H]
    \centering
    \subfigure[domain $\Omega$]{
		\begin{minipage}[]{0.23\linewidth}
		  \centering
            \begin{tikzpicture}
                \node[anchor=south west,inner sep=0] (image) at (0,0) {\includegraphics[width=1.0\textwidth]{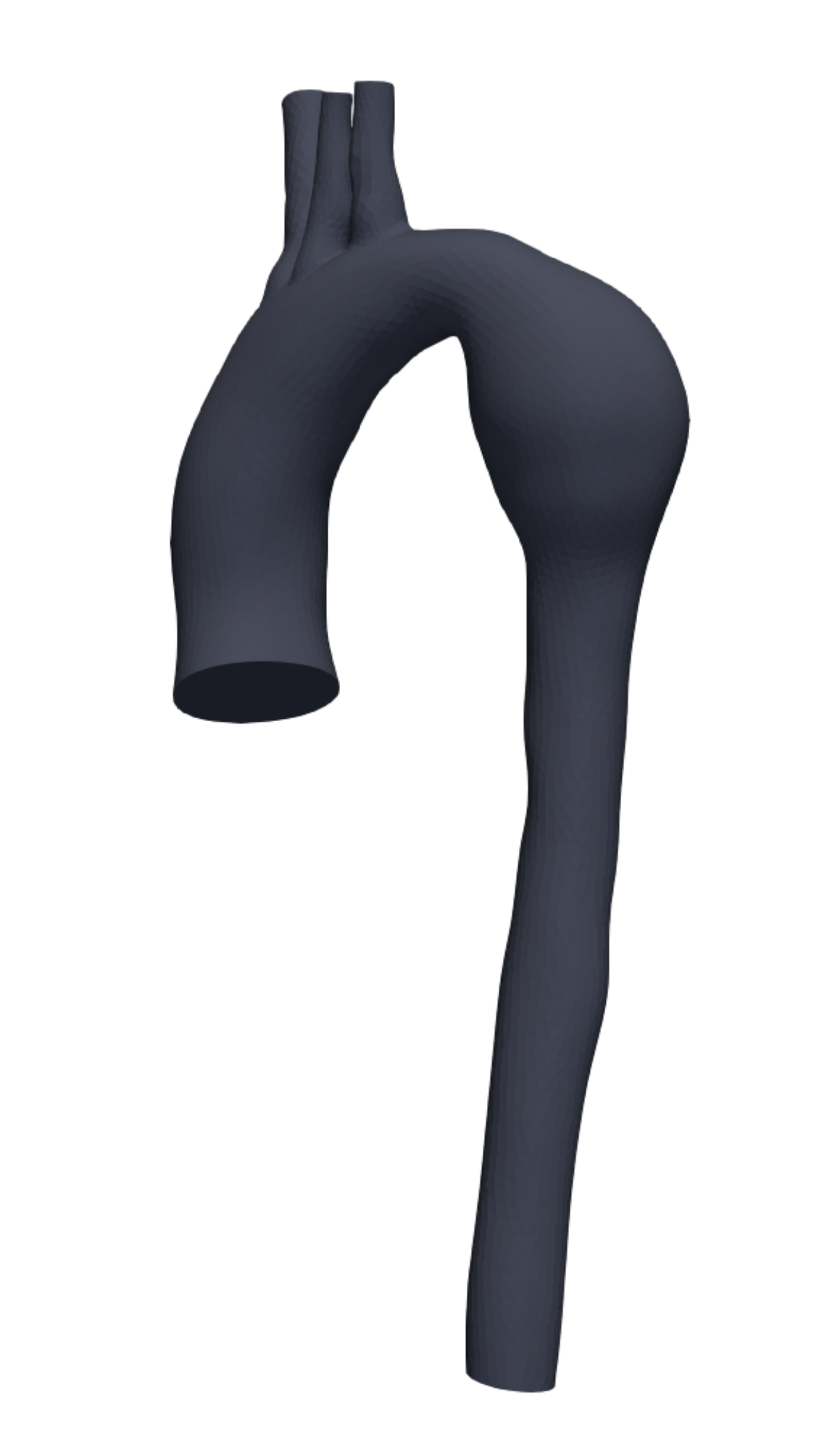}};
                \node at (1.2,3.15) {$\Gamma_{I}$};
                \node at (1.6,6.5) {$\Gamma_{O}$};
                \node at (1.9,0.3) {$\Gamma_{O}$};
                \node at (2.8,5.7) {$\Gamma_{W}$};
            \end{tikzpicture}
        \end{minipage}
    } \subfigure[pressure]{
		\begin{minipage}[]{0.23\linewidth}
		  \centering
		  \includegraphics[width=1.0\linewidth]{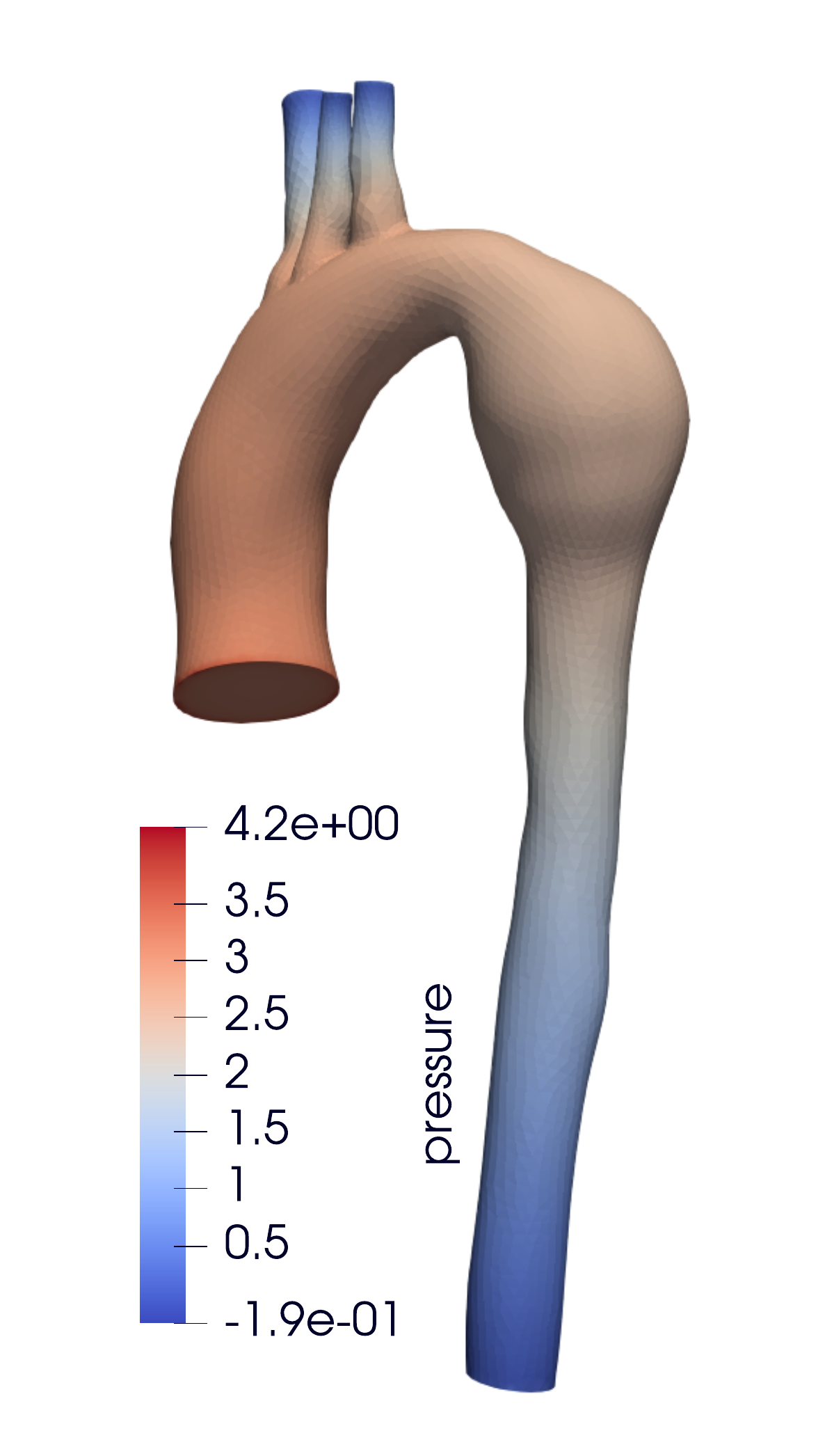}
		\end{minipage}
    } \subfigure[velocity - cut 1]{
		\begin{minipage}[]{0.23\linewidth}
		  \centering
		  \includegraphics[width=1.0\linewidth]{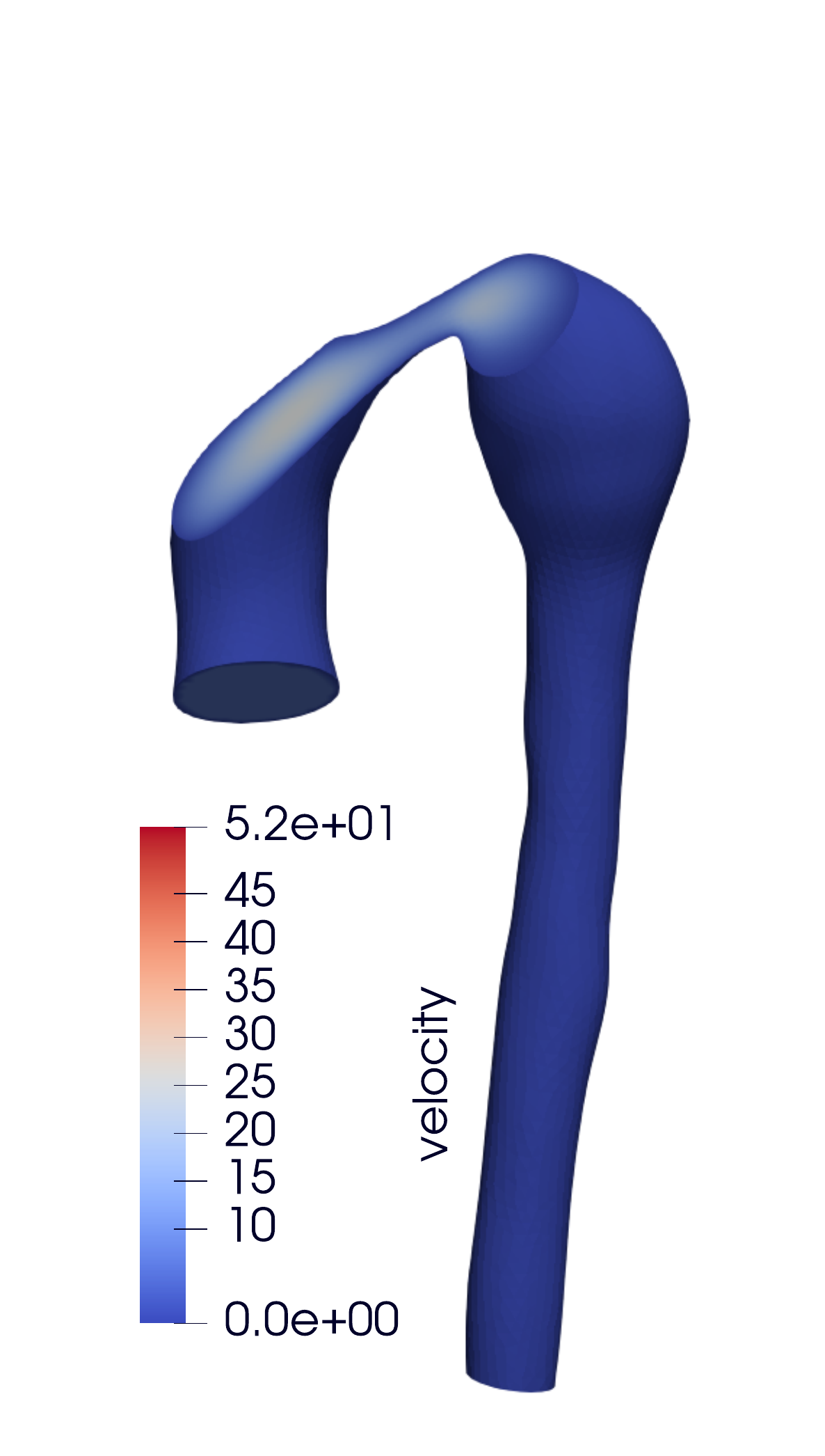}
		\end{minipage}
    } \subfigure[velocity - cut 2]{
		\begin{minipage}[]{0.23\linewidth}
		  \centering
		  \includegraphics[width=1.0\linewidth]{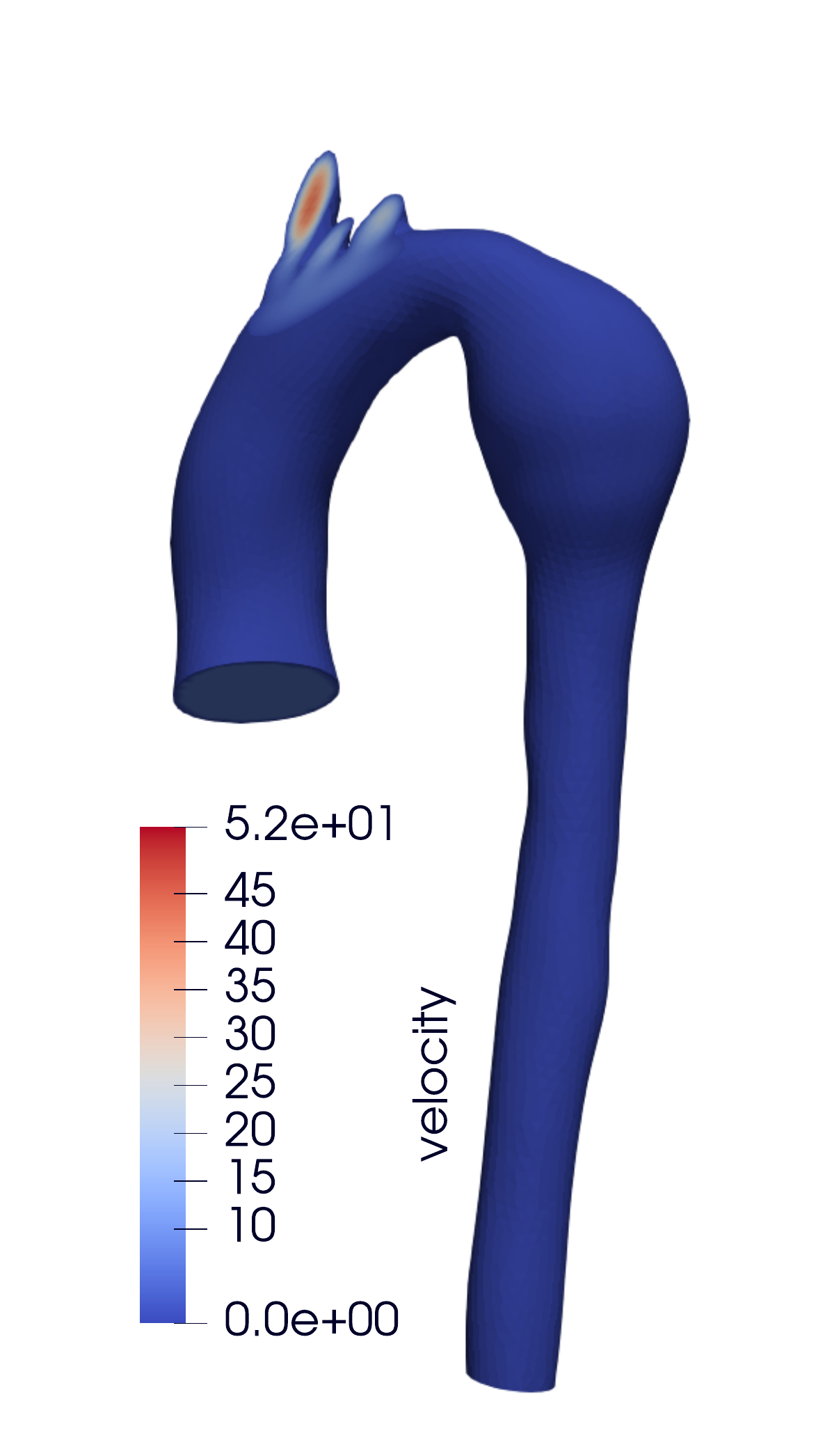}
		\end{minipage}
    } 
    \caption{(a) The 3D vessel for the computation of Stokes flow; (b) The pressure in the computed results for the $\nu = 1/1000$ case; (c) - (d) The magnitude of velocity for the $\nu = 1/1000$ case in the (c) artery, and (d) branches, respectively; }
    \label{fig: Stokes-Flow 3D}
\end{figure}

\begin{figure}[H]
    \centering
    \subfigure[$\mathcal{M}_0$($y < 1.5$)]{
		\begin{minipage}[]{0.135\linewidth}
		  \centering
		  \includegraphics[width=1.0\linewidth]{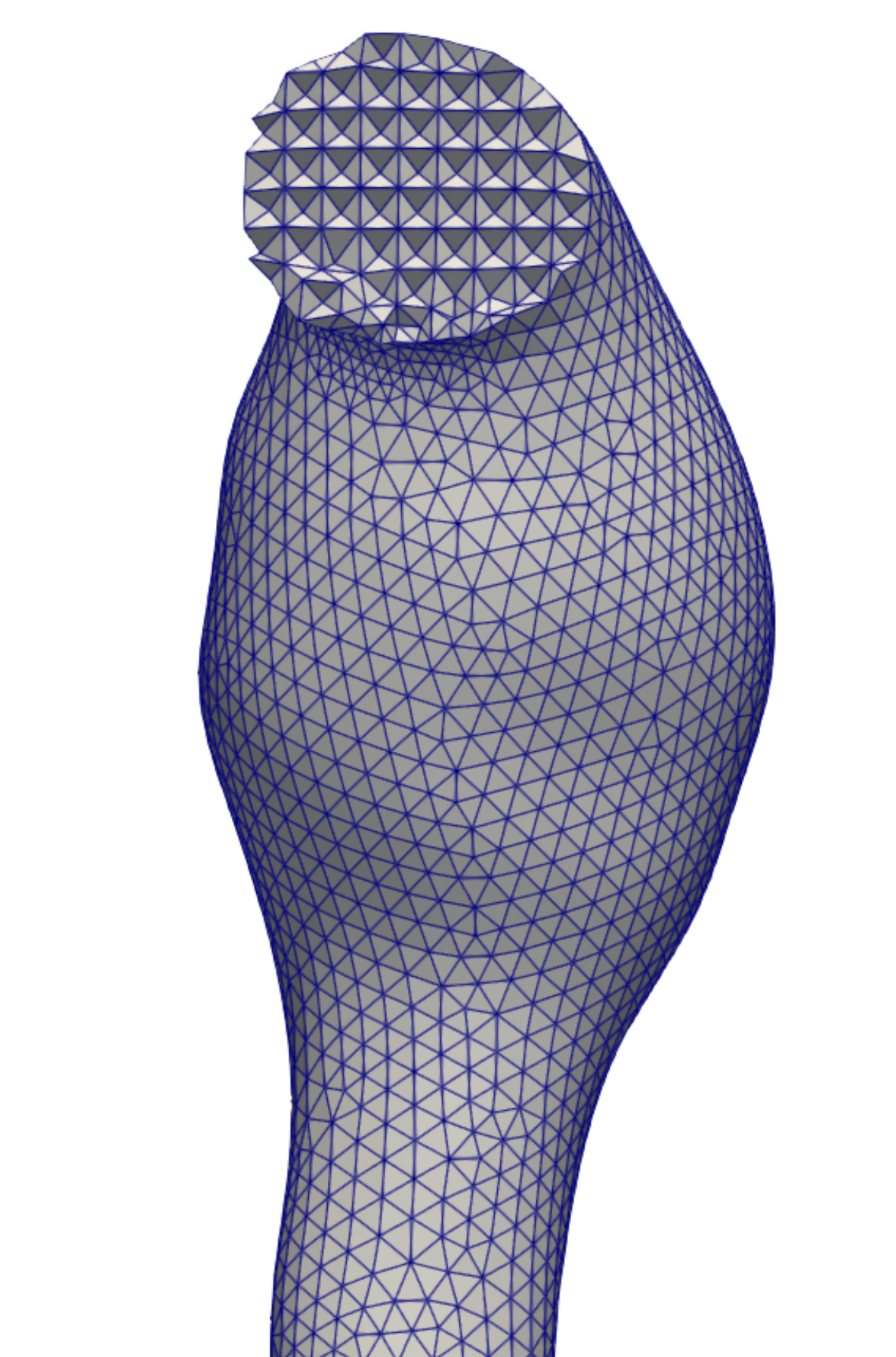}
		\end{minipage}
    } \subfigure[$\mathcal{M}_0$($z < 3.0$)]{
		\begin{minipage}[]{0.135\linewidth}
		  \centering
		  \includegraphics[width=1.0\linewidth]{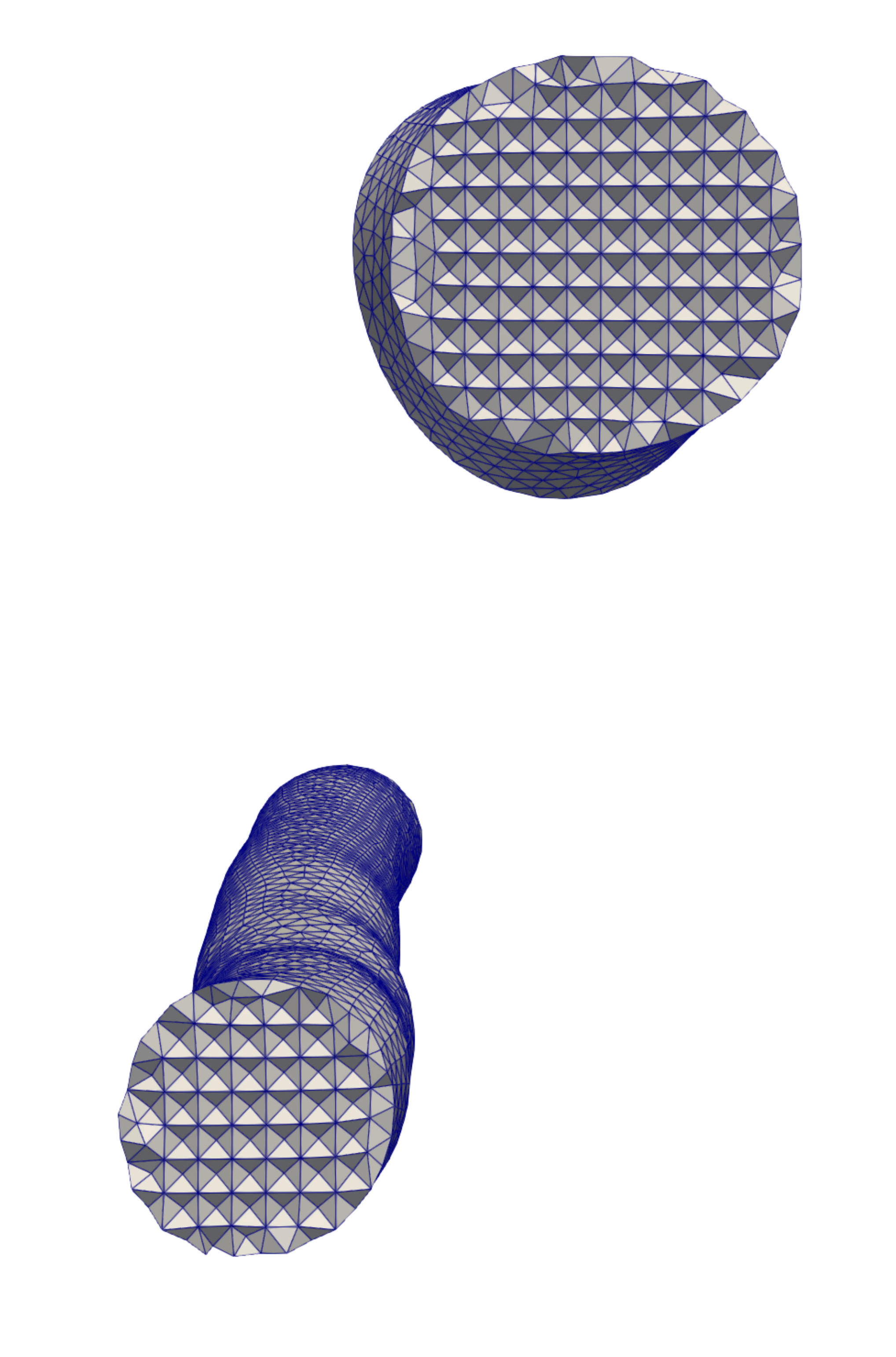}
		\end{minipage}
    } \subfigure[$\mathcal{M}_1$($y < 1.5$)]{
		\begin{minipage}[]{0.135\linewidth}
		  \centering
		  \includegraphics[width=1.0\linewidth]{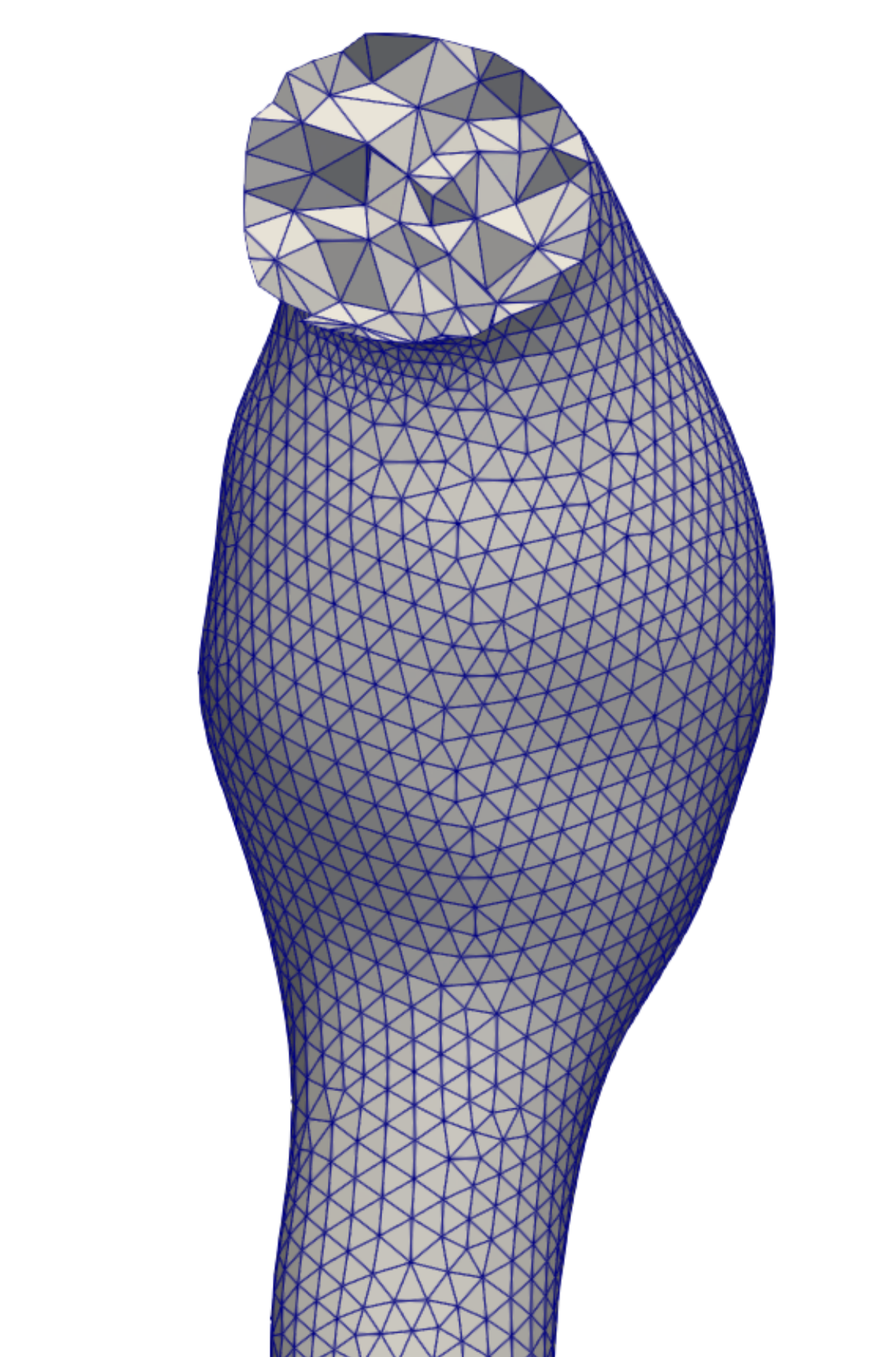}
		\end{minipage}
    } \subfigure[$\mathcal{M}_1$($z < 3.0$)]{
		\begin{minipage}[]{0.135\linewidth}
		  \centering
		  \includegraphics[width=1.0\linewidth]{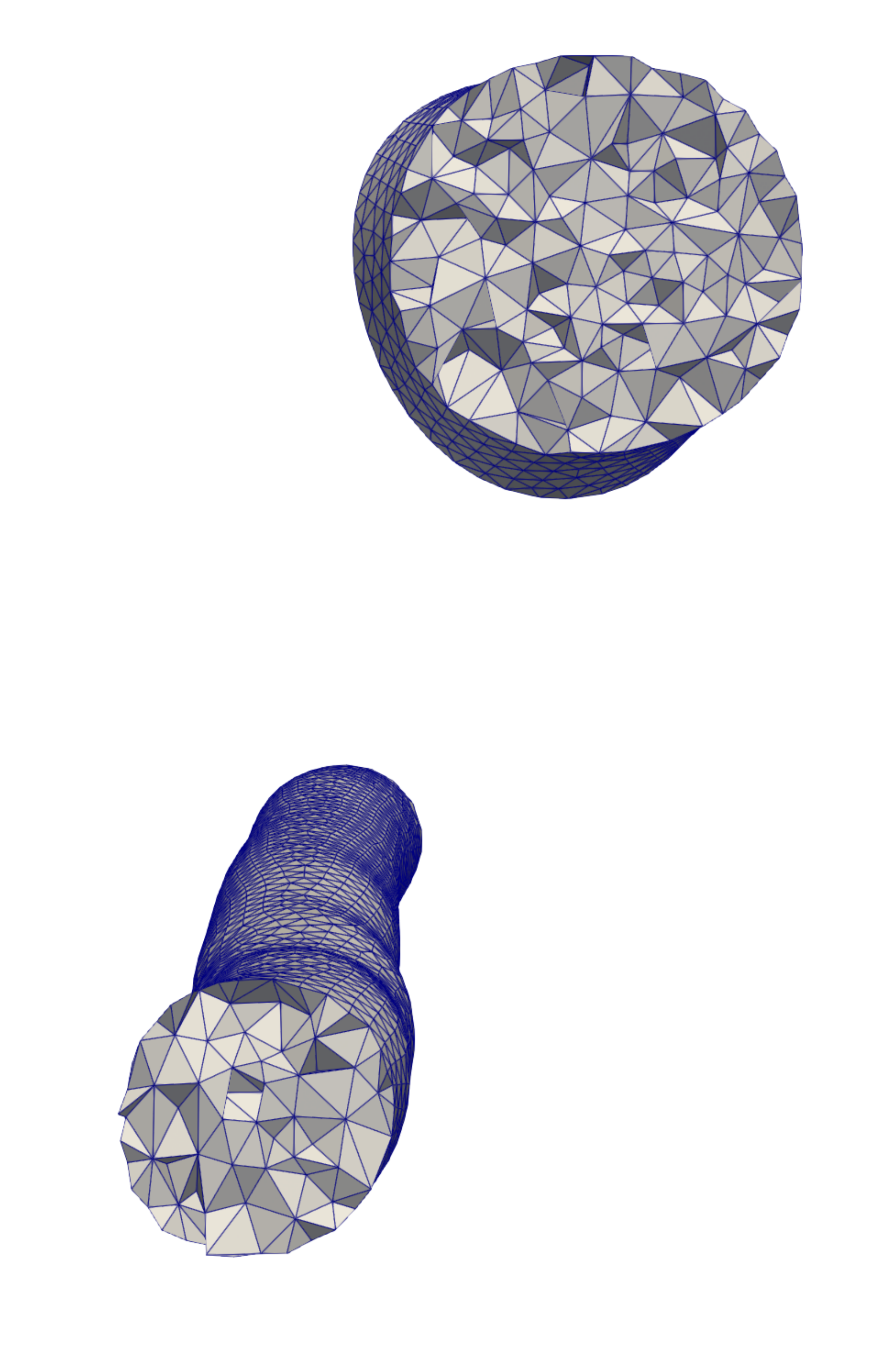}
		\end{minipage}
    } \subfigure[$\mathcal{M}_2$($y < 1.5$)]{
		\begin{minipage}[]{0.135\linewidth}
		  \centering
		  \includegraphics[width=1.0\linewidth]{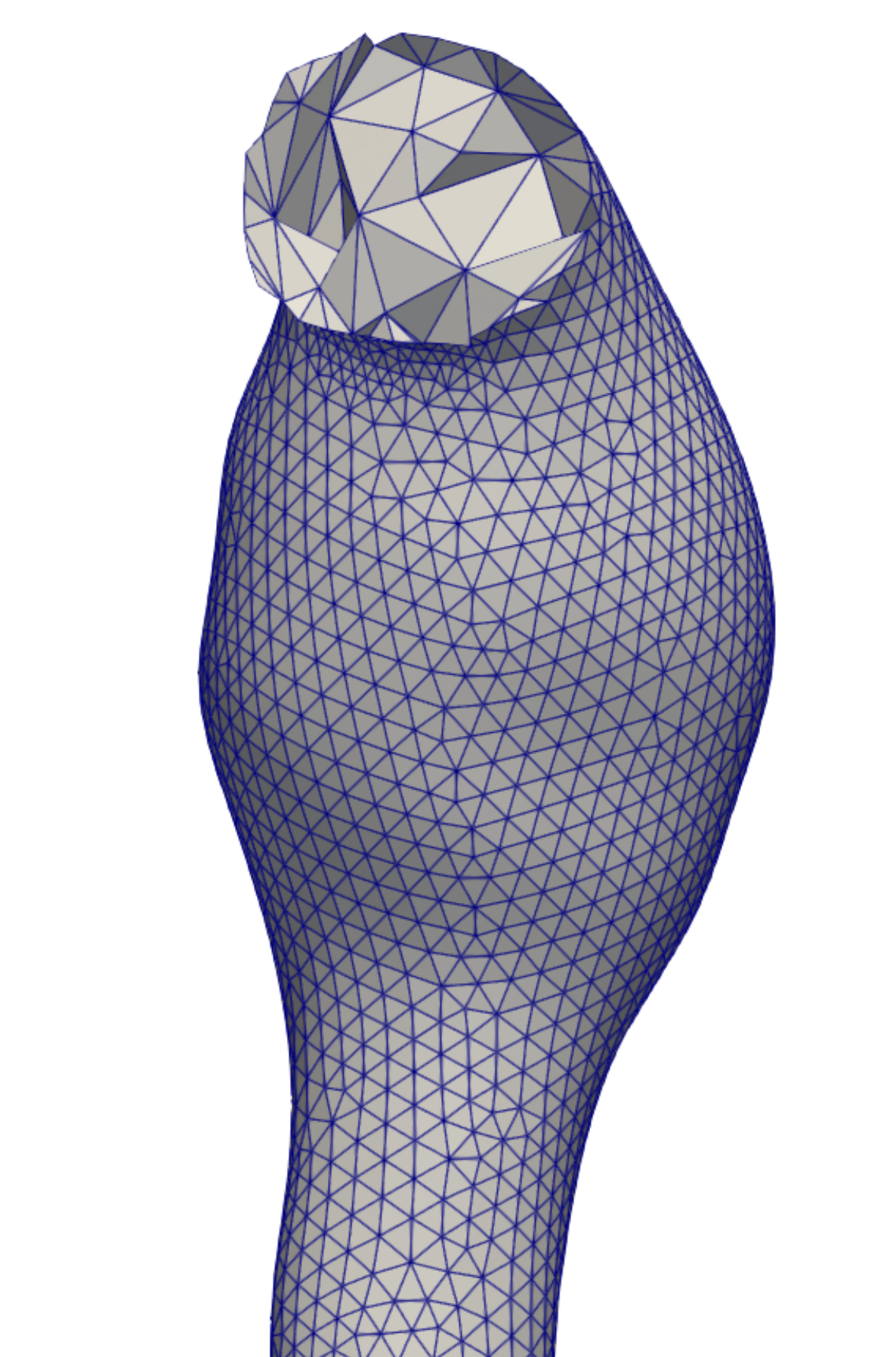}
		\end{minipage}
    } \subfigure[$\mathcal{M}_2$($z < 3.0$)]{
		\begin{minipage}[]{0.135\linewidth}
		  \centering
		  \includegraphics[width=1.0\linewidth]{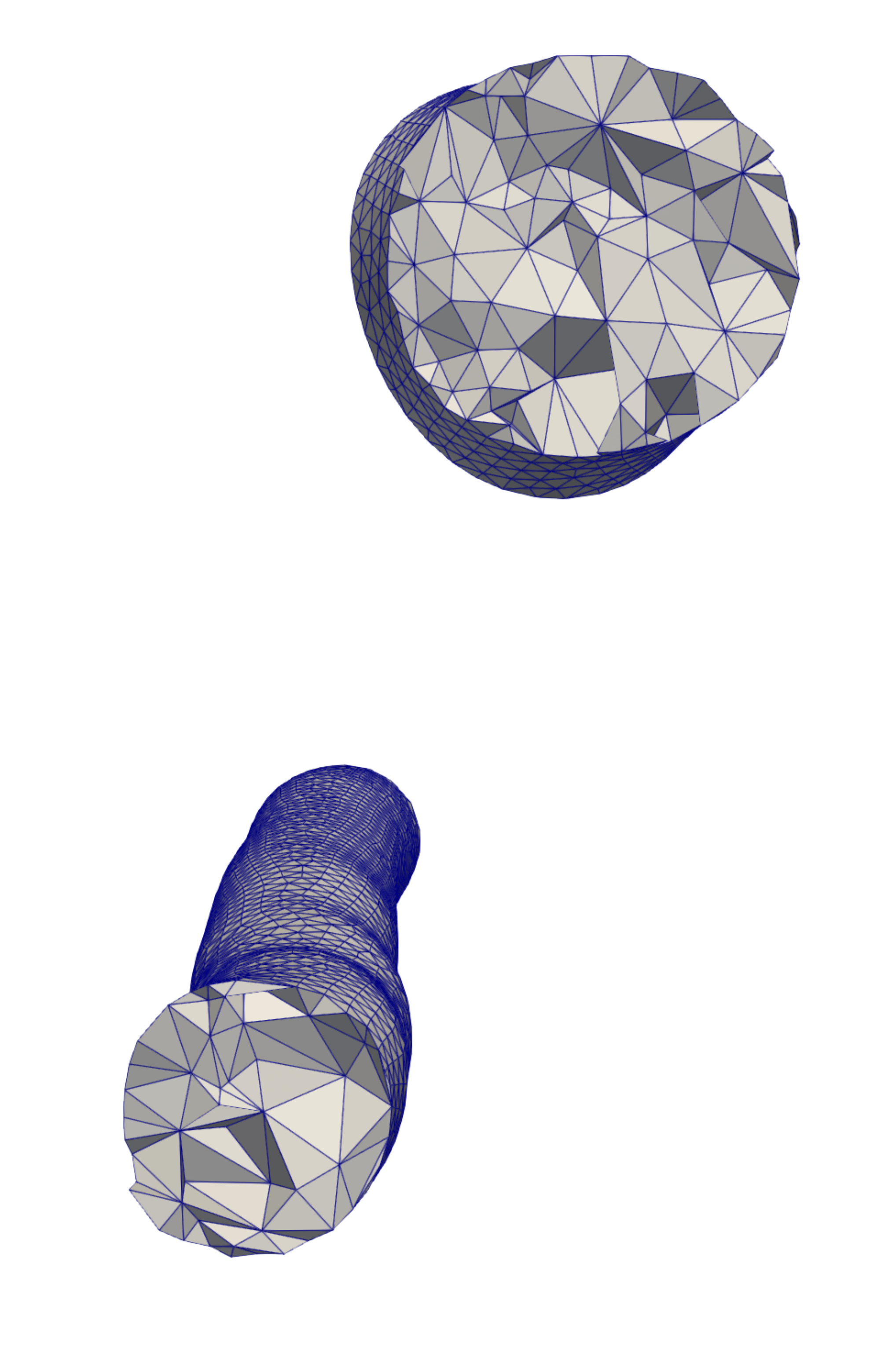}
		\end{minipage}
    } 
    \caption{Cross-sectional views of the meshes $\mathcal{M}_{0}$, $\mathcal{M}_{1}$, and $\mathcal{M}_{2}$ in the y- and z-directions; (a-b) $\mathcal{M}_{0}$: 144,880 elements and 28,193 vertices; (c-d) $\mathcal{M}_{1}$: 80,093 elements and 16,422 vertices; (e-f) $\mathcal{M}_{2}$: 43,869 elements and 10,923 vertices;}
    \label{fig: vessel-coarsening}
\end{figure}

In Figure~\ref{fig: 3DStokes-Flow-Iteration}, we depict the convergence history of the GMRES iterative solver with and without the proposed multilevel preconditioner. These results demonstrate that the proposed multilevel preconditioner significantly decreases the number of GMRES iterations. In this 3D problem, the three-level preconditioner performs slightly better than the two-level preconditioner, both of which are substantially more effective than the one-level case.

Table~\ref{tab: convergence-Stokes-Flow-3D} presents the number of GMRES iterations, computing time, and parallel efficiency for experiments with various $\nu$. Due to the high parallelism of the coarsening algorithm, the coarsening time consistently remains below $1\%$ of the simulation time as the number of processors increases from 3 to 12. An interesting phenomenon observed is that the more challenging the equation is to solve (with a smaller $\nu$), the more apparent the advantages of our multilevel preconditioner become. For $\nu = 1/2000$, the three-level preconditioner decreases the number of iterations to less than $40\%$ of that required by the one-level case, while reducing the computing time to about $70\%$. Moreover, the proposed multilevel preconditioner is highly scalable in the sense that the parallel efficiency does not decrease much when the number of processors is increased, which exhibits a parallel efficiency of $90\%$ for the three-level case at 12 processors. The outstanding parallel efficiency implies that we can utilize more processors to further accelerate the solution of this problem, which is demonstrated in the next subsection.

\begin{figure}[H]
    \centering
    \subfigure[$\nu = 1/500$]{
		\begin{minipage}[]{0.313\linewidth}
		  \centering
		  \includegraphics[width=1.0\linewidth]{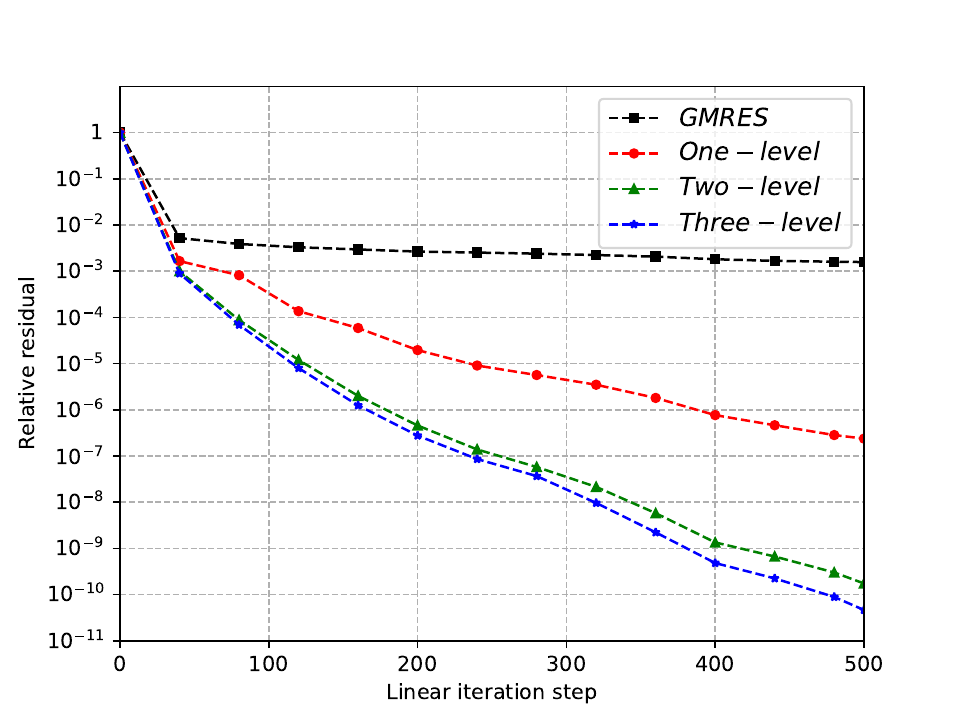}
		\end{minipage}
    } \subfigure[$\nu = 1/1000$]{
		\begin{minipage}[]{0.313\linewidth}
		  \centering
		  \includegraphics[width=1.0\linewidth]{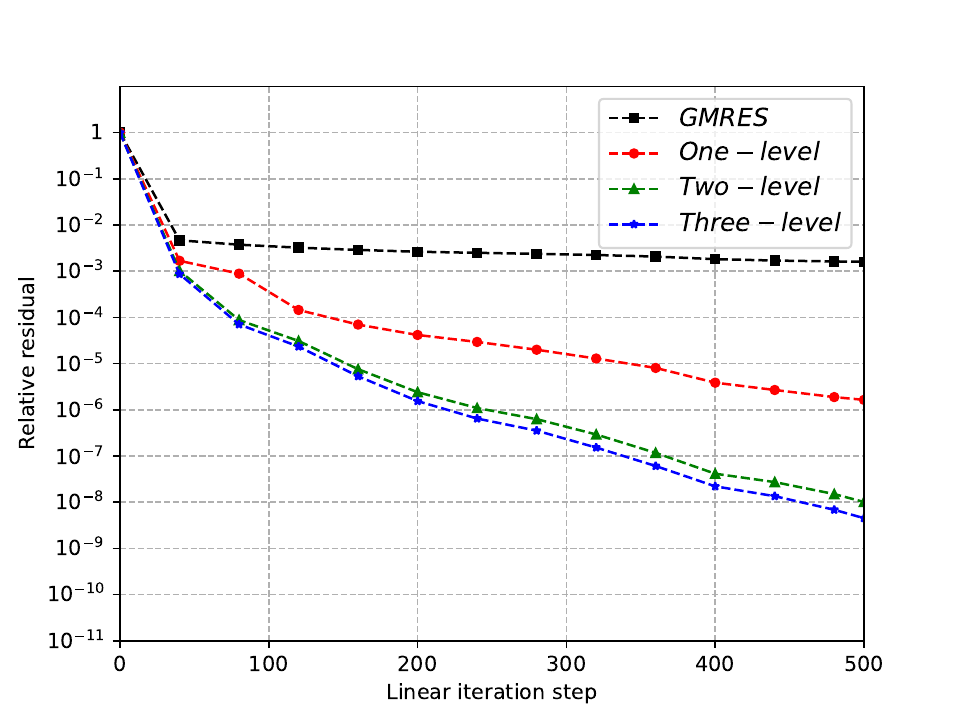}
		\end{minipage}
    } \subfigure[$\nu = 1/2000$]{
		\begin{minipage}[]{0.313\linewidth}
		  \centering
		  \includegraphics[width=1.0\linewidth]{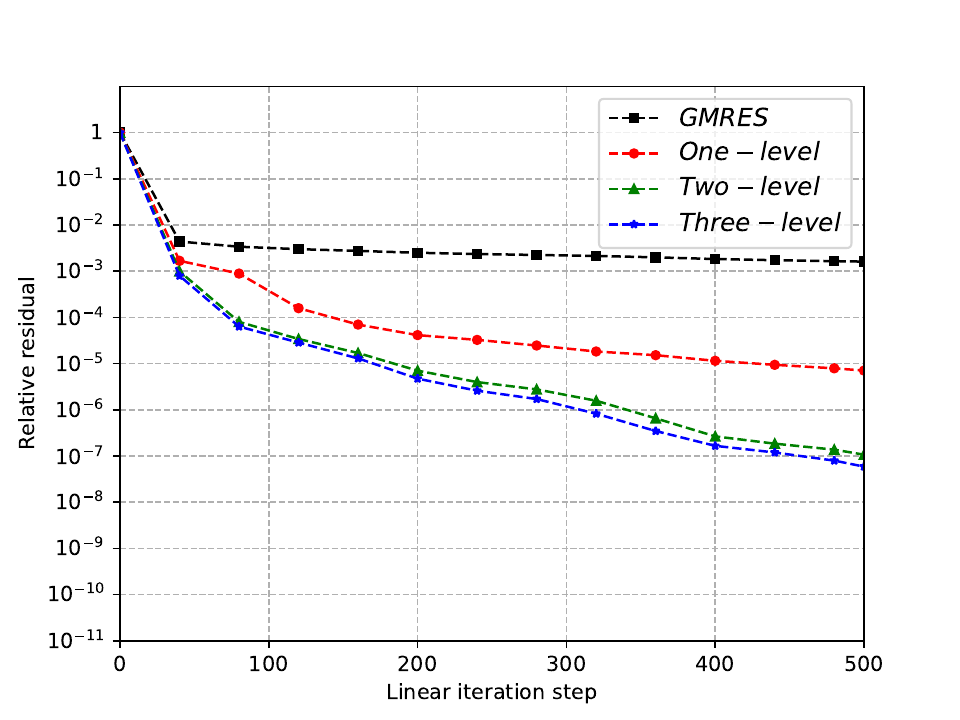}
		\end{minipage}
    }
    \caption{The first 500 iterations of the GMRES iterative solver for the $\nu = 1/500$, $\nu = 1/1000$, and $\nu = 1/2000$ cases, respectively. The above results are all computed using 12 processors. }
    \label{fig: 3DStokes-Flow-Iteration}
\end{figure}

\begin{table}[H]
    \centering
    \caption{The coarsening time for 3D vessel using $np$ = 3, 6, 9, 12 processors, respectively; Comparison of GMRES iterations (denoted as ``$\hbox{N.It}$''), computing time (``Time(s)''), and parallel efficiency of the proposed Schwarz preconditioners with different levels.}
    \setlength{\tabcolsep}{2.6mm}{
    \center
    \begin{tabular} {|c|l|c|c|c|c|}
        \hline
            &$np$ &3 &6 &9 &12 \\\cline{1-6}
        $1_{st}$ coarsening &Time &33.96($100\%$) &18.27($92.94\%$) &12.58($89.98\%$) &9.88($85.93\%$) \\ \hline
        $2_{nd}$ coarsening &Time &21.12($100\%$) &12.49($84.55\%$) &8.81($79.91\%$) &7.08($74.58\%$) \\ \hline
        \multicolumn{6}{l}{$\nu = 1/500$: }\\ \hline
        \multirow{2}{*}{One-level} &$\hbox{N.It}$ &1756 &1643 &1600 &1623 \\\cline{2-6}
		&Time &6111.17(100\%) &3299.29(92.61\%) &2386.04(85.37\%) &1926.23(79.32\%) \\\cline{1-6}
        \multirow{2}{*}{Two-level} &$\hbox{N.It}$ &749 &693 &662 &670 \\\cline{2-6}
		&Time &5925.57(100\%)  &3175.68(93.30\%) &2076.78(95.11\%) &1778.95(83.27\%) \\\hline
        \multirow{2}{*}{Three-level} &$\hbox{N.It}$ &718 &637 &606 &604 \\\cline{2-6}
		&Time &5963.82(100\%) &3083.64(96.70\%) &2009.92(98.91\%) &1666.89(89.45\%) \\\hline
        \multicolumn{6}{l}{$\nu = 1/1000$: }\\ \hline
        \multirow{2}{*}{One-level} &$\hbox{N.It}$ &3013 &2861 &2716 &2752 \\\cline{2-6}
		&Time &10077.0(100\%) &5223.27(96.46\%) &3782.85(88.80\%) &3069.36(82.08\%)\\\cline{1-6}
        \multirow{2}{*}{Two-level} &$\hbox{N.It}$ &1168 &1070 &996 &1031 \\\cline{2-6}
		&Time &9049.07(100\%)  &4693.81(96.39\%) &2953.36(102.1\%) &2562.25(88.29\%) \\\hline
        \multirow{2}{*}{Three-level} &$\hbox{N.It}$ &1104 &986 &925 &937 \\\cline{2-6}
		&Time &9634.09(100\%)  &4915.34(98.00\%) &3207.37(100.1\%) &2501.12(96.30\%) \\\hline
        \multicolumn{6}{l}{$\nu = 1/2000$: }\\ \hline
        \multirow{2}{*}{One-level} &$\hbox{N.It}$ &4531 &4735 &4548 &4649 \\\cline{2-6}
		&Time &14503.3(100\%) &8531.67(85.00\%) &5836.12(82.84\%) &5109.30(70.97\%) \\\cline{1-6}
        \multirow{2}{*}{Two-level} &$\hbox{N.It}$ &1802 &1649 &1539 &1600 \\\cline{2-6}
		&Time &12700.5(100\%)  &6718.97(94.51\%) &4228.81(100.1\%) &3745.97(84.76\%) \\\hline
        \multirow{2}{*}{Three-level} &$\hbox{N.It}$ &1681 &1461 &1357 &1399 \\\cline{2-6}
		&Time &13294.5(100\%)  &6632.56(100.2\%) &4049.31(109.4\%) &3582.78(92.77\%) \\\hline
    \end{tabular}
    }
    \label{tab: convergence-Stokes-Flow-3D}
\end{table}


\ 

\subsection{Parallel scalability}
We then focus on studying the parallel scalability of the proposed multilevel smoothed Schwarz preconditioner in strong scaling senses. Here, the restart value of GMRES is fixed at 200, and the absolute and relative tolerances for the GMRES iterative solver are set to $10^{-6}$. In this experiment, the RAS preconditioner with $\delta = 1$ serves as the smoother on each level. Then, we select ILU(2) as the subdomain solver on the finest mesh, while ILU(0) is utilized on the coarse mesh hierarchies $\mathcal{M}_{i}$ for $i \geq 1$. We take the case of $\nu = 1/1000$ for Stokes flow in 3D vessel (Figure~\ref{fig: Stokes-Flow 3D}) as the benchmark test case. In this simulation, the total number of elements is $1.16 \times 10^{6}$, and the total number of DoFs amounts to $5.21 \times 10^{6}$. Correspondingly, the numbers of elements for coarse mesh hierarchies $\mathcal{M}_{1}$ and $\mathcal{M}_2$ are $8.22 \times 10^{5}$ and $6.12 \times 10^{5}$, respectively. 

Figure~\ref{fig: scaling} illustrates (a) the number of iterations and (b) the total computing time and speedup of the simulation using preconditioners with different levels. The results show that the number of GMRES iterations does not increase, and the total computing time decreases proportionally with the increase in the number of processors, indicating good scalability of the proposed algorithm. Compared to the one-level preconditioner, the two-level and three-level preconditioners provide significant improvements in computing time, number of iterations, and parallel efficiency. Specifically, when the number of processors reaches 1,024, the parallel efficiencies of the one-level, two-level, and three-level simulations are approximately $64.42\%$, $67.32\%$, and $73.25\%$, respectively.

\begin{figure}[H]
    \centering
    \subfigure[Number of iterations]{
		\begin{minipage}[]{0.482\linewidth}
		  \raggedleft
            \includegraphics[width=1.0\linewidth]{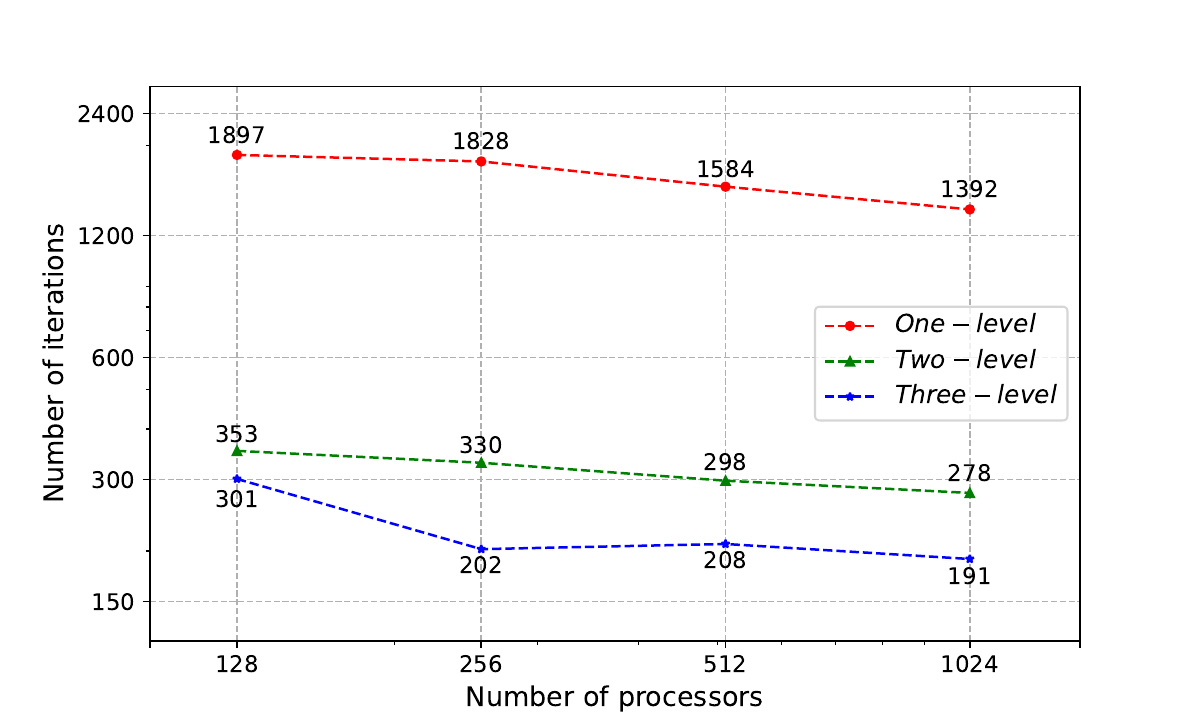}
		\end{minipage}
    } \subfigure[Total computing time and speedup]{
		\begin{minipage}[]{0.482\linewidth}
		  \raggedright
            \includegraphics[width=1.0\linewidth]{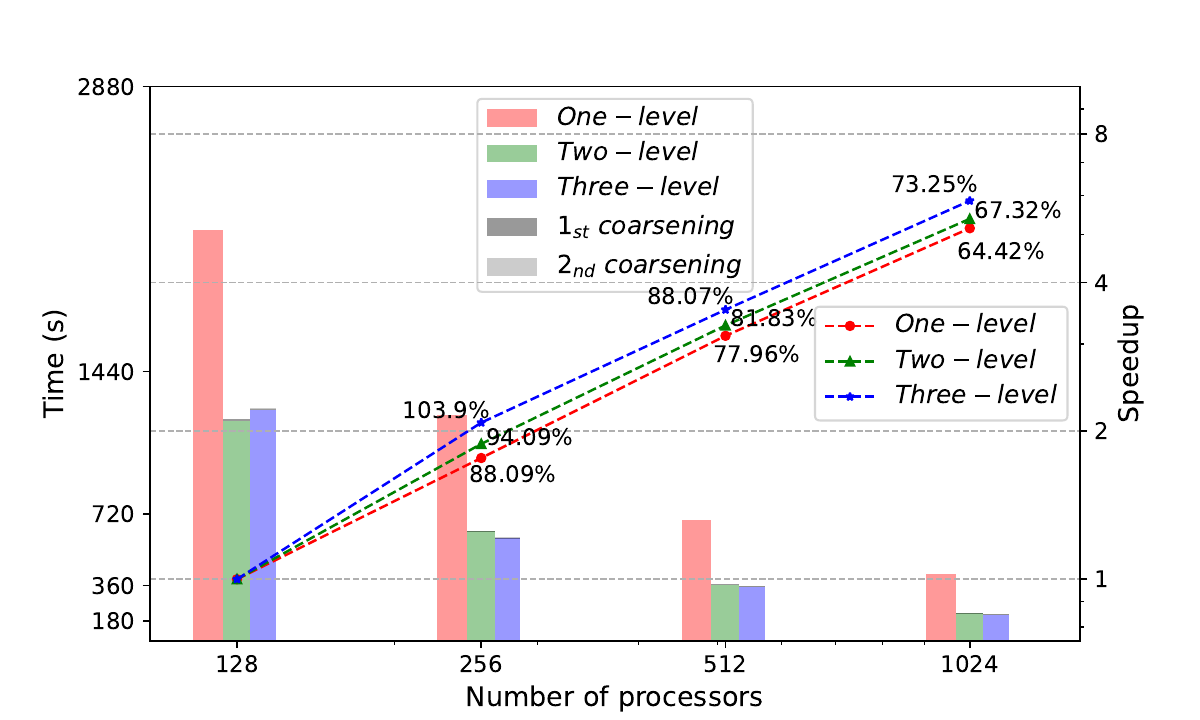}
		\end{minipage}
    }
    \caption{The parallel scalability of the multilevel preconditioner. (a) The number of iterations; (b) The total computing time and speedup; Here, the colors red, green, and blue respectively represent the one-level, two-level, and three-level preconditioners.}
    \label{fig: scaling}
\end{figure}

Table~\ref{tab: scalability-Stokes-Flow-3D} provides a detailed comparison of computing time for each simulation as shown in Figure~\ref{fig: scaling}(b). Notably, the coarsening algorithm scales efficiently from 128 to 1,024 processors, achieving a parallel efficiency of $51.29\%$ for the first coarsening and $55.26\%$ for the second coarsening. Due to the high parallelism of the coarsening algorithm, the coarsening time consistently remains below $1\%$ of the simulation time as the number of processors increases from 128 to 1,024.
\begin{table}[H]
    \centering
    \caption{The computing time (``Time(s)'') and parallel efficiency for each simulation. Here, $np$ = 128, 256, 512, and 1,024 processors are utilized. }
    \setlength{\tabcolsep}{4.7mm}{
    \center
    \begin{tabular} {|c|c|c|c|c|}
        \hline
        $np$ &128 &256 &512 &1,024 \\\cline{1-5}
        \multicolumn{5}{l}{\textbf{Coarsening:} }\\ \hline
        $1_{st}$ coarsening &6.713($100\%$) &3.638($92.26\%$) &2.493($67.32\%$) &1.636($51.29\%$)\\ \hline
        $2_{nd}$ coarsening &5.128($100\%$) &2.874($89.21\%$) &1.703($75.28\%$) &1.160($55.26\%$)\\ \hline
        \multicolumn{5}{l}{\textbf{Solving:} }\\ \hline
		One-level &2153.05($100\%$) &1222.08($88.09\%$) &690.47($77.96\%$) &417.74($64.43\%$) \\\hline
        Two-level &1191.50($100\%$) &633.07($94.10\%$) &363.58($81.93\%$) &220.86($67.44\%$) \\\hline
		Three-level &1244.16($100\%$) &597.76($104.1\%$) &352.32($88.28\%$) &211.54($73.52\%$) \\\hline
    \end{tabular}
    }
    \label{tab: scalability-Stokes-Flow-3D}
\end{table}

\section{Conclusions}
\label{sec: conclusions}
In this work, we systematically discussed the non-nested unstructured scenario of the parallel multilevel Schwarz preconditioner. The main contributions of this work include two components:
\begin{itemize}
    \item A new parallel mesh coarsening algorithm was proposed, which preserves geometric features while demonstrating excellent robustness.
    \item A new parallel interpolation operator for multiphysics problems on domain-decomposed non-nested meshes was designed. This parallel interpolation operator can be applied to various discrete stencils.
\end{itemize}
By integrating the above two algorithms into the V-cycle framework, a new multilevel smoothed Schwarz preconditioner was proposed. Several numerical simulations demonstrated that the proposed multilevel preconditioner exhibits excellent performance in large-scale unstructured mesh scenarios, achieving high convergence rates and effectively scaling up to one thousand processors. In the future, we plan to conduct further theoretical analysis of the proposed multilevel preconditioner.

\bibliographystyle{elsarticle-num-names} 
\bibliography{cas-refs}





\end{document}